\documentclass[reqno]{amsart}

\usepackage[utf8]{inputenc}

\usepackage[T1]{fontenc}
\usepackage{listings}
\usepackage{titlesec}
\usepackage{pict2e}
\usepackage{comment}
\usepackage{textcase}
\usepackage{tikz}
\usetikzlibrary{arrows.meta}

\usepackage{epsf,amsfonts,amsmath,amsthm}
\usepackage{booktabs,subcaption,amsfonts,dcolumn}
\newcolumntype{d}[1]{D..{#1}}

\usepackage{color}

\definecolor{codegreen}{rgb}{0,0.6,0}
\definecolor{codegray}{rgb}{0.5,0.5,0.5}
\definecolor{codepurple}{rgb}{0.58,0,0.82}
\definecolor{backcolour}{rgb}{0.95,0.95,0.92}

\usepackage{amsthm}

\newtheorem{theorem}{Theorem}[section]
\newtheorem{lemma}[theorem]{Lemma}
\newtheorem{defin}[theorem]{Definition}

\newtheorem{corollary}[theorem]{Corollary}
\newtheorem{remark}[theorem]{Remark}
\newtheorem{assumption}[theorem]{Assumption}
\newtheorem{example}[theorem]{Example}

\usepackage{graphicx}
\usepackage{amsmath,amsfonts,amsthm}
\usepackage{amsmath,stackengine}
\usepackage{algorithm}
\usepackage{amssymb}

\usepackage[noend]{algpseudocode}
\usepackage{accents}
\usepackage{stmaryrd} 

\titleformat{\subsection}[runin]
  {\normalfont\large\bfseries}{\thesubsection}{1em}{}

\makeatletter
\renewcommand{\@seccntformat}[1]{\csname the#1\endcsname. }
\makeatother

\usepackage{calrsfs}
\DeclareMathAlphabet{\pazocal}{OMS}{zplm}{m}{n}

\newcommand{\Ca}{\pazocal{C}}
\newcommand{\Da}{\pazocal{D}}
\newcommand{\Ea}{\pazocal{E}}
\newcommand{\Fa}{\pazocal{F}}
\newcommand{\Gg}{\pazocal{G}}

\newcommand{\Ia}{\pazocal{I}}
\newcommand{\Ja}{\pazocal{J}}
\newcommand{\Ka}{\pazocal{K}}

\newcommand{\Pa}{\pazocal{P}}
\newcommand{\Qa}{\pazocal{Q}}
\newcommand{\Ra}{\pazocal{R}}

\newcommand{\Ta}{\pazocal{T}}

\newcommand{\hh}{\mathfrak{h}}
\newcommand{\Kk}{\mathtt{K}}

\renewcommand{\vec}[1]{\mathbf{#1}}
\renewcommand{\d}{\,\mathrm{d}}
\newcommand{\inv}{\mathrm{inv}}
\newcommand{\INV}{\mathrm{INV}}

\numberwithin{equation}{section}

\newcommand*{\lbrac}{\{\mskip-6mu\{}
\newcommand*{\rbrac}{\}\mskip-6mu\}}
\newcommand*{\trn}{|\mskip-1mu|\mskip-1mu|}

\newcommand{\vx}{{\tiny\textbullet } }

\newcommand{\mycomment}[1]{}

\usepackage{amsthm}

\numberwithin{equation}{section}

\title[Robust $\lowercase{hp}$-IPDG methods for the $\lowercase{p}$-Laplacian]{$\lowercase{hp}$-Version robust interior penalty discontinuous Galerkin methods for the $\lowercase{p}$-Laplacian on simplicial and on essentially arbitrarily-shaped element meshes}

\author[E. H. Georgoulis]{Emmanuil H. Georgoulis}
\address{The Maxwell Institute for Mathematical Sciences \& Department of Mathematics, Heriot-Watt University, Edinburgh EH14 4AS, United Kingdom and Department of Mathematics, School of Mathematical and Physical Sciences, National Technical University of Athens, Zografou 15780, Greece and IACM, FORTH, 70013 Heraklion, Crete, Greece.}
\email{E.Georgoulis@hw.ac.uk}

\author[P. Paraschis]{Panagiotis Paraschis}
\address{Institute of Analysis and Scientific Computing, Vienna University of Technology,
Wiedner Hauptstra{\upshape{\ss}}e 8-10, A-1040 Vienna, Austria and Faculty of Mathematics, University of Vienna, Oscar-Morgestern-Platz 1, A-1090
Vienna, Austria.}
\email{panagiotis.paraschis@asc.tuwien.ac.at}
\date{}

\begin{document}

\begin{abstract}
We consider the discretization of the $p$-Laplacian equation with an interior penalty discontinuous Galerkin method. We prove novel trace-type inverse estimates, leading to unconditional stability of the method. Further, $hp$-version \emph{a priori}  norm and quasi-norm error estimates are established, subordinate to available polynomial approximation results. The analysis is extended to discontinuous Galerkin methods, based on meshes with essentially arbitrarily-shaped, curved polygonal/polyhedral elements. This extension requires the proof of new $hp$-version weighted inverse estimates on essentially arbitrarily-shaped elements. Numerical experiments are also presented, highlighting the relevance of the theoretical findings.
\end{abstract}

\maketitle

\section{Introduction}\label{DG_intro}
Equations involving the $p$-Laplacian serve as model problems in a plethora of nonlinear diffusion settings, stemming from mathematical biology, chemical kinetics or non-Newtonian fluid mechanics \cite{diensul2013, ko2018, malek2005, philip1961, ruzicka2000}. The nonlinear nature of the $p$-Laplacian-type diffusion results into a number of numerical challenges, such as the construction of unconditionally stable methods, robustness with respect to the non-linearity index, and solvers, to mention a few.

To fix ideas, we consider a boundary value problem:
\begin{equation}\label{BVP}
    \begin{split}
        -\operatorname{div}\left(|\nabla u|^{p-2}\nabla u\right) &= f \text{ in }\Omega, \\
        u &= 0 \text{ on }\partial\Omega,
    \end{split}
\end{equation}
 for $p\in(2,\infty)$, $p'=p/(p-1)$, whereby $f\in L^{p'}(\Omega)$ denotes the forcing term and $\Omega\subset\mathbb{R}^d$ $(d=1,2,3)$ is a  bounded domain. For the first part of the present work, we shall assume that the domain is polytopic (that is, polygonal for $d=2$ or polyhedral for $d=3$), allowing for $\Omega$ to be partitioned exactly by simplicial meshes. In Section \ref{sec_DGpoly}, this assumption will be relaxed to piecewise curved geometries, that can be partitioned by general meshes, consisting of curved polytopic elements with essentially arbitrary shape. In this work, we shall consider the case $p\in(2,\infty)$ only. This is done for technical reasons related to the availability of suitable trace-inverse estimates.  The results presented below can be immediately modified for non-homogeneous Dirichlet, or mixed Dirichlet/Neumann boundary conditions in standard fashion. Nevertheless, we shall only consider the homogeneous Dirichlet case to avoid further notational overhead.

The development and analysis of numerical methods for the elliptic $p$-Laplacian has been extensively studied in the literature. Considerable efforts have focused on the study of conforming finite element approximations of \eqref{BVP}, dating back to the pioneering work of Glowinski and Marroco \cite{GloMar}, where suboptimal error estimates with respect to the $W^{1,p}(\Omega)$-norm are established. These bounds were later improved by Chow \cite{Chow1989}, and subsequently sharpened by Barrett and Liu \cite{Barrett1993}, who also established optimal $W^{1,q}(\Omega)$-error estimates, for $q\in(1,2)$. The above works consider piecewise linear finite element methods for the discretization of \eqref{BVP}, due to the limited regularity of the solution. For high-order approximations, we refer to the works of Ainsworth and Kay \cite{AinsKay2000, AinsKay1999}, considering conforming $p$- and $hp$-version finite element methods, where $W^{1,p}(\Omega)$-error estimates are derived\footnote{When referring to $p$- and $hp$-version finite element methods, we mean the classical setting whereby convergence is achieved by increasing the local polynomial degree and/or reducing the local mesh-size; this is not to be confused with the $p$-Laplacian, which refers to the power in the nonlinear diffusion operator. To avoid further confusion, in this work, local polynomial degrees will be denoted by symbols involving the letter ``$r$''.}.  \emph{A posteriori} error analysis, for a nonconforming Crouzeix--Raviart-type method have been shown in \cite{lynm2001}, and for conforming finite elements in \cite{ LiuYan2001,CarLiuYan2006}. We mention the works of Andreianov, Boyer and Hubert \cite{Andri1, Andri2}, employing finite volume methods for \eqref{BVP}.

In the past two decades, discontinuous Galerkin (DG) methods for the $p$-Laplacian have gained attention in the literature, starting with the pioneering work of Burman and Ern \cite{BurErn2008}, whereby a local discontinuous Galerkin (LDG) method was developed employing  jump-lifting operators, stemming from the discretization of the underlying energy minimization problem; convergence is proved via compactness arguments. The results have been later extended by Buffa and Ortner \cite{Buffa2009} and Diening et al. \cite{Diening2014}. The latter contains quasi-norm error estimates and in fact, represents a broader class of nonlinear elliptic problems with $p$-structure, including the $p$-Laplacian. The use of  jump-lifting operators was earlier pioneered by Ten Eyck and Lew for DG methods for nonlinear elasticity problems \cite{EyckLew2006}. Jump-lifting operators lead to larger stencils than standard interior-penalty DG (IPDG) methods, due to local solution of additional linear systems with the local mass matrix on each element. The methods in \cite{BurErn2008, Diening2014} introduce consistency errors, resulting to suboptimality of the corresponding DG schemes for $p\neq 2$. A quasi-optimal discontinuous Galerkin method has been developed by Blechta et al. \cite{Blechta2023}, following the approach of Veeser and Zanotti \cite{VesZan}, whereby a smoothing operator is employed to provide conformity in the test function. This approach eliminates the skeleton integrals appearing in discontinuous Galerkin methods. For the discretization of \eqref{BVP} with mixed hybridizable discontinuous Galerkin (HDG) methods, we refer to \cite{CoShen2016}.

Further in Malkmus et al. \cite{Ruz2018}, quasi-norm \emph{a priori} error estimates are proven for two different discontinuous Galerkin methods for systems with $p$-structure: a jump-lifting discontinuous Galerkin (DG-lifting) method resembling the LDG schemes in \cite{BurErn2008, Diening2014,Blechta2023} , and a symmetric IPDG method. The numerical examples with piecewise linear discontinuous Galerking spaces that are presented in that paper indicate that indeed, the DG-lifting method --- as well as the LDG scheme --- has higher complexity than the IPDG method, whereby jump-liftings are not employed. 
On the other hand, the error bounds of the IPDG method in  \cite{Ruz2018} are of first order only for any polynomial degree; hence, they are optimal for piecewise linear elements, or for high-order elements with low-regularity solutions. However, these estimates do not improve for high-order elements under higher regularity assumptions. This is due to the \emph{absence} of an appropriate trace inverse estimate that is subordinate to the $p$-structure. The low global regularity has been used as justification to consider only low-order methods in the literature. Typically, however, although \eqref{BVP} may not admit solutions with high global regularity properties, these solutions may admit high local regularity in sub-domains of $\Omega$.

The development of IPDG methods is of significant interest as they typically admit ``compact'' stencils, (compared to DG methods using liftings,) thereby leading to minimal communication in the context of domain decomposition and, thus, facilitating the use of parallel computer architectures. Moreover, as we shall see below, IPDG methodologies are easy to be generalized to admit general element shapes, which may lead to considerable complexity reduction \cite{CanDongGeoHoustSB}.

In response to the above shortcomings in the literature, the first part of this work is devoted to the development and analysis of an $hp$-version \emph{robust} interior penalty discontinuous Galerkin method on simplicial meshes for \eqref{BVP} admitting \emph{arbitrary} order for regular enough solutions for exponents $p\in[2,\infty)$. The IPDG method below is related to the one in \cite{Ruz2018}, combined with the robust IPDG point of view for linear elliptic problems proposed in \cite{DongGeo2022}. This new robust IPDG is stable under extreme local mesh and polynomial degree variations and has favourable properties in the context of non-overlapping domain decomposition offering minimal ``zero-halo'' communication between elements, irrespectively of the local polynomial degree combinations used.

To that end, we begin by proving a new class of polynomial trace-inverse estimates of the form 
\begin{equation*}
	\int_F |v|^q \d s \leq 2^{q+1}C_{\inv,d}\frac{(2r^2+1)|F|}{|K|}\int_K |v|^q \d\vec{x} \quad\forall v\in \Pa_r(K);
\end{equation*}
here, $q\in(0,\infty)$, $K$ is a $d$-dimensional simplex, $F$ is one of its faces, $r\in\mathbb{N}_0$ and $C_{\inv,d}$ is a constant, depending only on the dimension $d$. Note that, quite surprisingly, the above estimate holds also for $q\in(0,1)$, whence $\|\cdot\|_{L^q(K)}$ is not a norm. This result may be of independent interest.  The above inequality is the key step to prove a new trace-inverse estimate for quasi-norms for $p\in\mathbb{Q}$, which, in turn, is used for the proof of stability of the method and determines the, so-called, discontinuity-penalization (a.k.a. penalty) parameter of the method. Nevertheless, the robust stability of the method is established for all $p\in[2,\infty)$.

Further, \emph{a priori} error analysis is carried out in a broken Sobolev quasi-norm, requiring the new non-trivial quasi-norm inverse estimates. For the extraction of convergence rates with respect to the mesh-size and the local polynomial degrees, $hp$-version best approximation estimates in non-Hilbertian Sobolev norms
are necessary. These are unavailable in the literature, to the best of our knowledge. To ensure generality for the results presented below, we prove the error bounds upon a generic assumption of $hp$-convergence of suitable polynomial approximations. Furthermore, we also provide a specific $hp$-version best approximation result (with reduced rates with respect to the polynomial degrees) through the application of Gagliardo--Nirenberg interpolation inequalities and standard $hp$-approximation results in Hilbertian Sobolev norms from \cite{BabuskaSuri1987, MunozSola1997}. We stress that $p$-suboptimality of the final error estimates is not a result of the analysis of the method itself, but it arises from the (non-)availability of polynomial approximation results.

The development of robust IPDG methods on general curved polytopic elements offers the possibility of exact approximation of complex geometries without proliferation of numerical degrees of freedom (see, e.g., \cite{CanDongGeoHoustSB} for some examples,) and also can form a basis for the development of domain decomposition frameworks. These aspects become particularly pertinent in the present context whereby each non-linear iteration requires multiple runs of linear solvers. To that end, the second part of the present work is devoted to the extension of the robust IPDG method for \eqref{BVP} to meshes comprising essentially arbitrarily-shaped elements. In particular, we are able to extend the framework in Cangiani, Dong and Georgoulis \cite{CanDongGeo2021} to the robust IPDG method for $p$-Laplacian; see also \cite{CanDongGeoHoustSB}  for a detailed exposition of DG methods on polytopic meshes as and the review article \cite{antonietti} for the respective developments until then. Again, a key step is the proof of new polynomial trace-inverse estimates which, together with $hp$-version polynomial approximation results on curved polytopic elements,  lead to the proof of stability and convergence of the underlying robust IPDG method. The results hold under extremely mild assumptions on the admissible element shapes: elements with Lipschitz boundaries and \emph{piecewise} star-shapedeness conditions are enforced; see Section \ref{sec_DGpoly} for precise definitions.

The remainder of this work is structured as follows. In Section \ref{section_pre}, we set the notation and the finite element spaces. In Section \ref{DG_method}, we define the discontinuous Galerkin method. In Section \ref{DG_inverse}, we derive the key inverse estimates and polynomial approximation results. In Section \ref{DG_monot_cont}, we prove the monotonicity, boundedness and coercivity conditions of the underlying nonlinear form. In Section \ref{DG_staberr}, we prove the main stability and convergence results. In Section \ref{sec_DGpoly}, we generalize the results to curved polytopic meshes. Section \ref{DG_numerics} is devoted to the numerical experiments.

\section{Preliminaries}\label{section_pre}

Let $\omega\subset\mathbb{R}^d$, $d=1,2,3$ be a bounded domain, and let $\gamma\subset\mathbb{R}^d$ be a $(d-1)$-dimensional manifold for $d=2,3$. For all $q\in[1,\infty]$, $\ell\in[0,\infty)$ we denote by $L^q(\gamma)$, $L^q(\omega)$ and $W^{\ell,q}(\omega)$ the standard Lebesgue and Sobolev spaces, endowed with norms $\|\cdot\|_{L^q(\gamma)}$, $\|\cdot\|_{L^q(\omega)}$, $\|\cdot\|_{W^{\ell,q}(\omega)}$, respectively. We also consider the case $q\in(0,1)$ with a slight abuse of notation, since the latter are not norms in this case. The latter symbols will be used also to denote the norms of the spaces $L^q(\gamma;\mathbb{R}^d)$, $L^q(\omega;\mathbb{R}^d)$ and $W^{\ell,q}(\omega;\mathbb{R}^d)$ of vector-valued functions. In the special case $q=2$, we use the notation $H^\ell(\omega) = W^{\ell,2}(\omega)$, with the corresponding norm being denoted by $\|\cdot\|_{H^\ell(\omega)}$. Moreover, $W^{1,q}_0(\omega)$ and $H^1_0(\omega)$ refer to the spaces of functions in $W^{1,q}(\omega)$ and $H^1(\omega)$, respectively, having zero traces on the boundary $\partial\omega$.

\subsection{Notation and approximation spaces.}\label{DG_fes}
Aiming for a clearer exposition, we shall first consider approximation spaces defined over families of simplicial meshes and then, in Section \ref{sec_DGpoly}, extension to meshes comprising general curved polytopic element shapes. Much of the notation defined in this section extends trivially to the setting of Section \ref{sec_DGpoly} also.

First, we consider a family of simplicial meshes $\{\Ta\}$ of $\Omega$, that are assumed to be shape-regular in the sense that there exists a constant $c_\mathrm{sr}$, independent of $\Ta$, such that
\begin{equation*}
    h_K / \rho_K \leq c_\mathrm{sr} \text{ for all } K\in\Ta,
\end{equation*}
uniformly for all $\Ta$, where $\rho_K$ is the radius of the largest inscribed ball of $K$ and $h_K=\operatorname{diam}(K)$. Note that hanging nodes for $d=2$, or hanging edges for $d=3$, are allowed.  Each simplicial element $K\in\Ta$ has exactly $d+1$ $(d-1)$-dimensional faces which we denote by $e_1(K),\dots,e_{d+1}(K)$.

Let $\Gamma = \cup_{K\in\Ta}\partial K$ denote the mesh skeleton. We call $F\subset\Gamma$ an {\it interface} of $\Ta$, if $F=\partial K^+\cap \partial K^-$ whenever $F\subset\Gamma\setminus\partial\Omega$ for two adjacent elements $K^+,K^-\in\Ta$, or
$F=\partial K\cap \partial\Omega$ for an element $K\in\Ta$ adjacent to the boundary.

In other words, $F$ will be called an interface, if it is the total intersection of two neighboring element boundaries, or if it is a boundary face. We also define the set $\Fa := \{F\subset\Gamma: \text{ } F \text{ is an interface of } \Ta\}$. 

For $K\in\Ta$, if $F\in\Fa$ is an interface of $\Ta$ that lies on $\partial K$, then $F\subset e_j(K)$, for a unique $j\in\{1,\dots,d+1\}$, i.e., an interface $F\in\Fa$ is either an entire face, or part of a unique face. We define the face $F_K$ corresponding to an interface $F\in\Fa$ via
\begin{equation*}
    F_K := e_j(K) \text{ for the unique } j\in\{1,\dotso,d+1\} \text{ such that } F\subset e_j(K),
\end{equation*}
see Figure \ref{fig:faces} for a two-dimensional example. In particular, if an interface $F\in\Fa$ does not contain any hanging nodes/edges, then $F_K = F$, otherwise $F_K$ is a union of two or more co-planar interfaces. The $(d-1)$-dimensional measure $|F_K|$ of $F_K$ is an important mesh parameter appearing in the definition of the discontinuity-penalization parameter below. Note that \emph{no} assumption between the relatives sizes of $F_K$ and of $F$ is needed.

\begin{figure}[h!]
    \centering
    \begin{tikzpicture}
    \draw[thick] (2.25, 4.299038105) -- (0, 0.40192379);
    \draw[thick] (2.25, 4.299038105) -- (4.5, 0.40192379);
    \draw[thick] (0, 0.40192379) -- (4.5, 0.40192379);

    \draw[blue, thick] (0, 0.40192379) -- (1.5, 0.40192379);
    \filldraw[red] (1.5, 0.40192379) circle (2pt);

    \filldraw[black] (2.25, 4.299038105) circle (2pt);
    \filldraw[black] (0, 0.40192379) circle (2pt);
    \filldraw[black] (4.5, 0.40192379) circle (2pt);

    \node at (2.25,2) {$K$};
    \node at (2.6,-0.10192379) {$e_1(K)=F_K$};
    \node at (3.7,3) {$e_2(K)$};
    \node at (0.8,3) {$e_3(K)$};

    \node at (0.9, 0.80192379) {$\textcolor{blue}{F}$};
    \end{tikzpicture}
    \caption{An element $K$ with one hanging node (red dot), faces $e_1(K)$, $e_2(K)$, $e_3(K)$ and interface $F$ (blue line), and interfaces $F$, $e_2(K)$, $e_3(K)$ and $e_1(K)\setminus F$.}
    \label{fig:faces}
\end{figure}
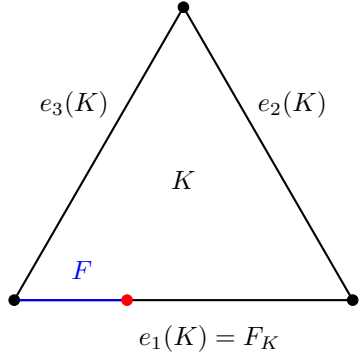
We denote the number of interfaces of an element $ K\in\Ta$ by
    $m_K := \operatorname{card}\{F\in\Fa: \text{ }F\subset\partial K\} $.
  If $K$ has no hanging nodes, then $m_K=d+1$.

For $\vec{k} = (k_K: \text{ }K\in\Ta)\in(\mathbb{R}_0^+)^{\operatorname{card}(\Ta)}$, we define the broken Sobolev spaces
\begin{align*}
    W^{\vec{k},p}(\Omega;\Ta) &:= \{v\in L^p(\Omega): \text{ } v|_K\in W^{k_K,p}(K), \text{ } K\in\Ta\}, \\
    W^{\vec{k},p}(\Omega;\Ta;\mathbb{R}^d) &:= \{\vec{v}\in L^p(\Omega;\mathbb{R}^d): \text{ } \vec{v}|_K\in W^{k_K,p}(K;\mathbb{R}^d), \text{ } K\in\Ta\}.
\end{align*}
For $k\geq0$, if $k_K = k$ for all $K\in\Ta$, then we set $W^{k,p}(\Omega;\Ta):=W^{\vec{k},p}(\Omega;\Ta)$.

Also, for $v\in W^{1,p}(\Omega;\Ta)$, $\nabla_h v$ the denotes the broken gradient of $v$, defined element-wise by $(\nabla_h v)|_K = \nabla(v|_K)$, for all $K\in\Ta$. Thus, if $v\in W^{1,p}(\Omega)$, then $\nabla_h v = \nabla v$.

To $\Ta$, we associate the vector $\vec{r} = (r_K: \text{ }K\in\Ta)\in\mathbb{N}^{\operatorname{card}(\Ta)}$, consisting of the local elemental polynomial degrees. Let also $\Pa_r(K)$ and $\Pa_r(K;\mathbb{R}^d)$ denote the spaces of scalar- and vector-valued polynomials of total degree at most $r$ in $K$, $r\in\mathbb{N}_0$, respectively. We also define the discrete space on which we shall be seeking the approximate solution by
\begin{equation*}
    S^{\vec{r}}_\Ta := \{v\in L^p(\Omega): \text{ } v|_K\in \Pa_{r_K}(K), \text{ } K\in\Ta\}, 
\end{equation*}
observing that $S^{\vec{r}}_\Ta\subset W^{1,p}(\Omega;\Ta)$.

Now, let $F\in\Fa$ be an interior interface of $\Ta$ shared by two elements $K^+$ and $K^-$. We denote by $v^\pm$ and $\vec{v}^\pm$, the traces of $v$ and $\vec{v}$ on $\partial K^\pm$, for $v$ and $\vec{v}$ element-wise continuous scalar- and vector-valued functions, respectively. We also consider an interface-wise constant vector field $\vec{w}:\Gamma\to[0,1]^2$, with $\vec{w}|_F = (w_F^+,w_F^-)$ and $w_F^+ + w_F^- = 1$. Furthermore, we denote by $\vec{n}^\pm$ the outward unit normal of $\partial K^\pm$. If $F$ is a boundary face, we set $K^+\in\Ta$ to be the boundary element that has the face $F$, i.e., $F=K^+\cap\partial\Omega$ and set $K^-=K^+$, $\vec{w}|_F=(1,0)$. We also denote by $v^+$ and $\vec{v}^+$, the traces of $v$ and $\vec{v}$ respectively on $\partial K^+$, by $\vec{n}^+$ the outward unit normal of $\partial K^+$, $\vec{n}^-=-\vec{n}^+$ and we set $v^- = 0$, $\vec{v}^- = \vec{0}$. Then, we define the jumps and weighted averages of $v$ and $\vec{v}$ on $F$, by
\begin{align*}
    \llbracket v \rrbracket|_F =\ & v^+\vec{n}^+ + v^-\vec{n}^-, \text{ } \llbracket \vec{v} \rrbracket|_F =  \vec{v}^+\cdot\vec{n}^+ + \vec{v}^-\cdot\vec{n}^-, \\
    \lbrac v \rbrac_{\vec{w}}|_F =\ & w_F^+v^+ + w_F^-v^-, \text{ } \lbrac \vec{v} \rbrac_{\vec{w}}|_F =  w_F^+\vec{v}^+ + w_F^-\vec{v}^-.
\end{align*}
Furthermore, for $F\in\Fa$, we denote by $\Ta_F = \{K\in\Ta: \text{ } F\subset\partial K\}$ and we observe that $\operatorname{card}\Ta_F=2$ if $F$ is an interior interface and $\operatorname{card}\Ta_F=1$, otherwise.

\subsection{Algebraic lemmas.}\label{prelim_algebra}

We collect some useful algebraic lemmas used below. 
\begin{lemma}[\cite{Barrett1993}, Lemma 2.1, \cite{diening2007}, Lemma 2.1]\label{th_delDG}
    For all $p\in(1,\infty)$, there exist positive constants $C_{1,p}$ and $C_{2,p}$, depending only on $p$, such that for all $\vec{y},\vec{z}\in\mathbb{R}^d$ and $a\geq0$,
\begin{align}
\left| (a^2+|\vec{y}|^2)^{\frac{p-2}{2}}\vec{y} - (a^2+|\vec{z}|^2)^{\frac{p-2}{2}}\vec{z} \right| &\leq C_{1,p}(a + |\vec{y}| + |\vec{z}|)^{p-2}|\vec{y} - \vec{z}|, \label{del1DG}\\
\left( (a^2+|\vec{y}|^2)^{\frac{p-2}{2}}\vec{y} - (a^2+|\vec{z}|^2)^{\frac{p-2}{2}}\vec{z} \right)\cdot\left( \vec{y} - \vec{z} \right) &\geq C_{2,p}(a + |\vec{y}| + |\vec{z}|)^{p-2}|\vec{y} - \vec{z}|^2. \label{del2DG}
\end{align}
\end{lemma}

Furthermore, we collect the following generalized Young inequalities.
\begin{lemma}\label{th_young}
Let $\varepsilon>0$, $\alpha,\beta_1,\beta_2\geq 0,$ and $p,p' \in (1, \infty),$ with $1/p + 1/p'=1.$ Then, with $\mu_\varepsilon = \max\{1,(4\varepsilon)^{-1}\},$ we have,
 \begin{align} & \beta_1 \beta_2 \leq \frac{\varepsilon^p}{p} \beta^p_1 + \frac{1}{p' \varepsilon^{p'}} \beta^{p'}_2 \text{ for all $p \in (1, \infty)$},
 \label{youngoriginal} \\
 &  (\alpha+\beta_1)^{p-2}\beta_1\beta_2 \leq \mu_\varepsilon(\alpha+\beta_1)^{p-2}\beta_1^2 + \varepsilon(\alpha+\beta_2)^{p-2}\beta_2^2 \text{ for all $p \in [2, \infty)$.} \label{young}
\end{align}
Moreover, for $\gamma>0$, we have
\begin{equation}\label{genYoung2}
    (\alpha+\beta_1)^{p-2}\beta_1\beta_2 \leq \gamma^{-1}\mu_\varepsilon(\alpha+\beta_1)^{p-2}\beta_1^2 + \gamma\varepsilon(\alpha+\gamma\beta_2)^{p-2}\beta_2^2.
\end{equation}
\end{lemma}
\begin{proof}  
The proof of \eqref{youngoriginal} is standard, while the proof of \eqref{young} can be found in \cite{ChPa2024} for $p\in(2,\infty)$. For the proof of \eqref{genYoung2}, we apply \eqref{young} with $\gamma\beta_2$ in place of $\beta_2$. 
\end{proof}
\begin{lemma}[\cite{lynm2001}, Lemma 2.4]\label{le_tri_vec}
For  $p\in(1,\infty)$, $\vec{y},\vec{z}\in\mathbb{R}^d$ and $\lambda\geq0$, it holds,
\begin{equation}\label{tri_vec}
    (\lambda + |\vec{y} + \vec{z}|)^{p-2}|\vec{y} + \vec{z}|^2 \leq \max\{2,2^{p-1}\}\left((\lambda + |\vec{y}|)^{p-2}|\vec{y}|^2+ (\lambda + |\vec{z}|)^{p-2}|\vec{z}|^2\right).
\end{equation}
\end{lemma}

\section{Robust interior penalty discontinuous Galerkin method}\label{DG_method}
We now are in position to define the robust interior penalty discontinuous Galerkin (IPDG) method on simplicial meshes; its extension to curved, polytopic element meshes will be given in Section \ref{sec_DGpoly}. 

To that end, we begin by defining the nonlinear form $B_\Ta:[W^{1+1/p,p}(\Omega;\Ta)]^3\to\mathbb{R}$, by
\begin{align} \label{BF}
    &B_\Ta(z;u,v) = \int_\Omega |\nabla_h u|^{p-2}\nabla_h u\cdot\nabla_h v \d \vec{x} + \int_\Gamma \sigma\lbrac \left(|\nabla z|^2 + \sigma^2|\llbracket u \rrbracket|^2\right)^{\frac{p-2}{2}}\rbrac_\vec{w} \llbracket u \rrbracket \cdot \llbracket v \rrbracket \d s \notag \\ 
    &\qquad - \int_\Gamma \lbrac |\nabla u|^{p-2}\nabla u\rbrac_\vec{w}\cdot \llbracket v \rrbracket \d s + \theta \int_{\Gamma} \lbrac \left(|\nabla z|^2 + \sigma^2|\llbracket u \rrbracket|^2\right)^{\frac{p-2}{2}}\nabla v\rbrac_\vec{w}\cdot \llbracket u \rrbracket \d s,
\end{align}
for all $z,u,v\in W^{1+1/p,p}(\Omega;\Ta)$, with $\theta\in[-1,1]$, $\sigma:\Gamma\to\mathbb{R}^+$ is the discontinuity-penalization (penalty) parameter, which is assumed to be interface-wise constant, that is, $\sigma|_F = \sigma_F\in\mathbb{R}^+$, for $F\in\Fa$. The values of $\sigma_F$ will be specified in the analysis of the method, as will be the weights $\vec{w}$ of the averages that appear in the integrals on $\Gamma$.
Then, the robust IPDG method reads: find $u_h\in S_\Ta^\vec{r}$, such that 
\begin{equation}\label{IPDG}
   B_\Ta(u_h;u_h,v_h) = \int_\Omega fv_h \d \vec{x} \quad\forall v_h\in S_\Ta^\vec{r}.
\end{equation}
The choices $\theta=-1$, $\theta=0$ and $\theta=1$ correspond to the symmetric, incomplete and non-symmetric IPDG methods, respectively; the terminology stems from the linear elliptic equation with $p=2$.  

The above method is a variant of the IPDG method proposed and analyzed in \cite{Ruz2018}, with some important departures. In particular, the non-differentiable terms $(|\nabla u_h| + h^{-1}|\llbracket u_h \rrbracket|)^{p-2}\llbracket u_h \rrbracket$ that appear in some of the skeleton integrals on the respective scheme in \cite{Ruz2018}, are replaced by $(|\nabla u_h|^2 + \sigma^2|\llbracket u_h \rrbracket|^2)^{\frac{p-2}{2}}\llbracket u_h \rrbracket$ here. The reason for this is two-fold: first, it renders the implementation with Newton-type non-linear solvers less sensitive to degeneracies and, secondly, it appears to be ``dimensionally'' correct with respect to the mesh parameters $h$ and $\vec{r}$, as it will transpire in the analysis below. Furthermore, \eqref{IPDG} involves weighted averages in the integrals on the skeleton $\Gamma$, in the spirit of the respective robust IPDG method proposed in \cite{DongGeo2022} for a linear elliptic problem. The choice of weighted averages prevents potential over-penalization, when two elements sharing a face have significantly different local polynomial degrees and/or diameters, cf. \cite{DongGeo2022} for details. The respective terms in the IPDG method in \cite{Ruz2018} correspond to $\mathbf{w}=(1/2,1/2)$ globally on the interior skeleton. Finally, here we consider the $hp$-version framework, thereby, we will be tracking the dependence on the local polynomial degree in $\mathbf{w}$ and in $\sigma$.

One can easily check that \eqref{IPDG} is consistent, in the sense that if the solution $u$ of \eqref{BVP} lies in $W^{1,p}_0(\Omega)\cap W^{1+1/p,p}(\Omega;\Ta)$, it holds
\begin{equation}\label{consistencyDG}
    B_\Ta(u_h;u,v_h) = \int_\Omega fv_h \d \vec{x} \quad\forall v_h\in S_\Ta^\vec{r}.
\end{equation}

We will carry out the error analysis of \eqref{IPDG}, with respect to the quasi-norm 
\begin{equation}\label{DGQN}
    \trn u \trn_{(\Ta;p,v)} := \bigg(\int_\Omega (|\nabla_h v| + |\nabla_h u|)^{p-2}|\nabla_h u|^2\d \vec{x} + \int_\Gamma\sigma\lbrac (|\nabla v| + \sigma|\llbracket u \rrbracket|)^{p-2}\rbrac_\vec{w}|\llbracket u \rrbracket|^2 \d s\bigg)^{\frac{1}{2}},
\end{equation}
for all $u\in W^{1,p}(\Omega;\Ta)$ and $v\in W^{1+1/p,p}(\Omega;\Ta)$, where $v$ will be referred to as the weight of the quasi-norm. This is, indeed, a quasi-norm as it satisfies a weaker version of the triangle inequality with a multiplicative constant. Also, for $q\in[1,\infty)$, we define the broken Sobolev norm $\trn\cdot\trn_q$ by
\begin{equation}\label{dgnorm}
    \trn u \trn_q = \left(\int_\Omega |\nabla_h u|^q\d \vec{x} + \int_\Gamma\sigma^{q-1}|\llbracket u \rrbracket|^q \d s\right)^{1/q} \text{ for all } u\in W^{1,q}(\Omega;\Ta),
\end{equation}
and we observe that for $q=2$, we obtain the standard broken norm of $H^1(\Omega;\Ta)=W^{1,2}(\Omega;\Ta)$.

\section{Inverse estimates and polynomial approximation}\label{DG_inverse}
\subsection{Inverse estimates.}\label{sub_DGinv}
We prove some new polynomial inverse estimates, that will be crucial for the definition of the so-called discontinuity-penalization parameter and for the stability and error analysis of the robust IPDG method.  

We begin by recalling the classical Markov's inequality stating
that 
\begin{equation}\label{Markov}
        \|\hat v'\|_{L^\infty(0,1)} \leq 2 r^2\|\hat v\|_{L^\infty(0,1)} \quad\forall \hat v\in\Pa_r(0,1),
    \end{equation}
    for all $r\in\mathbb{N}_0$; see, e.g., 
 \cite{BeDaMa07,Schwab} for a proof. This will be used in the proof of following technical lemma.
\begin{lemma}\label{th_infinv}
    For all $r\in\mathbb{N}_0$, $\hat v\in\Pa_r(0,1)$, there exists an interval $\hat I = \hat I(r,\hat v)\subset[0,1]$, such that
    \begin{equation}\label{infinv}
        |\hat I| \geq \frac{1}{2(1 + 2 r^2)} \text{ and } \inf_{\hat x \in \hat I}|\hat v(\hat x)| \geq \frac{1}{2}\|\hat v\|_{L^\infty(0,1)}.
    \end{equation}
\end{lemma}
\begin{proof}
    Let $\hat x_0\in[0,1]$ be such that $|\hat v(\hat x_0)| = \|\hat v\|_{L^\infty(0,1)}$. For $\delta\in(0,1/2)$, we define the interval $\hat I_\delta = [\hat x_0 - \delta, \hat x_0 + \delta]\cap[0,1]$. We will show that $|\hat I_\delta|\geq \delta$. Indeed, if $\hat x_0\leq \delta$, then $\hat x_0 + \delta \leq 2\delta\leq 1$. In this case, $\hat I_\delta = [0,\hat x_0 + \delta]$ and $|\hat I_\delta| = \hat x_0 + \delta \geq \delta$. On the other hand, if $\hat x_0 > \delta$, we have two cases: if $\hat x_0 + \delta < 1$, then $\hat I_\delta = [\hat x_0 - \delta, \hat x_0 + \delta]$, and hence $|\hat I_\delta|=2\delta\geq\delta$; if $\hat x_0+\delta\geq1$, then $\hat I_\delta=[\hat x_0-\delta,1]$ and $|\hat I_\delta|=1-\hat x_0+\delta\geq \delta$.

    Now, for all $\hat x\in\hat I_\delta$, it holds
    \begin{align}\label{beforedel}
        |\hat v(\hat x_0)| - |\hat v(\hat x)| &\leq |\hat v(\hat x_0) - \hat v(\hat x)| \leq \|\hat v'\|_{L^\infty(0,1)}|\hat x_0 - \hat x| \leq\delta\|\hat v'\|_{L^\infty(0,1)} \notag\\
        &\leq 2\delta r^2\|\hat v\|_{L^\infty(0,1)} \leq \delta(1 + 2 r^2)|\hat v(\hat x_0)|,
    \end{align}
    where at the fourth step, we applied \eqref{Markov}. We select $\delta=1/(2(1+2 r^2))\leq1/2$ and $\hat I = \hat I_\delta$. Then, we have $|\hat I|\geq\delta=1/(2(1+2 r^2))$ and from \eqref{beforedel}, it holds
    \begin{equation*}
        \frac{1}{2}\|\hat v\|_{L^\infty(0,1)} = \left(1 - \delta(1 + 2 r^2)\right)|\hat v(\hat x_0)| \leq |\hat v(\hat x)| \quad\forall \hat x\in\hat I;
    \end{equation*}
   the proof concludes upon considering the infimum over $\hat x\in \hat I$ on the right-hand side.
\end{proof}
 A $d$-dimensional, $h$-version variant of the above result has been proven in \cite{GrHaSa05}. We have provided complete proof above to showcase the dependence on the polynomial degree $r$. The above result will now be used to prove the following trace-inverse estimate. 
\begin{lemma}\label{th_norm_invest}
    Let $K$ be a $d$-dimensional simplex, $F\subset\partial K$ be one of its faces, $q\in(0,\infty)$ and $r\in\mathbb{N}_0$. Then, for $v\in\Pa_r(K)$, we have 
    \begin{equation}\label{LpFinv}
        \int_F |v|^q \d s \leq 2^{q+1}C_{\inv,d}\frac{(2 r^2 + 1)|F|}{d|K|}\int_K |v|^q \d \vec{x},
    \end{equation}
    whereby $C_{\inv,d}=1,4,8$, for $d=1,2,3$, respectively.
\end{lemma}
\begin{proof}
    Let us consider the one-dimensional case first. For $\hat v\in\Pa_r(0,1)$ and $\hat y\in[0,1]$, Lemma \ref{th_infinv} shows that there exists an interval $\hat I\subset[0,1]$, such that $|\hat I|\geq 1/(2(2 r^2 + 1))$ and
    \begin{equation*}
        |\hat v(\hat y)|^q \leq \|\hat v\|_{L^\infty(0,1)}^q \leq 2^q\inf_{\hat x\in \hat I}|\hat v(\hat x)|^q \leq \frac{2^q}{|\hat I|}\int_{\hat I}|\hat v|^q \d\hat x \leq 2^{q+1}(2 r^2 + 1)\int_0^1 |\hat v|^q \d\hat x.
    \end{equation*}
    Now, the above inequality together with a standard scaling argument shows that for $a<b\in\mathbb{R}$,
    \begin{equation}\label{1Dinv}
        |v(y)|^q \leq \|v\|_{L^\infty(a,b)}^q \leq \frac{2^{q+1}(2 r^2 + 1)}{b-a}\int_a^b|v|^q \d x \quad\forall v\in\Pa_r(a,b), \text{ } y\in[a,b],
    \end{equation}
    which is the desired estimate for $d=1$.

    For $d=2$, we consider the reference triangle $\hat K = \{(\hat x_1, \hat x_2)\in[0,1]^2: \text{ } \hat x_1 + \hat x_2 \leq 1\}$, and
 let $\hat F = \{0\}\times[0,1]$ be one of its faces. Then, for $\hat v \in \Pa_r(\hat K)$, it holds
    \begin{equation}\label{splitF}
        \int_{\hat F}|\hat v|^q\d \hat s = \int_0^1 |\hat v(0,\hat x_2)|^q \d \hat x_2 = \int_0^{1/2} |\hat v(0,\hat x_2)|^q \d \hat x_2 + \int_{1/2}^1 |\hat v(0,\hat x_2)|^q \d \hat x_2.
    \end{equation}
    For $\hat x_2\in[0,1/2]$, we define the polynomial $w_1^{\hat x_2}\in\Pa_r(0,1/2)$ by $w_1^{\hat x_2}(\hat x_1) = \hat v(\hat x_1,\hat x_2)$, for all $\hat x_1\in[0,1/2]$. In order to bound the first integral on the right-hand side of \eqref{splitF}, we apply \eqref{1Dinv} with $a=0$ and $b=1/2$, to obtain
    \begin{align}\label{F1}
        \int_0^{1/2}&|\hat v(0, \hat x_2)|^q\d \hat x_2 = \int_0^{1/2}|w_1^{\hat x_2}(0)|^q\d \hat x_2 \leq 2^{q+2}(2 r^2+1)\int_0^{1/2}\int_0^{1/2}|w_1^{\hat x_2}(\hat x_1)|^q \d \hat x_1 \d \hat x_2 \notag\\
        &= 2^{q+2}(2 r^2+1)\int_0^{1/2}\int_0^{1/2}|\hat v(\hat x_1, \hat x_2)|^q \d \hat x_1 \d \hat x_2 \leq 2^{q+2}(2 r^2+1)\int_{\hat K}|\hat v|^q \d \hat{\vec{x}},
    \end{align}
    since $[0,1/2]\times[0,1/2]\subset\hat K$. In order to estimate the second integral on the right-hand side of \eqref{splitF}, we consider the rectangle $\hat \Ra = [0,1/2]\times[1/2,1]$ and the parallelogram $\hat \Qa \subset \hat K$, with vertices $(0,1/2)$, $(0,1)$, $(1/2,1/2)$ and $(1/2,0)$. Then, we observe that $\hat \Qa = G(\hat\Ra)$, where $G:\hat\Ra\to\hat\Qa$ is the linear mapping, defined by $    G(\hat x_1, \hat x_2) = (\hat x_1, \hat x_2 - \hat  x_1) \quad \forall (\hat x_1, \hat x_2)\in \hat\Ra$. A standard change of variables shows that, for all $g\in L^1(\hat\Qa)$,
    \begin{equation*}
        \int_{1/2}^1\int_0^{1/2}g(\hat x_1, \hat x_2 - \hat x_1)\d \hat x_1\d \hat x_2 = \int_{\hat \Ra}g\circ G\d \hat{\vec{x}} = \int_{\hat\Qa}g \d \hat{\vec{x}},
    \end{equation*}
    upon observing that $\det JG= 1$. Now, for $\hat x_2\in[1/2,1]$, we consider the polynomial $w_2^{\hat x_2}\in\Pa_r(0,1/2)$ by $w_2^{\hat x_2}(\hat x_1) = \hat v(\hat x_1,\hat x_2 - \hat x_1)$, for all $\hat x_1\in[0,1/2]$. Another application of \eqref{1Dinv} shows
    \begin{align}\label{F2}
        \int_{1/2}^1|\hat v(0, \hat x_2)|^q\d \hat x_2 &= \int_{1/2}^1|w_2^{\hat x_2}(0)|^q\d \hat x_2 \leq 2^{q+2}(2 r^2+1)\int_{1/2}^1\int_0^{1/2}|w_2^{\hat x_2}(\hat x_1)|^q \d \hat x_1 \d \hat x_2 \notag\\
        &= 2^{q+2}(2 r^2+1)\int_{\hat \Qa}|\hat v|^q \d \hat{\vec{x}} \leq 2^{q+2}(2 r^2+1)\int_{\hat K}|\hat v|^q \d \hat{\vec{x}}.
    \end{align}
    Thus, a combination of \eqref{splitF}, \eqref{F1} and \eqref{F2} shows
    that $
        \int_{\hat F}|\hat v|^q \d \hat s \leq 2^{q+3}(2 r^2+1)\int_{\hat K}|\hat v|^q \d \hat{\vec{x}}$, while the final result follows by a standard scaling argument for any two-dimensional simplex $K$, any face $F\subset\partial K$ and any $v\in\Pa_r(K)$, upon selecting the respective affine mapping to map $F$ into $\hat F$. 
        
    The proof for the three-dimensional case is completely analogous, but we present a sketch, for the sake of completeness. We now consider our reference element $\hat K$ to be the tetrahedron with vertices $(0,0,0)$, $(1,0,0)$, $(0,1,0)$ and $(0,0,1)$, and let $\hat F = \{(\hat x_1, \hat x_2, 0)\in[0,1]^3: \text{ } \hat x_1 + \hat x_2 \leq 1\}$ be one of its faces. We partition $\hat F$ into four mutually exclusive triangles $\{\hat F_i\}_{i=1}^4$ with equal areas, such that the partition points are $(0,0,0)$, $(1,0,0)$, $(0,1,0)$, $(1/2,0,0)$, $(1/2,1/2,0)$ and $(0,1/2,0)$. Now, let $\hat K_1=\{(\hat x_1, \hat x_2)\in[0,1/2]^2: \text{ } \hat x_1 + \hat x_2 \leq 1/2\}$ and choose $\hat F_1 = \hat K_1 \times \{0\}$. For $(\hat x_1,\hat x_2)\in\hat K_1$, we consider the function $\hat w^{\hat x_1,\hat x_2}\in\Pa_r(0,1/2)$, defined by $\hat w^{\hat x_1,\hat x_2}(\hat x_3) = \hat v(\hat x_1,\hat x_2,\hat x_3)$, for $\hat x_3\in[0,1/2]$. Then, we have
    \begin{align*}
        \int_{\hat F_1}|\hat v|^q \d \hat s &= \int_{\hat K_1} |\hat v(\hat x_1,\hat x_2,0)|^q \d \hat x_1 \d \hat x_2 = \int_{\hat K_1} |\hat w^{\hat x_1,\hat x_2}(0)|^q \d \hat x_1 \d \hat x_2 \\
        &\leq 2^{q+2}(2r^2 + 1)\int_{\hat K_1}\int_0^{1/2} |\hat w^{\hat x_1,\hat x_2}(\hat x_3)|^q \d \hat x_1 \d \hat x_2 \d \hat x_3 \\
        &= 2^{q+2}(2r^2 + 1)\int_{\hat K_1}\int_0^{1/2} |\hat v(\hat x_1,\hat x_2,\hat x_3)|^q \d \hat x_1 \d \hat x_2 \d \hat x_3\leq 2^{q+2}(2r^2 + 1)\int_{\hat K}|\hat v|^q \d \hat{\vec{x}},
    \end{align*}
    where at the third step, we applied the one-dimensional inverse estimate \eqref{1Dinv}. Similarly, with the help of three-dimensional analogues to the mapping $G$ defined above, we can show that for $i=2,3,4$, that 
       $ \int_{\hat F_i}|\hat v|^q \d \hat s \leq 2^{q+2}(2r^2 + 1)\int_{\hat K}|\hat v|^q \d \hat{\vec{x}}$; the assertion follows from the summation of all the inverse estimates at each $\hat F_i$ and applying a scaling argument.
\end{proof}
\begin{remark}\label{rem_coolinvest}
    For $q=2$, the above inverse estimate reduces to the standard $hp$-version $L^2$-trace inverse estimate. However, for $q \neq 2$ the result appears to be novel, 
    to the best of the authors' knowledge. 
    Note that the result holds also for $q \in (0,1)$, whereby even the $h$-version is highly non-trivial, since $\|\cdot\|_{L^q(K)}$ is not a norm and thus, finite-dimensionality arguments do not apply. This result may be of independent interest. 

\end{remark}

The following quasi-norm-inverse estimate will be crucial for the analysis below. 
\begin{lemma}\label{th_QN_invest}
    Let $p\in(2,\infty)\cap\mathbb{Q}$ and let $k_p, \ell_p\in\mathbb{N}$, such that $p = 2 + 2k_p/\ell_p$. Let also $K$ be a $d$-dimensional simplex, $F\subset\partial K$ be one of its faces, and let $r\in\mathbb{N}$. Then, for $w,v\in\Pa_r(K)$, we have
    \begin{align}\label{QN_invest}
        \int_F&\left( |\nabla w| + |\nabla v| \right)^{p-2}|\nabla v|^2\d s \\
        &\qquad\leq 2^{p/2+1/k_p}C_{\inv,d}\frac{(2p^2k_p^2(r-1)^2+1)|F|}{d|K|}\int_K\left( |\nabla w| + |\nabla v| \right)^{p-2}|\nabla v|^2\d \vec{x}.
    \end{align}
    where $C_{\inv,d}$ is the constant in \eqref{LpFinv}.
\end{lemma}
\begin{proof}
    From the equivalence of norms in $\mathbb{R}^2$ and the definitions of $k_p$ and $\ell_p$, we have
    \begin{align*}
        \int_F\left( |\nabla w| + |\nabla v| \right)^{p-2}|\nabla v|^2\d s &\leq 2^{p/2 -1}\int_F\left( |\nabla w|^2 + |\nabla v|^2 \right)^{\frac{p-2}{2}}|\nabla v|^2\d s \\
        &= 2^{p/2 -1}\int_F\left|\left( |\nabla w|^2 + |\nabla v|^2 \right)^{\ell_p}|\nabla v|^{2k_p}\right|^{1/k_p}\d s.
    \end{align*}
    We observe that $\left( |\nabla w|^2 + |\nabla v|^2 \right)^{\ell_p}|\nabla v|^{2k_p}\in\Pa_{pk_p(r-1)}(K)$ --- since $pk_p(r-1)\in\mathbb{N}_0$ --- and we apply the inverse estimate \eqref{LpFinv} with $q=1/k_p$ in the above inequality, to obtain
    \begin{align}\label{invweight}
        \int_F&\left( |\nabla w| + |\nabla v| \right)^{p-2}|\nabla v|^2\d s \notag\\
        & \qquad \leq 2^{p/2 + 1/k_p}C_{\inv,d}\frac{(2p^2k_p^2(r-1)^2+1)|F|}{d|K|}\int_K\left|\left( |\nabla w|^2 + |\nabla v|^2 \right)^{\ell_p}|\nabla v|^{2k_p}\right|^{1/k_p}\d \vec{x} \notag\\ 
        & \qquad = 2^{p/2 + 1/k_p}C_{\inv,d}\frac{(2p^2k_p^2(r-1)^2+1)|F|}{d|K|}\int_K\left( |\nabla w|^2 + |\nabla v|^2 \right)^{\frac{p-2}{2}}|\nabla v|^2\d \vec{x} \notag\\ 
        &\qquad \leq 2^{p/2 + 1/k_p}C_{\inv,d}\frac{(2p^2k_p^2(r-1)^2+1)|F|}{d|K|}\int_K\left( |\nabla w| + |\nabla v| \right)^{p-2}|\nabla v|^2\d \vec{x},
    \end{align}
    and the proof is complete.
\end{proof}

\begin{remark}\label{rem_k}
Some comments are in order.
\begin{enumerate}
  \item The above result is an inverse estimate for non-polynomial functions, which may be quite surprising at first. It is expected that the arguments applied for the proof of \eqref{QN_invest} can produce more general non-polynomial inverse estimates that could be useful for IPDG methods for other degenerate nonlinear PDEs.
  \item A result of this type for linear $w$ (giving constant $\nabla w$) without tracking the dependence on the polynomial degree and on $p$ has been given in \cite[Lemma 3.6]{lynm2001}. The proof given therein using Jensen's inequality is not applicable in the present $hp$-version setting of general polynomial $w$, requiring the completely different-in-spirit proof given above.  
    \item The result holds for rational $p$ only. It is conceivable that, with a density or similar argument, validity for all $p$ can be recovered. We have not been able to carry this through, however, without introducing unknown constants in the estimate. Since this estimate will drive the definition of the discontinuity-penalization parameter below, we have opted for the most general result with a known constant.
    \item In practical computations, the choice of penalty parameter based on \eqref{QN_invest} need not be selected proportional to the square of the denominator $k_p$. Thus, it is not clear at this point if this dependence is an artifact of the method of proof. Nevertheless, this practical observation suggests that it may be sufficient to vary the choice of discontinuity-penalization parameter continuously with respect to $p$.
  \item   The analysis below can be immediately generalized for $p \notin \mathbb{Q}$, once the above inverse estimate is proven for irrational $p$.

\end{enumerate}
\end{remark}

\subsection{Polynomial approximation.}\label{DG_approx}
We consider $hp$-version polynomial approximation on simplices, to be used below for the proof of \emph{a priori} error bounds. The required estimates will be given as an assumption for generality. However, we will present an example, where the assumption holds true, based on known results from the literature. 
\begin{assumption}\label{assumpt_best_app}
    Let $K\in\Ta$, $p\in(2,\infty)$ and $v\in W^{s,p}(K)$, for some $2 \leq s \leq r_K+1$. We assume that there exists a polynomial $\Pi_{r_K} v \in \Pa_{r_K}(K)$ and there exist $0\leq\beta_0\leq\beta_1\leq\beta_2$, such that $\Pi_{r_K}v = v $, for all  $v\in\Pa_{r_K}(K)$, and
    \begin{align}
        \|v - \Pi_{r_K} v\|_{W^{m,p}(K)} &\leq C \frac{h_K^{s-m}}{r_K^{s-m-\beta_m}}\|v\|_{W^{s,p}(K)}, \text{ } 0 \leq m \leq 2 \leq s; \label{proj_elem}
    \end{align}
    here $C>0$ is a constant depending on $s,p$ and the shape-regularity constant of $\Ta$. Moreover, $\beta_0$, $\beta_1$ and $\beta_2$ are allowed to depend on $s,p$ and $d$, but are assumed to be bounded with respect to $s$ and to $p$.
\end{assumption}

Assumption \ref{assumpt_best_app} is a generalization to the standard optimal $hp$-polynomial approximation results (with $\beta_m=0$) in Hilbertian Sobolev norms, proven in \cite{BabuskaSuri1987, MunozSola1997}. Below, we give an example of exponents for which a proof is available. 
\begin{lemma}\label{th_hp_polyap}
    Assumption \ref{assumpt_best_app} is satisfied with $\beta_0=\beta_1=\beta_2=1$ for $d=1$, while for $d=2,3$, we have the cases
    \begin{equation*}
        \beta_0 = d/2 - d/p,
    \end{equation*}
    \begin{equation*}
        \beta_1 = \begin{cases}
            d/2 - d/p, & d=2, \text{ or } d = 3 \text{ \& } p\in(2,6], \text{ or } d=3 \text{ \& } p\in(6,\infty) \text{ \& } s \geq 3; \\
            d - 2d/p, & d=3 \text{ \& } p\in(6,\infty) \text{ \& } s = 2,
        \end{cases}
    \end{equation*}
    and
    \begin{equation*}
        \beta_2 = \begin{cases}
            d/2 - d/p, & d=2 \text{ \& } s \geq 3, \text{ or } d=3 \text{ \& } p\in(2,6] \text{ \& } s \geq 3, \text{ or } d=3 \text{ \& } p\in(6,\infty) \text{ \& } s \geq 4; \\
            d - 2d/p, & d = 2 \text{ \& } s = 2, \text{ or } d=3 \text{ \& } p\in(2, 6] \text{ \& } s = 2, \text{ or } d = 3 \text{ \& } p\in(6,\infty) \text{ \& } s = 2, 3. 
        \end{cases}
    \end{equation*}

\end{lemma}
\begin{proof}
    The exponents are proven in \cite{Melenk2005} for $d=1$. For $d=2,3$, they follow by applying Gagliardo--Nirenberg interpolation inequalities and the Hilbertian norm best approximation estimates from \cite{BabuskaSuri1987, MunozSola1997}. We omit the details in the interest of brevity.
\end{proof}

Below, we prefer to make use of the approximation estimates as presented in Assumption \ref{assumpt_best_app}, where the $\beta_m$'s are not explicitly defined. We choose this approach, so that, if sharper $hp$-polynomial approximation results are proven, then the convergence rates with respect to the local polynomial degrees of the robust IPDG method will automatically improve.
\section{Monotonicity, coercivity and boundedness}\label{DG_monot_cont}
We begin by showing a monotonicity result with respect to the quasi-norm.
\begin{lemma}\label{th_DGmonot}
    Let $\varepsilon>0$, $p\in(2,\infty)\cap\mathbb{Q}$ and let $k_p, \ell_p\in\mathbb{N}$, such that $p = 2 + 2k_p/\ell_p$. For $F\in\Fa$, we define
    \begin{equation}\label{GKF}
        \Gg_{K,F} := 2^{p/2+1/k_p}C_{\inv,d}p^2k_p^2\frac{(2(r-1)^2+1)|F_K|}{d|K|}, \text{ } K\in\{K^+,K^-\},
    \end{equation}
    For $\zeta_F^\pm := \varepsilon/(m_{K^\pm}\mu_\varepsilon \Gg_{K^\pm,F})$, we select
    \begin{equation}\label{sigma}
        w_F^\pm = \zeta_F^\pm/(\zeta_F^+ + \zeta_F^-) \quad\text { and }\quad  \sigma_F = 1/(\zeta_F^+ + \zeta_F^-).
    \end{equation} 
    Then, for $u_h,v_h\in S_\Ta^\vec{r}$ and $u\in W^{1,p}_0(\Omega)$, we have
    \begin{equation}\label{DGmonot}
    \begin{aligned}
    C_{{\rm mon},p}(\varepsilon)  \trn u_h - v_h \trn_{(\Ta;p,u_h)}^2 
      \leq&\  B_\Ta(u_h;u_h,u_h - v_h) - B_\Ta(u_h;v_h,u_h - v_h) \\&+ (2^{p-1}+1)|\theta|\varepsilon\trn u - v_h \trn_{(\Ta;p,u_h)}^2,
    \end{aligned}
    \end{equation}
    with $C_{{\rm mon},p}(\varepsilon):=  2^{2-p}C_{2,p} - (2^{p-2}C_{1,p}+(2^{p-1}+1)|\theta|)\varepsilon$, where $C_{1,p},C_{2,p}$ as in Lemma \ref{th_delDG} and $\mu_\varepsilon$ as in Lemma \ref{th_young}.
\end{lemma}
\begin{proof} Set $w_h=u_h-v_h$ for brevity. Then,
    \begin{align}\label{J1234}
        B_\Ta&(u_h;u_h,w_h) - B_\Ta(u_h;v_h,w_h) \notag\\
        &\quad= \int_\Omega \left( |\nabla_h u_h|^{p-2}\nabla_h u_h - |\nabla_h v_h|^{p-2}\nabla_h v_h \right)\cdot\nabla_h w_h\d \vec{x} \notag\\
        &\quad\quad + \int_\Gamma \sigma\lbrac\left(|\nabla u_h|^2 + \sigma^2|\llbracket u_h\rrbracket|^2\right)^{\frac{p-2}{2}}\llbracket u_h \rrbracket - \left(|\nabla u_h|^2 + \sigma^2|\llbracket v_h\rrbracket|^2\right)^{\frac{p-2}{2}}\llbracket v_h \rrbracket \rbrac_\vec{w} \cdot \llbracket w_h \rrbracket\d s \notag\\
        &\quad\quad - \int_\Gamma \lbrac |\nabla u_h|^{p-2}\nabla u_h - |\nabla v_h|^{p-2}\nabla v_h \rbrac_\vec{w}\cdot\llbracket w_h \rrbracket \d s \notag \\
        &\quad\quad + \theta\int_\Gamma \lbrac\Big(\left(|\nabla u_h|^2 + \sigma^2|\llbracket u_h\rrbracket|^2\right)^{\frac{p-2}{2}}\llbracket u_h \rrbracket - \left(|\nabla u_h|^2 + \sigma^2|\llbracket v_h\rrbracket|^2\right)^{\frac{p-2}{2}}\llbracket v_h\rrbracket \Big) \cdot \nabla w_h\rbrac_\vec{w} \d s \notag \\
        &\quad\equiv \Ja_1 + \Ja_2 - \Ja_3 + \theta \Ja_4.
    \end{align}
    For the estimation of $\Ja_1$, \eqref{del2DG} shows
    \begin{align}\label{J1}
        \Ja_1 
        &\geq 2^{2-p}C_{2,p}\int_\Omega (|\nabla_h u_h| + |\nabla_h w_h|)^{p-2}|\nabla_h w_h|^2\d \vec{x}.
    \end{align}
    For $\Ja_2$, we have, again from \eqref{del2DG},
    \begin{align}\label{J2}
        \Ja_2 &= \sum_{F\in\Fa}\sum_{\pm\in\{+,-\}}\!\! \! w_F^\pm\sigma_F\int_F \Big(\left(|\nabla u_h^\pm|^2 + \sigma_F^2|\llbracket u_h\rrbracket|^2\right)^{\frac{p-2}{2}}\llbracket u_h \rrbracket - \left(|\nabla u_h^\pm|^2 + \sigma_F^2|\llbracket v_h\rrbracket|^2\right)^{\frac{p-2}{2}}\llbracket v_h \rrbracket \Big) \cdot \llbracket w_h \rrbracket\d s\notag\\
        &\geq\ 2^{2-p}C_{2,p}\int_\Gamma \sigma \lbrac(|\nabla_h u_h| + \sigma|\llbracket w_h\rrbracket|)^{p-2}\rbrac_\vec{w}| \llbracket w_h \rrbracket |^2\d s.
    \end{align}
Thus, \eqref{J1} and \eqref{J2} show
\begin{equation}\label{J1J2}
    \Ja_1 + \Ja_2 \geq 2^{2-p}C_{2,p}\trn w_h\trn_{(\Ta;p,u_h)}^2.
\end{equation}
For $\Ja_3$, we apply \eqref{del1DG}, to show
\begin{align}\label{J3first}
    |\Ja_3| &\leq \sum_{F\in\Fa}\sum_{\pm\in\{+,-\}}w_F^\pm\int_F \left| |\nabla u_h^\pm|^{p-2}\nabla u_h^\pm - |\nabla v_h^\pm|^{p-2}\nabla v_h^\pm \right| |\llbracket w_h \rrbracket| \d s \notag\\
    &\leq 2^{p-2}C_{1,p}\mu_\varepsilon\sum_{F\in\Fa}\sum_{\pm\in\{+,-\}} w_F^\pm\sigma_F^{-1}\int_F \left( |\nabla u_h^\pm| + |\nabla w_h^\pm| \right)^{p-2}|\nabla w_h^\pm|^2\d s \notag\\
    &\qquad\qquad\qquad + 2^{p-2}C_{1,p}\varepsilon\int_\Gamma\sigma \lbrac\left( |\nabla u_h^\pm| + \sigma|\llbracket w_h \rrbracket| \right)^{p-2}\rbrac_\vec{w}|\llbracket w_h \rrbracket|^2 \d s,
\end{align}
with $w_h^\pm=u_h^\pm - v_h^\pm$, where in the last step, \eqref{genYoung2} with $\alpha = |\nabla u_h^\pm|$, $\beta_1 = |\nabla w_h^\pm |$, $\beta_2 = |\llbracket w_h \rrbracket|$ and $\gamma = \sigma_F$ is applied. To each $F\in\Fa$, with $F=\partial K^+\cap\partial K^-$,  we employ \eqref{QN_invest}, to show
\begin{align}\label{invest_complete}
    \int_F \big( |\nabla u_h^\pm| + |\nabla w_h^\pm | \big)^{p-2}|\nabla w_h^\pm|^2\d s
    \leq &\ \Gg_{K^\pm,F}\int_{K^\pm} \big( |\nabla u_h| + |\nabla w_h| \big)^{p-2}|\nabla w_h |^2\d \vec{x}.
\end{align}
Using \eqref{invest_complete} on \eqref{J3first} and substituting the values of $\sigma$ and of $\vec{w}$ we obtain
\begin{equation}\label{J3_even}
    |\Ja_3| \leq 2^{p-2}C_{1,p}\varepsilon\trn w_h  \trn_{(\Ta;p,u_h)}^2.
\end{equation}
It remains to bound $\Ja_4$. To that end, we observe that
\begin{align}\label{J4first}
    \Ja_4 = &\int_\Gamma \lbrac\left(|\nabla u_h|^2 + \sigma^2|\llbracket u_h - u\rrbracket|^2\right)^{(p-2)/2}\llbracket u_h - u \rrbracket \cdot \nabla w_h \rbrac_\vec{w} \d s \notag\\
    &\quad + \int_\Gamma \lbrac\left(|\nabla u_h|^2 + \sigma^2|\llbracket u - v_h\rrbracket|^2\right)^{(p-2)/2}\llbracket u - v_h \rrbracket \cdot \nabla w_h \rbrac_\vec{w} \d s = \Ja_{4,1} + \Ja_{4,2}.
\end{align}
For $\Ja_{4,1}$, we apply \eqref{young} to show
\begin{align*}
    |\Ja_{4,1}| &\leq \sum_{F\in\Fa}\sum_{\pm\in\{+,-\}}w_F^\pm \sigma_F^{-1} \int_F\left(|\nabla u_h^\pm|^2 + \sigma_F^2|\llbracket u_h - u\rrbracket|^2\right)^{\frac{p-2}{2}}\sigma_F\llbracket u_h - u \rrbracket \cdot \nabla w_h^\pm  \d s \\
    &\leq \mu_\varepsilon\sum_{F\in\Fa}\sum_{\pm\in\{+,-\}}w_F^\pm \sigma_F^{-1}\int_F\left(|\nabla u_h^\pm| + |\nabla w_h^\pm|\right)^{p-2}|\nabla w_h^\pm|^2  \d s \\
    &\qquad + \varepsilon\int_\Gamma \sigma\lbrac \left( |\nabla u_h| + \sigma|\llbracket u_h - u \rrbracket| \right)^{p-2}\rbrac_\vec{w}|\llbracket u_h - u \rrbracket|^2 \d s \\
    &\leq \mu_\varepsilon\sum_{F\in\Fa}\sum_{\pm\in\{+,-\}}w_F^\pm \sigma_F^{-1}\int_F\left(|\nabla u_h^\pm| + |\nabla w_h^\pm|\right)^{p-2}|\nabla w_h^\pm|^2  \d s \\
    &\qquad + 2^{p-1}\varepsilon\int_\Gamma \sigma\lbrac \left( |\nabla u_h| + \sigma|\llbracket w_h \rrbracket| \right)^{p-2}\rbrac_\vec{w}|\llbracket w_h \rrbracket|^2 \d s \\
    &\qquad + 2^{p-1}\varepsilon\int_\Gamma \sigma\lbrac \left( |\nabla u_h| + \sigma|\llbracket v_h - u \rrbracket| \right)^{p-2}\rbrac_\vec{w}|\llbracket v_h - u \rrbracket|^2 \d s,
\end{align*}
applying \eqref{tri_vec} in the last inequality. Treating the first term on the last step of the above inequality as in \eqref{invest_complete}, we infer
\[
    |\Ja_{4,1}| \leq 2^{p-1}\varepsilon \left(\trn u_h - v_h \trn_{(\Ta;p,u_h)}^2 + \trn u - v_h \trn_{(\Ta;p,u_h)}^2\right).
\]
Following the same steps, except for the application of \eqref{tri_vec}, we arrive to
\begin{equation}\label{J42}
    |\Ja_{4,2}| \leq \varepsilon \left(\trn w_h  \trn_{(\Ta;p,u_h)}^2 + \trn u - v_h \trn_{(\Ta;p,u_h)}^2 \right),
\end{equation}
and hence,
\begin{equation}\label{J4_even}
    |\Ja_4| \leq (2^{p-1} + 1)\varepsilon \left(\trn w_h \trn_{(\Ta;p,u_h)}^2 + \trn u - v_h \trn_{(\Ta;p,u_h)}^2 \right).
\end{equation}
We substitute \eqref{J1J2}, \eqref{J3_even} and \eqref{J4_even} into \eqref{J1234}, to complete the proof.
\end{proof}

\begin{remark}\label{rem_good_penalty}
    We observe that for $F\in\Fa$, we have $\sigma_F = 1/(\zeta_F^+ + \zeta_F^-) \leq 1/\zeta_F^*$, for  $*\in\{+,-\}$ and, thus,
    \begin{equation}\label{good_penalty}
        \sigma_F \leq \mu_\varepsilon \min_{K\in\Ta_F}\left(m_K\Gg_{K,F}\right)/\varepsilon.
    \end{equation}
    In standard IPDG with $w_F^\pm=1/2$, $\sigma_F$ is usually defined by the locally largest inverse estimate constant, to ensure coercivity. Inequality \eqref{good_penalty} shows that when using carefully constructed weights, the penalty parameter is bounded proportionally to the minimum of all the respective trace-inverse estimate constants. Thus, the weightes in the robust IPDG method result to small penalization, even in extreme scenarios, such as high local mesh-size or local polynomial degree variations. For a computational demonstration of this advantage for linear elliptic problems we refer to \cite{DongGeo2022}.
\end{remark}

We continue with a boundedness result.
\begin{lemma}\label{th_DGcont}
    Let $\varepsilon>0$, $p\in(2,\infty)\cap\mathbb{Q}$. We assume that $\sigma$ and $\vec{w}$ are defined by \eqref{sigma}. Then, for all $u_h,v_h\in S_\Ta^\vec{r}$, $u\in W^{1,p}_0(\Omega)\cap W^{1+1/p,p}(\Omega;\Ta)$, we have
    \begin{align}\label{DGcont}
        |B_\Ta(u_h;&u,u_h-v_h) - B_\Ta(u_h;v_h,u_h-v_h)| \notag \\
       \leq&\  \Tilde{C}_p\varepsilon\trn u_h - v_h\trn_{(\Ta;p,u_h)}^2 + (\mu_\varepsilon + \varepsilon)\trn u - v_h\trn_{(\Ta;p,u_h)}^2 \notag\\
        & + 2^{2p-4}C_{1,p}\mu_\varepsilon\bigg( \int_\Omega \big( |\nabla u| + |\nabla_h(u-v_h)| \big)^{p-2}|\nabla_h(u-v_h)|^2 \d \vec{x} \notag\\
        & + 2\int_\Gamma \sigma^{-1} \lbrac\left( |\nabla u_h| + |\nabla (u-v_h)| \right)^{p-2}|\nabla (u-v_h)|^2\rbrac_\vec{w} \d s\bigg),
    \end{align}
    where $\Tilde{C}_p = (2^{p-1}+2^{2p-4}+2^{2p-3})C_{1,p}+1$.
\end{lemma}
\begin{proof} We observe that $\llbracket u\rrbracket=0$, since $u\in W^{1,p}(\Omega)$. Thus,
    \begin{align}\label{J123cont}
        |B_\Ta(u_h;u,u_h-v_h) - &B_\Ta(u_h;v_h,u_h-v_h)|\notag\\ 
        \leq&\  \int_\Omega \left| |\nabla u|^{p-2}\nabla u - |\nabla_h v_h|^{p-2}\nabla_h v_h \right| |\nabla_h(u_h-v_h)|\d \vec{x} \notag\\
        & + \int_\Gamma\sigma\left|\lbrac\left(|\nabla u_h|^2 + \sigma^2|\llbracket v_h\rrbracket|^2\right)^{\frac{p-2}{2}}\llbracket v_h \rrbracket \rbrac_\vec{w}\right| |\llbracket u_h- v_h \rrbracket|\d s \notag\\
        & + \int_\Gamma \left|\lbrac |\nabla u|^{p-2}\nabla u - |\nabla v_h|^{p-2}\nabla v_h \rbrac_\vec{w}\right| |\llbracket u_h - v_h \rrbracket| \d s \notag\\
        & + |\theta|\int_\Gamma \left|\lbrac \left(|\nabla u_h|^2 + \sigma^2|\llbracket v_h\rrbracket|^2\right)^{\frac{p-2}{2}}\nabla (u_h - v_h) \rbrac_\vec{w}\right| |\llbracket v_h \rrbracket| \d s \notag\\
        &\mkern-36mu = \Ka_1 + \Ka_2 + \Ka_3 + |\theta|\Ka_4.
    \end{align}
    For $\Ka_1$, \eqref{del1DG} and \eqref{young} show
    \begin{align}\label{J1cont}
        \Ka_1 &\leq C_{1,p}\int_\Omega \left( |\nabla u| + |\nabla_h v_h|\right)^{p-2}|\nabla_h (u - v_h)| |\nabla_h(u_h-v_h)|\d \vec{x} \notag \\
        &\leq 2^{2p-4}C_{1,p}\bigg( \mu_\varepsilon\int_\Omega \left( |\nabla u| + |\nabla_h (u - v_h)|\right)^{p-2}|\nabla_h (u - v_h)|^2\d \vec{x} \notag\\
        &\qquad\qquad\qquad\qquad + \varepsilon \int_\Omega \left( |\nabla_h u_h| + |\nabla_h (u_h - v_h)|\right)^{p-2}|\nabla_h(u_h-v_h)|^2\d \vec{x} \bigg).
    \end{align}
    Next,  we apply \eqref{young}, to obtain
    \begin{align}\label{J2cont}
        \Ka_2 &\leq \sum_{F\in\Fa}\sum_{\pm\in\{+,-\}}w_F^\pm\sigma_F\int_F\left(|\nabla u_h^\pm|^2 + \sigma_F^2|\llbracket v_h\rrbracket|^2\right)^{\frac{p-2}{2}}|\llbracket v_h \rrbracket| |\llbracket u_h- v_h \rrbracket|\d s \notag\\
        &\leq \mu_\varepsilon\int_\Gamma\sigma\lbrac\left(|\nabla u_h| + \sigma|\llbracket u - v_h\rrbracket|\right)^{p-2}\rbrac_\vec{w}|\llbracket u - v_h \rrbracket|^2 \d s \notag\\
        &\qquad + \varepsilon\int_\Gamma \sigma \lbrac(|\nabla u_h| + \sigma|\llbracket u_h - v_h\rrbracket|)^{p-2}\rbrac_\vec{w}| \llbracket u_h - v_h \rrbracket |^2\d s.
    \end{align}
    Now, for $\Ka_3$, we have
    \begin{align}\label{J31J32}
        \Ka_3 &\leq \int_\Gamma \left|\lbrac |\nabla u|^{p-2}\nabla u - |\nabla u_h|^{p-2}\nabla u_h \rbrac_\vec{w}\right| |\llbracket u_h - v_h \rrbracket| \d s \notag\\
        &\qquad + \int_\Gamma \left|\lbrac |\nabla u_h|^{p-2}\nabla u_h - |\nabla v_h|^{p-2}\nabla v_h \rbrac_\vec{w}\right| |\llbracket u_h - v_h \rrbracket| \d s = \Ka_{3,1} + \Ka_{3,2}.
    \end{align}
    To bound $\Ka_{3,1}$, we work as follows
    \begin{align}\label{J31}
        \Ka_{3,1}
        &\leq 2^{p-2}C_{1,p}\sum_{F\in\Fa}\sum_{\pm\in\{+,-\}}w_F^\pm\frac{\mu_\varepsilon}{\sigma_F}\int_F \left( |\nabla u_h^\pm| + |\nabla(u - u_h^\pm)| \right)^{p-2}|\nabla(u - u_h^\pm)|^2\d s \notag\\
        &\quad + 2^{p-2}C_{1,p}\varepsilon\int_\Gamma\sigma\lbrac\left( |\nabla u_h| + \sigma|\llbracket u_h - v_h \rrbracket| \right)^{p-2}\rbrac_\vec{w}|\llbracket u_h - v_h \rrbracket|^2 \d s \notag\\
        &\leq 2^{2p-3}C_{1,p}\sum_{F\in\Fa}\sum_{\pm\in\{+,-\}}w_F^\pm\frac{\mu_\varepsilon}{\sigma_F}\int_F \left( |\nabla u_h^\pm| + |\nabla(u_h^\pm - v_h^\pm)| \right)^{p-2}|\nabla(u_h^\pm - v_h^\pm)|^2\d s \notag\\
        &\quad + 2^{2p-3}C_{1,p}\mu_\varepsilon\int_\Gamma \sigma^{-1} \lbrac\left( |\nabla u_h| + |\nabla(u - v_h)| \right)^{p-2}|\nabla(u - v_h)|^2 \rbrac_\vec{w}\d s \notag\\
        &\quad + 2^{p-2}C_{1,p}\varepsilon\int_\Gamma\sigma\lbrac\left( |\nabla u_h| + \sigma|\llbracket u_h - v_h \rrbracket| \right)^{p-2}\rbrac_\vec{w}|\llbracket u_h - v_h \rrbracket|^2 \d s,
    \end{align}
    with \eqref{tri_vec} applied in the last step. Similarly, we can show that
    \begin{align}\label{J32}
        \Ka_{3,2} &\leq 2^{p-2}C_{1,p}\sum_{F\in\Fa}\sum_{\pm\in\{+,-\}}w_F^\pm\frac{\mu_\varepsilon}{\sigma_F}\int_F \left( |\nabla u_h^\pm| + |\nabla(u_h^\pm - v_h^\pm)| \right)^{p-2}|\nabla(u_h^\pm - v_h^\pm)|^2\d s \notag\\
        &\qquad + 2^{p-2}C_{1,p}\varepsilon\int_\Gamma\sigma\lbrac\left( |\nabla u_h| + \sigma|\llbracket u_h - v_h \rrbracket| \right)^{p-2}\rbrac_\vec{w}|\llbracket u_h - v_h \rrbracket|^2 \d s.
    \end{align}
    Now, we apply \eqref{invest_complete} in \eqref{J31} and \eqref{J32}, working exactly as for the bound of $\Ja_3$ in the proof of Lemma \ref{th_DGmonot}, to show
    \begin{align}\label{J3cont}
        \Ka_3 &\leq (2^{2p-3} + 2^{p-2})C_{1,p}\varepsilon\int_\Omega \left( |\nabla_h u_h| + |\nabla_h(u_h - v_h)| \right)^{p-2}|\nabla_h(u_h - v_h)|^2\d \vec{x} \notag \\
        &\qquad + 2^{p-1}C_{1,p}\varepsilon\int_\Gamma\sigma\lbrac\left( |\nabla u_h| + \sigma|\llbracket u_h - v_h \rrbracket| \right)^{p-2}\rbrac_\vec{w}|\llbracket u_h - v_h \rrbracket|^2 \d s \notag \\
        &\qquad + 2^{2p-3}C_{1,p}\mu_\varepsilon\int_\Gamma \sigma^{-1} \lbrac\left( |\nabla u_h| + |\nabla(u - v_h)| \right)^{p-2}|\nabla(u - v_h)|^2 \rbrac_\vec{w}\d s.
    \end{align}
    Finally, we work for $\Ka_4$ exactly as in \eqref{J42}, to infer
    \begin{equation}\label{J4cont}
        \Ka_4 \leq \varepsilon \left(\trn u_h - v_h \trn_{(\Ta;p,u_h)}^2 + \trn u - v_h \trn_{(\Ta;p,u_h)}^2 \right).
    \end{equation}
    We substitute \eqref{J1cont}, \eqref{J2cont}, \eqref{J3cont} and \eqref{J4cont} into \eqref{J123cont}, to complete the proof.
\end{proof}

We also have the following coercivity result, for all values of $p\in(2,\infty)$.
\begin{lemma}\label{th_coer}
    Let $p\in(2,\infty)$ (not necessarily rational!) and $\varepsilon\in(0,2^{2-p}/|\theta|)$ (under the convention $1/0 = \infty$), and suppose that  $\sigma$ and $\vec{w}$ satisfy \eqref{sigma}, with the constant $2^{p/2+1/k_p}p^2k_p^2$ in \eqref{GKF} being replaced by any constant $C_p \geq 2^{p+1}$. Then, for all $u_h\in S_\Ta^\vec{r}$, we have
    \begin{equation}\label{coer}
        B_\Ta(u_h;u_h,u_h) \geq \left( 1 - (2^{p-2}|\theta| + \max\{ 1, (4/p')^{p-1}\varepsilon^{p-2}/p \})\varepsilon \right)\trn u_h \trn_p^p.
    \end{equation}
\end{lemma}

\begin{proof}
    We have
    \begin{align}\label{coer_before_inv}
        B_\Ta(u_h;&u_h,u_h) = \int_\Omega |\nabla_h u_h|^p\d \vec{x} + \int_\Gamma \sigma\lbrac\left(|\nabla u_h|^2 + \sigma^2|\llbracket u_h \rrbracket|^2 \right)^{\frac{p-2}{2}}\rbrac_\vec{w}|\llbracket u_h \rrbracket|^2 \d s \notag\\
        & - \int_\Gamma\lbrac|\nabla u_h|^{p-2}\nabla u_h\rbrac_\vec{w}\cdot\llbracket u_h \rrbracket \d s + \theta\int_\Gamma\lbrac\left(|\nabla u_h|^2 + \sigma^2|\llbracket u_h \rrbracket|^2 \right)^{\frac{p-2}{2}}\nabla u_h\rbrac_\vec{w}\cdot\llbracket u_h \rrbracket \d s.
    \end{align}
    For the third term in \eqref{coer_before_inv}, we have
    \begin{align*}
        \bigg|\int_\Gamma\lbrac|\nabla u_h|^{p-2}\nabla u_h\rbrac_\vec{w}&\cdot\llbracket u_h \rrbracket \d s \bigg| \leq \sum_{F\in\Fa}\sum_{\pm\in\{+,-\}}w_F^\pm \int_F|\nabla u_h^\pm|^{p-1}|\llbracket u_h \rrbracket|\d s \\
        &\leq \sum_{F\in\Fa}\sum_{\pm\in\{+,-\}}w_F^\pm \Big(\frac{\sigma_F^{-1}}{4\varepsilon}\|\nabla u_h^\pm\|_{L^p(F)}^p + \left(\frac{4}{p'}\right)^{p-1}\frac{\varepsilon^{p-1}}{p}\sigma_F^{p-1}\|\llbracket u_h \rrbracket\|_{L^p(F)}^p\Big),
    \end{align*}
    whereby H\"{o}lder's inequality and the standard Young inequality \eqref{youngoriginal} are applied in the last step. Now, the inverse estimate \eqref{LpFinv} shows
    \begin{equation}\label{invest_coer}
        \frac{\sigma_F^{-1}}{4\varepsilon}\|\nabla u_h^\pm\|_{L^p(F)}^p \leq \mu_\varepsilon\sigma_F^{-1}\Gg_{K^\pm,F}\|\nabla u_h\|_{L^p(K^\pm)}^p = \varepsilon m_{K^\pm}^{-1}\|\nabla u_h\|_{L^p(K^\pm)}^p.
    \end{equation}
    Combining the above two inequalities results into
    \begin{equation}\label{coer_after_inv}
        \left|\int_\Gamma\lbrac|\nabla u_h|^{p-2}\nabla u_h\rbrac_\vec{w}\cdot\llbracket u_h \rrbracket \d s \right| \leq \varepsilon\int_\Omega |\nabla_h u_h|^p\d \vec{x} + \left(\frac{4}{p'}\right)^{p-1}\frac{\varepsilon^{p-1}}{p}\int_\Gamma \sigma^{p-1}|\llbracket u_h \rrbracket|^p \d s.
    \end{equation}
    For the last term in \eqref{coer_before_inv}, we apply \eqref{young} and standard algebra, to infer
    \begin{align}\label{symterm}
        \bigg| \int_\Gamma\lbrac\big(&|\nabla u_h|^2 + \sigma^2|\llbracket u_h \rrbracket|^2 \big)^{\frac{p-2}{2}}\nabla u_h\rbrac_\vec{w}\cdot\llbracket u_h \rrbracket \d s \bigg| \notag \\
        &\leq 2^{p-2}\varepsilon\left( \int_\Omega |\nabla u_h|^p \d\vec{x} + \int_\Gamma \sigma\lbrac\left(|\nabla u_h|^2 + \sigma^2|\llbracket u_h \rrbracket|^2 \right)^{\frac{p-2}{2}}\rbrac_\vec{w}|\llbracket u_h \rrbracket|^2 \d s \right),
    \end{align}
   applying \eqref{invest_coer} in the last step once again. A substitution of \eqref{coer_after_inv} and \eqref{symterm} into \eqref{coer_before_inv} yields the desired estimate, upon observing the bound
    \begin{equation*}
        (1 - 2^{p-2}|\theta|\varepsilon)\int_\Gamma \sigma\lbrac\left(|\nabla u_h|^2 + \sigma^2|\llbracket u_h \rrbracket|^2 \right)^{\frac{p-2}{2}}\rbrac_\vec{w}|\llbracket u_h \rrbracket|^2 \d s \geq (1 - 2^{p-2}|\theta|\varepsilon)\int_\Gamma \sigma^{p-1}|\llbracket u_h \rrbracket|^p \d s,
    \end{equation*}
    for $\varepsilon\in(0,2^{2-p}/|\theta|)$.
\end{proof}

   Note that \eqref{coer} holds for all $p\in[2,\infty)$, that is, \emph{without} assuming $p\in\mathbb{Q}$. This is due to the proof requiring only the $L^p$-inverse estimate \eqref{LpFinv}.  Below we will see that this ensures unconditional stability of the robust IPDG method \eqref{IPDG} with the choice of weights as in \eqref{sigma}, for all $p\in[2,\infty)$.

\section{Stability and error analysis}\label{DG_staberr}
From now on, we select 
\begin{equation}\label{epsilon}
    \varepsilon = \min\left\{ \frac{2^{1-p}C_{2,p}}{2^{p-2}C_{1,p}+\Tilde{C}_p+(2^{p-1}+1)|\theta|}, \frac{2^{4-p}}{|\theta|}, \frac{p'}{4}\left( \frac{p}{4} \right)^{\frac{1}{p-1}}, \frac{1}{4} \right\},
\end{equation}
in the definition of $\Gg_{K,F}$ in Lemma \ref{th_DGmonot} and, hence, in the definitions of $\sigma$ and $\vec{w}$,
under the convention $\min\{a,\infty,b,c\} = \min\{a,b,c\}$, for $a,b,c\in\mathbb{R}$. With this choice, all terms of the form $C_1 - C_2\varepsilon$ on the left-hand sides of the inequalities in Lemmas \ref{th_DGmonot}, \ref{th_DGcont} and \ref{th_coer} become strictly positive.
\subsection{Stability.}\label{DG_stabest}
To show stability with respect to $\trn\cdot\trn_p$, we require a broken Poincar\'{e}-type inequality.
\begin{lemma}\label{le_poinc}
    Let $p\in[2,\infty)$. There exists a constant $C_\mathrm{PF}$, depending on $\Omega$ and $p$, such that
    \begin{equation}\label{poinc}
        \|v_h\|_{L^p(\Omega)} \leq C_\mathrm{PF}\trn v_h \trn_p \quad \forall v_h\in S_\Ta^\vec{r},
    \end{equation}
    where $\trn\cdot\trn_p$ is the broken Sobolev norm from \eqref{dgnorm}.
\end{lemma}
\begin{proof}
    The proof for general $p$ is given in \cite{BoMasc2025} for polytopic meshes. In particular, upon considering simplicial elements with hanging nodes as polytopes with straight angles, the result in \cite{BoMasc2025} is sufficient of the present setting.
\end{proof}
Now, we are in position to prove a stability result for all $p\in[2,\infty)$.
\begin{theorem}\label{th_stab}
    Let $p\in[2,\infty)$ and suppose that the discontinuity-penalization and the weighted averages are given by \eqref{sigma}. Then, for $f\in L^{p'}(\Omega)$, we have
    \begin{align}\label{stab}
        \trn u_h \trn_p^{p-1} \leq 2C_\mathrm{PF}\|f\|_{L^{p'}(\Omega)}.
    \end{align}
\end{theorem}
\begin{proof}
    We apply the coercivity condition \eqref{coer} and we observe that if $\varepsilon$ satisfies \eqref{epsilon}, then it holds $1 - (2^{p-2}|\theta| + \max\{ 1, (4/p')^{p-1}\varepsilon^{p-2}/p \})\varepsilon \geq 1/2$. Thus, we conclude
   \[
        \frac{1}{2}\trn u_h \trn_p^p \leq B_\Ta(u_h;u_h,u_h) 
        \leq \|f\|_{L^{p'}(\Omega)}\|u_h\|_{L^p(\Omega)} \leq C_\mathrm{PF}\|f\|_{L^{p'}(\Omega)}\trn u_h \trn_p.\vspace{-.7cm}\]
\end{proof}
 
\subsection{Error analysis.}\label{DG_errest}
We begin by a bound of the quasi-norm error by best approximation errors. 
\begin{theorem}\label{th_best_app}
    Let $p\in(2,\infty)\cap\mathbb{Q}$ and let $u\in W^{1,p}_0(\Omega)$  and  $u_h\in S_\Ta^\vec{r}$ be the solutions of \eqref{BVP} and \eqref{IPDG}, respectively. Assume further that $u\in W^{1+1/p,p}(\Omega;\Ta)$. Then, there exists a constant $C_p>0$, depending on $p$, such that 
    \begin{align}\label{best_app}
        \trn &u_h - u\trn_{(\Ta;p,u_h)}^2 + \trn u_h - u\trn_p^p \notag\\
        &\leq C_p\inf_{v_h\in S_\Ta^\vec{r}} \bigg(\trn v_h - u\trn_{(\Ta;p,u_h)}^2 + \int_\Omega \left(|\nabla u| + |\nabla_h(v_h-u)|\right)^{p-2}|\nabla_h(v_h-u)|^2 \d \vec{x} \notag\\
        & \qquad\qquad\qquad + \int_\Gamma\sigma^{-1}\lbrac \left(|\nabla u_h| + |\nabla(v_h - u)| \right)^{p-2}|\nabla(v_h - u)|^2\rbrac_\vec{w}\d s\bigg).
    \end{align}
\end{theorem}
\begin{proof}
    Let $v_h\in S_\Ta^\vec{r}$. We decompose the error as $e = u_h-u = \xi + \eta$, where $\xi = u_h - v_h$ and $\eta = v_h-u$. Then, \eqref{DGmonot} shows 
    \begin{align*}
        C_{{\rm mon},p}(\varepsilon)  \trn \xi \trn_{(\Ta;p,u_h)}^2 
        &\leq B_\Ta(u_h;u_h,\xi) - B_\Ta(u_h;v_h,\xi) +  (2^{p-1}+1)|\theta|\varepsilon\trn \eta \trn_{(\Ta;p,u_h)}^2\\
        & = B_\Ta(u_h;u,\xi) - B_\Ta(u_h;v_h,\xi) +  (2^{p-1}+1)|\theta|\varepsilon\trn \eta \trn_{(\Ta;p,u_h)}^2,
    \end{align*}
    upon employing \eqref{consistencyDG}. Using \eqref{DGcont} on the last bound,
   and recalling the value of $\varepsilon$ from \eqref{epsilon}, gives
    \begin{align}\label{errxi}
        2^{1-p} \trn \xi \trn_{(\Ta;p,u_h)}^2 \leq C_p\bigg(&\trn \eta\trn_{(\Ta;p,u_h)}^2  +  \int_\Omega \left( |\nabla u| + |\nabla_h\eta| \right)^{p-2}|\nabla_h\eta|^2 \d \vec{x} \notag\\
        &+ \int_\Gamma \sigma^{-1} \lbrac\left( |\nabla u_h| + |\nabla \eta| \right)^{p-2}|\nabla \eta|^2\rbrac_\vec{w} \d s\bigg).
    \end{align}
  We combine the ``triangle'' inequality for the quasi-norm, $\trn e \trn_{(\Ta;p,u_h)}^2 \leq 2^{p-1}(\trn\xi\trn_{(\Ta;p,u_h)}^2+\trn\eta\trn_{(\Ta;p,u_h)}^2)$, with \eqref{errxi} to obtain the desired upper bound for the quasi-norm error.

  Also, since $p\ge 2$, we have 
    \begin{align*}
        \trn e \trn_p^p 
        &\leq \int_\Omega (|\nabla_h u_h| + |\nabla_h e|)^{p-2}|\nabla_h e|^2\d \vec{x} + \int_\Gamma\sigma\lbrac (|\nabla u_h| + \sigma|\llbracket e \rrbracket|)^{p-2}\rbrac_\vec{w}|\llbracket e \rrbracket|^2 \d s =  \trn e \trn_{(\Ta;p,u_h)}^2,
    \end{align*}
    and the proof is complete.
\end{proof}

Theorem \ref{th_best_app} implies the following \emph{a priori} error bound for the quasi-norm and the broken Sobolev norm, under additional regularity assumptions on the exact solution and the availability of best approximation estimates.

For brevity, we employ the auxiliary function $Z:\mathbb{R}^+\to\mathbb{R}^+$, by 
\[
Z(x) := \big( C_\mathrm{st} + x\big)^{1-2/p}x^{2/p},
\] with the constant $C_\mathrm{st}>0$ to be defined precisely later.

\begin{theorem}\label{th_rates}
    Suppose that Assumption \ref{assumpt_best_app} holds and that $u\in W^{\vec{s},p}(\Omega;\Ta)$ for some $\vec{s} = (s_K: K\in\Ta)\in\mathbb{N}^{\operatorname{card}(\Ta)}$ with $2\leq s_K\leq r_K+1$. Then, there exists a constant $C$, independent of $\Ta$, of $\vec{r}$ and of $u$, such that 
    \begin{align}\label{rates}
        \trn u_h - u \trn_{(\Ta;p,u_h)}^2 + \trn u_h - u \trn_p^p \leq CZ\bigg( \sum_{K\in\Ta}\Ca_\mathrm{sub}^p(m_K,r_K)\frac{h_K^{p(s_K-1)}}{r_K^{p(s_K-1)}}\|u\|_{W^{s_K,p}(K)}^p \bigg),
    \end{align}
    with $ C_\mathrm{st}=(2C_\mathrm{PF})^{p'}\|f\|_{L^{p'}(\Omega)}^{p'} + \|\nabla u\|_{L^p(\Omega)}^p,$ in $Z$, and \[
    \Ca_\mathrm{sub}(m_K,r_K) = \max\{m_K r_K^{1/p'+\beta_0+(\beta_1-\beta_0)/p}, r_K^{\beta_1+(\beta_2-\beta_1-1)/p},r_K^{\beta_1}\}.\]
\end{theorem}

\begin{proof} Let $v_h\in S_\Ta^\vec{r}$ and set $\eta = v_h-u$ for brevity. Since the exact solution satisfies $u\in W^{2,p}(\Omega;\Ta)$, by further bounding \eqref {best_app} from above, we first show the estimate
 \begin{align}\label{best_app_high}
        &\trn u_h - u \trn_{(\Ta;p,u_h)}^2 + \trn u_h - u \trn_p^p \notag \\
        &\text{ } \leq C_p\inf_{v_h\in S_\Ta^\vec{r}}\Bigg[ Z\Big( \|\nabla_h \eta \|^p_{L^p(\Omega)}\Big) + Z\left( \sum_{K\in\Ta}m_K^p\frac{r_K^{2(p-1)}|\partial K|^p}{|K|^p}\|\eta\|_{L^p(K)}^{p-1}\left( \|\eta\|_{L^p(K)} + h_K\|\eta\|_{W^{1,p}(K)} \right)\right) \notag\\
        &\quad\quad\qquad\qquad + Z\left( \sum_{K\in\Ta}\frac{1}{r_K^2}\|\eta\|_{W^{1,p}(K)}^{p-1}\left( \|\eta\|_{W^{1,p}(K)} + h_K\|\eta\|_{W^{2,p}(K)} \right)\right)\Bigg],
    \end{align}
    where $C_\mathrm{st}=(2C_\mathrm{PF})^{p'}\|f\|_{L^{p'}(\Omega)}^{p'} + \|\nabla u\|_{L^p(\Omega)}^p$ in $Z$ is such that $\|\nabla_h u_h\|_{L^p(\Omega)}^p+\|\nabla u\|_{L^p(\Omega)}^p\leq C_\mathrm{st}$.

     First, H\"{o}lder's inequality implies
    \begin{equation}\label{quasi_to_norm_Omega}
        \int_\Omega \left(|\nabla u| + |\nabla_h \eta|\right)^{p-2}|\nabla_h \eta|^2 \d \vec{x} \leq C_p\left(\|\nabla u\|_{L^p(\Omega)}^p + \|\nabla_h \eta\|_{L^p(\Omega)}^p\right)^{1-2/p}\|\nabla_h \eta\|_{L^p(\Omega)}^2.
    \end{equation}
   Also,
    \begin{equation*}
        \trn \eta \trn_{(\Ta;p,u_h)}^2 = \int_\Omega \left(|\nabla_h u_h| + |\nabla_h \eta|\right)^{p-2}|\nabla_h \eta|^2 \d \vec{x} + \int_\Gamma \sigma\lbrac \left( |\nabla u_h| + \sigma|\llbracket \eta \rrbracket| \right)^{p-2}\rbrac_\vec{w}|\llbracket \eta \rrbracket|^2 \d s \equiv \Ia_1 + \Ia_2.
    \end{equation*}
    For $\Ia_1$, we work exactly as in \eqref{quasi_to_norm_Omega} to show
    \begin{align}\label{QNbound1}
        \Ia_1 \leq C_p\left(\|\nabla_h u_h\|_{L^p(\Omega)}^p + \|\nabla_h \eta\|_{L^p(\Omega)}^p\right)^{1-\frac{2}{p}}\|\nabla_h \eta\|_{L^p(\Omega)}^2
        \leq Z\Big( \|\nabla_h \eta \|^p_{L^p(\Omega)}\Big)
    \end{align}
    For $\Ia_2$, we apply H\"{o}lder's inequality to infer
    \begin{align}\label{I2first}
        \Ia_2 &= \sum_{F\in\Fa}\sum_{\pm\in\{+,-\}}w_F^\pm \sigma_F \int_F \left( |\nabla u_h^\pm| + \sigma_F|\llbracket \eta \rrbracket| \right)^{p-2}|\llbracket \eta \rrbracket|^2 \d s \notag\\
        &\leq \Bigg( \sum_{F\in\Fa}\sum_{\pm\in\{+,-\}}w_F^\pm \sigma_F^{-1} \int_F \left( |\nabla u_h^\pm| + \sigma_F|\llbracket \eta \rrbracket| \right)^p \d s \Bigg)^{1-2/p}\Bigg( \sum_{F\in\Fa}\sum_{\pm\in\{+,-\}}w_F^\pm \sigma_F^{p-1} \int_F |\llbracket \eta \rrbracket|^p \d s \Bigg)^{2/p} \notag\\
        &\leq C_p\Bigg( \sum_{F\in\Fa}\sum_{\pm\in\{+,-\}}w_F^\pm \sigma_F^{-1} \|\nabla u_h^\pm\|_{L^p(F)}^p + \int_\Gamma\sigma^{p-1}|\llbracket \eta \rrbracket|^p \d s \Bigg)^{1-2/p}\Bigg( \int_\Gamma\sigma^{p-1}|\llbracket \eta \rrbracket|^p \d s \Bigg)^{2/p}.
    \end{align}
    For $F\in\Fa$, $\pm\in\{+,-\}$, an application of the inverse estimate \eqref{LpFinv} gives
    \begin{align*}
        w_F^\pm \sigma_F^{-1} \|\nabla u_h^\pm\|_{L^p(F)}^p &\leq C_pw_F^\pm \sigma_F^{-1}\frac{r_{K^\pm}^2|F_{K^\pm}|}{|K^\pm|} \|\nabla u_h\|_{L^p(K^\pm)}^p \leq C_pw_F^\pm \sigma_F^{-1}\Gg_{K^\pm,F} \|\nabla u_h\|_{L^p(K^\pm)}^p \\
        &= C_p\zeta_F^\pm\Gg_{K^\pm,F} \|\nabla u_h\|_{L^p(K^\pm)}^p = C_pm_{K^\pm}^{-1} \|\nabla u_h\|_{L^p(K^\pm)}^p,
    \end{align*}
    and hence, the above inequality implies
    \begin{equation}\label{bound_uh_face}
        \sum_{F\in\Fa}\sum_{\pm\in\{+,-\}}w_F^\pm \sigma_F^{-1} \|\nabla u_h^\pm\|_{L^p(F)}^p \leq C_p\|\nabla_h u_h\|_{L^p(\Omega)}^p \leq C_pC_\mathrm{st}.
    \end{equation}
    Moreover, we apply a multiplicative trace inequality, to show
    \begin{align}\label{jump_er_1}
        \int_\Gamma\sigma^{p-1}|\llbracket \eta \rrbracket|^p \d s &= \sum_{F\in\Fa}\sigma_F^{p-1}\|\llbracket \eta \rrbracket\|_{L^p(F)}^p \leq C_p \sum_{F\in\Fa}\sum_{K\in\Ta_F}\sigma_F^{p-1}\| \eta|_K \|_{L^p(F)}^p \notag \\
        &\leq C_p \sum_{F\in\Fa}\sum_{K\in\Ta_F}\sigma_F^{p-1}\frac{|F_K|}{|K|}\|\eta\|_{L^p(K)}^{p-1}\left( \|\eta\|_{L^p(K)} + h_K\|\eta\|_{W^{1,p}(K)} \right).
    \end{align}
    For $F\in\Fa$, we have by definition of $\sigma_F$, that
    \begin{equation*}
        \sigma_F^{p-1} = (\zeta_F^+ + \zeta_F^-)^{1-p} \leq (\zeta_F^\pm)^{1-p} = C_pm_K^{p-1}\left(\frac{r_K^2|F_K|}{|K|}\right)^{p-1} \quad\forall K=K^\pm\in\Ta_F.
    \end{equation*}
    We substitute the above inequality into \eqref{jump_er_1}, to obtain
    \begin{align}\label{jump_er}
        \int_\Gamma\sigma^{p-1}|\llbracket \eta \rrbracket|^p \d s &\leq C_p\sum_{F\in\Fa}\sum_{K\in\Ta_F}m_K^{p-1}\frac{r_K^{2(p-1)}|F_K|^p}{|K|^p}\|\eta\|_{L^p(K)}^{p-1}\left( \|\eta\|_{L^p(K)} + h_K\|\eta\|_{W^{1,p}(K)} \right) \notag \\
        &\leq C_p\sum_{K\in\Ta}m_K^p\frac{r_K^{2(p-1)}|\partial K|^p}{|K|^p}\|\eta\|_{L^p(K)}^{p-1}\left( \|\eta\|_{L^p(K)} + h_K\|\eta\|_{W^{1,p}(K)} \right).
    \end{align}
    The combination of \eqref{bound_uh_face}, \eqref{jump_er} and \eqref{I2first} shows
    \begin{align}\label{QNbound2}
        \Ia_2 &\leq C_p Z\Big(\sum_{K\in\Ta}m_K^p\frac{r_K^{2(p-1)}|\partial K|^p}{|K|^p}\|\eta\|_{L^p(K)}^{p-1}\left( \|\eta\|_{L^p(K)} + h_K\|\eta\|_{W^{1,p}(K)} \right)\Big).
    \end{align}
    It remains to bound the last term in \eqref{best_app}. To that end, we apply H\"{o}lder's inequality, to obtain
    \begin{align}\label{trace1}
        &\int_\Gamma\sigma^{-1}\lbrac \left(|\nabla u_h| + |\nabla\eta| \right)^{p-2}|\nabla\eta|^2\rbrac_\vec{w}\d s = \sum_{F\in\Fa}\sum_{\pm\in\{+,-\}}w_F^\pm\sigma_F^{-1}\int_F \left(|\nabla u_h^\pm| + |\nabla\eta^\pm| \right)^{p-2}|\nabla\eta^\pm|^2 \d s \notag \\
        &\text{ } \leq C_p\Bigg( \sum_{F\in\Fa}\sum_{\pm\in\{+,-\}}w_F^\pm\sigma_F^{-1}\int_F \left(|\nabla u_h^\pm|^p + |\nabla\eta^\pm|^p \right) \d s \Bigg)^{1-2/p}\Bigg( \sum_{F\in\Fa}\sum_{\pm\in\{+,-\}}w_F^\pm\sigma_F^{-1}\int_F |\nabla\eta^\pm|^p \d s \Bigg)^{2/p} \notag\\
        &\text{ } \leq C_pZ\bigg( \sum_{F\in\Fa}\sum_{\pm\in\{+,-\}}w_F^\pm\sigma_F^{-1}\|\nabla\eta^\pm\|_{L^p(F)}^p \bigg),
    \end{align}
    where at the last step, \eqref{bound_uh_face} is applied. Furthermore, we have
    \begin{align*}
        w_F^\pm\sigma_F^{-1}\|\nabla\eta^\pm\|_{L^p(F)}^p &= \zeta_F^\pm\|\nabla\eta^\pm\|_{L^p(F)}^p \leq C_p m_{K^\pm}^{-1}\Gg_{K^\pm,F}\|\nabla\eta^\pm\|_{L^p(F)}^p \leq C_p m_{K^\pm}^{-1}\frac{|K^\pm|}{|F_{K^\pm}|r_{K^\pm}^2}\|\nabla\eta^\pm\|_{L^p(F)}^p \notag\\
        &\leq C_p \frac{m_{K^\pm}^{-1}}{r_{K^\pm}^2}\|\eta\|_{W^{1,p}(K^\pm)}^{p-1}\left( \|\eta\|_{W^{1,p}(K^\pm)} + h_{K^\pm}\|\eta\|_{W^{2,p}(K^\pm)} \right),
    \end{align*}
    where at the last step, we applied a multiplicative trace inequality. Thus, it holds
    \begin{equation*}
        \sum_{F\in\Fa}\sum_{\pm\in\{+,-\}}w_F^\pm\sigma_F^{-1}\|\nabla\eta^\pm\|_{L^p(F)}^p \leq C_p\sum_{K\in\Ta}\frac{1}{r_{K}^2}\|\eta\|_{W^{1,p}(K)}^{p-1}\left( \|\eta\|_{W^{1,p}(K)} + h_K\|\eta\|_{W^{2,p}(K)} \right).
    \end{equation*}
    Then, \eqref{trace1} becomes
    \begin{align}\label{er_trace_grad}
        &\int_\Gamma\sigma^{-1}\lbrac \left(|\nabla u_h| + |\nabla\eta| \right)^{p-2}|\nabla\eta|^2\rbrac_\vec{w}\d s \leq C_p Z\Big(\sum_{K\in\Ta}\frac{1}{r_K^2}\|\eta\|_{W^{1,p}(K)}^{p-1}\left( \|\eta\|_{W^{1,p}(K)} + h_K\|\eta\|_{W^{2,p}(K)} \right)\Big).
    \end{align}
    The summation of inequalities \eqref{quasi_to_norm_Omega}, \eqref{QNbound1}, \eqref{QNbound2} and \eqref{er_trace_grad}, and the substitution of the resulting inequality into \eqref{best_app} yield \eqref{best_app_high}.

To complete the estimate, we observe that there exists a constant $C_\mathrm{reg}$, depending on the shape-regularity of $\Ta$, such that 
    \begin{equation}\label{shapereg}
        \frac{|\partial K|}{|K|}\leq \frac{C_\mathrm{reg}}{h_K}, \text{ for all } K\in\Ta.
    \end{equation}
     Moreover, we select $v_h\in S_\Ta^\vec{r}$ as $v_h|_K = \Pi_{r_K}(u|_K)$. We substitute \eqref{proj_elem} and \eqref{shapereg} into \eqref{best_app_high} and the proof is complete after some elementary calculations.
\end{proof}

    The above theorem shows that the convergence rates of the robust IPDG method \eqref{IPDG} are \emph{optimal} with respect to the mesh size and suboptimal with respect to the polynomial degrees. The rate reduction depends on the exponents $\beta_0$, $\beta_1$ and $\beta_2$ that appear in Assumption \ref{assumpt_best_app}. We observe that if we had an optimal $hp$-approximation result, namely, with $\beta_0=\beta_1=\beta_2=0$, then the rates with respect to $r_K$ error bound \eqref{rates} would be reduced only by $1/p'$ order of $r$, which stems from the choice of the penalty parameter and cannot be improved. This is in accordance with the corresponding error estimates for the linear elliptic problem with $p=2$ (and $1/p'=1/2$) which is proven in \cite{GHM10}.

Below, we present an explicit error bound that follows immediately from substituting the $\beta_m$'s from Remark \ref{th_hp_polyap}.

\begin{corollary}\label{cor_rates}
    Under the assumptions of Theorem \ref{th_rates}, with $\beta_0$, $\beta_1$ and $\beta_2$ stated in Remark \ref{th_hp_polyap}, \eqref{rates} holds with
    \begin{equation*}
        \Ca_\mathrm{sub}(m_K,r_K) = \begin{cases}
            m_Kr_K^{1/p'}, & d = 1; \\
            \max\{m_Kr_K^{1/p'+d(1+1/p)(1/2-1/p)}, r_K^{d(1-2/p)}\}, & d=3, \text{ } p\in(6,\infty), \text{ } s_K=2; \\
            m_Kr_K^{1/p'+d(1/2-1/p)}, &\text{otherwise}.
        \end{cases}
    \end{equation*}
\end{corollary}

\begin{remark}\label{rem_subopt_peq2}
    Substituting $p=2$ into \eqref{rates} with the exponents of Corollary \ref{cor_rates}, we retrieve the $hp$-version error estimates of the robust IPDG method \cite{DongGeo2022} for the Poisson problem, that are optimal in $h_K$ and suboptimal in $r_K$ by $1/2$.
\end{remark}

\section{Extension to curved polytopic meshes}\label{sec_DGpoly}

We now discuss the extension of the robust IPDG method \eqref{IPDG} for meshes consisting of very general curved polytopic elements; additional mesh assumptions will be given below. Due to the general structure of $\Ta$, we remove the assumption that $\Omega$ is polytopic, and we allow it to have curved Lipschitz boundary as well. The framework of this section also admits meshes with curved simplices, and thus immediately generalizes the results obtained above to curved domains. However, the analysis will be carried out in the general setting of essentially arbitrarily-shaped elements introduced in \cite{CanDongGeo2021} for linear PDE problems.

We recall the definition of faces and interfaces from Section \ref{DG_fes}; in particular, in the context of polytopic meshes, the definitions of the \emph{interfaces} and the related parameters remain those of Section \ref{DG_fes}; see Figure \ref{fig:interfaces} for an illustration.

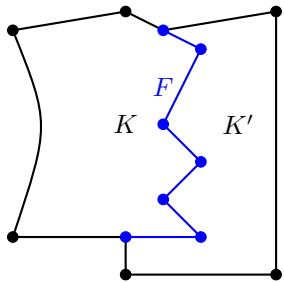
\begin{figure}[h!]
    \centering
    \begin{tikzpicture}
    \draw[thick] (0, 2.75) -- (1.5, 3) -- (2,2.75)  -- (3.5,3);
    \draw[thick] (3.5, 3) -- (3.5, -0.5) -- (1.5, -0.5) -- (1.5, 0)  -- (0, 0);
    \draw[thick] (0, 0) .. controls (0.5, 1.5) .. (0, 2.75);

    \filldraw[black] (0, 2.75) circle (2pt);
    \filldraw[black] (1.5, 3) circle (2pt);
    \filldraw[black] (3.5, 3) circle (2pt);
    \filldraw[black] (3.5, -0.5) circle (2pt);
    \filldraw[black] (1.5, -0.5) circle (2pt);
    \filldraw[black] (0, 0) circle (2pt);

    \draw[blue, thick] (2, 2.75) -- (2.5, 2.5) -- (2, 1.5);
    \draw[blue, thick] (2, 1.5) -- (2.5, 1) -- (2, 0.5) -- (2.5, 0) -- (1.5, 0);
    
    \filldraw[blue] (2, 2.75) circle (2pt);
    \filldraw[blue] (2.5, 2.5) circle (2pt);
    \filldraw[blue] (2, 1.5) circle (2pt);
    \filldraw[blue] (2.5, 1) circle (2pt);
    \filldraw[blue] (2, 0.5) circle (2pt);
    \filldraw[blue] (2.5, 0) circle (2pt);
    \filldraw[blue] (1.5, 0) circle (2pt);

    \node at (1.5, 1.5) {$K$};
    \node at (3, 1.5) {$K'$};
    \node at (2,2) {$\textcolor{blue}{F}$};
    \end{tikzpicture}
    \caption{Two neighboring elements $K$ and $K'$ sharing an interface $F$ (blue) containing six faces. The elements $K$ and $K'$ have nine faces each.}
    \label{fig:interfaces}
\end{figure}

For the design and analysis of the robust IPDG method \eqref{IPDG} on polytopic meshes, we impose the following (very mild) mesh assumptions; we refer to \cite{CanDongGeo2021} for more details.

\begin{assumption}\label{starshaped}
    For all $K\in\Ta$, we assume that $K$ is a Lipschitz domain and its boundary $\partial K$ is subdivided into mutually exclusive sets $\{F_i\}_{i=1}^{n_K}$, such that for $i=1,\dotso,n_K$, the following property is satisfied: there exists an interior point $\vec{x}_i^0\in K$ and a sub-element $K_{F_i}$, depending on $\vec{x}_i^0$, with $F_i\subset K_{F_i}$ and $K_{F_i}$ having $d$ planar faces, meeting at $\vec{x}_i^0$, such that
    \begin{enumerate}
        \item $K_{F_i}$ is star-shaped with respect to $\vec{x}_i^0$, and
        \item $\vec{m}_i(\vec{x}) \cdot \vec{n}(\vec{x}) > 0$, where $\vec{m}_i(\vec{x}) = \vec{x} - \vec{x}_i^0$ for $\vec{x}\in K_{F_i}$, and $\vec{n}$ is the outward unit normal to $F_i$.
    \end{enumerate}
\end{assumption}

We begin by stating a multiplicative trace inequality, whose proof is omitted, since it is similar to that of \cite[Lemma 4.7]{CanDongGeo2021} for the case $q=2$.
\begin{lemma}\label{th_multitr}
    Let $q\in(1,\infty)$, $K\in\Ta$ and let $\{F_i\}_{i=1}^{n_K}$ be the partition of $\partial K$, as per Assumption \ref{starshaped}. Then, for $v\in W^{1,q}(K)$ and $i=1,\dotso,n_K$, we have the trace inequality
    \begin{equation}\label{multitr}
        \|v\|_{L^q(F_i)}^q \leq \frac{\|v\|_{L^q(K_{F_i})}^{q-1}}{\displaystyle\min_{F_i}(\vec{m}_i\cdot\vec{n})}\left( d\|v\|_{L^q(K_{F_i})} + q\|\vec{m}_i\|_{L^\infty(K_{F_i})}\|\nabla v\|_{L^q(K_{F_i})} \right).
    \end{equation}
\end{lemma}

\subsection{Inverse estimates on sub-elements.}\label{DGpoly_subelem}
We now extend the polynomial trace-type inverse estimates proven for simplices above to arbitrarily-shaped elements satisfying Assumption \ref{starshaped}. As before, such estimates are crucial for the definition of the discontinuity-penalization parameter and for the stability and error analysis of the robust IPDG method. 

We start with a trace-type inverse estimate that bounds integrals over the sets $\{F_i\}_{i=1}^{n_K}$ that form the partitions of $\partial K$, $K\in\Ta$, by the respective integrals over the sub-elements $K_{F_i}$. 
\begin{lemma}\label{th_LpFinv_poly}
    Let $q\in(0,\infty)$ and let $K\in\Ta$ be a Lipschitz domain, satisfying Assumption \ref{starshaped}. Then, for each $F_i\subset\partial K$, $i=1,\dotso,n_K$ and each $v\in\Pa_r(K)$, we have
    \begin{equation}\label{LpFinv_poly}
        \int_{F_i}|v|^q\d s \leq 2^{q+1}C_{\inv,d}\frac{2r^2 + 1}{\displaystyle\min_{F_i}(\vec{m}_i\cdot\vec{n})}\int_{K_{F_i}}|v|^q\d\vec{x}.
    \end{equation}
\end{lemma}
\begin{proof} 
 The proof for $q\in(1,\infty)$ follows broadly that of \cite[Lemma 4.4]{CanDongGeo2021} for $q = 2$ (We note in passing that \cite[Lemma 4.4]{CanDongGeo2021} gives a slightly smaller constant than (7.2) for $q = 2$). In particular, the only significant departures are replacing $|v|^2$ by $|v|^q$ and using \eqref{LpFinv} instead of a standard $L^2$-norm trace inverse estimate.
 
 The case $q\in(0,1]$, however, requires additional considerations. This is because $\operatorname{div}(|v|^q\vec{m}_i)$ may not be integrable for $q\in(0,1]$. To that end, following \cite[Lemma 4.4]{CanDongGeo2021}, we partition $F_i$ into $N$ $(d-1)$-dimensional curved simplices $F_i^j$, $j=1,\dotso,N$, which are subordinate to the vertices possibly contained in $F_i$; we assume that $N$ is large enough for this requirement to be fulfilled. 
 We consider a partition of $K_{F_i}$ into, possibly curved, simplices $K_i^j$ with one face $F_i^j$ and the remaining vertex being $\vec{x}_i^0$, connected to the other vertices with straight/planar faces; this is possible, due to the star-shapedness of $K_{F_i}$, with respect to $\vec{x}_i^0$. Note that each $F_i^j$ may include at most one constituent face of $F_i$, or part thereof.

    Let now $\Tilde{F}_i^j$ denote the straight/planar face connecting the $d-1$ vertices of $F_i^j$. Let also $\underline{K}_i^j$ be the largest straight-faced simplex contained in $K_i^j$, with a face $\underline{F}_i^j$ parallel to $\Tilde{F}_i^j$ and the remaining faces being subsets of the straight faces of $K_i^j$; we refer to Figure \ref{fig:curved_inv_est} for an illustration.
    
    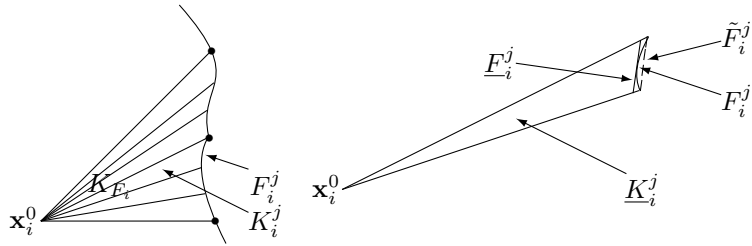
\begin{figure}[h]
			\centering
			\setlength{\unitlength}{2.3cm}
			\begin{picture}(0,1.2)
				\put(-1,0){\line(1,1){.99}}
				\put(-1,0){\line(1,0){1}}
				
				\put(-1,0){\line(3,2){.96}}
				\put(-1,0){\line(5,4){1}}
				\put(-1,0){\line(2,1){.97}}
				\put(-1,0){\line(3,1){.93}}
				\put(-1,0){\line(6,1){.95}}
				\put(-.055, .46){\vx}
				\put(-.04, .96){\vx}
				\put(-.02, -.02){\vx}

				\qbezier(.07, -.13)(-0.15, 0.3)(-.03, .49)
				\cbezier(-.03, .49)(-0.1, 0.8)(0.2, 0.8)(-.2, 1.25)
				\put(-.6,.2){\makebox(0,0){$\color{black} K_{F_i}$}}
				\put(.3,.0){\makebox(0,0){$\color{black} K_i^j$}}
				\put(.2,.05){\vector(-2,1){.5}}
				\put(.3,.2){\makebox(0,0){$\color{black} F_i^j$}}
				\put(.2,.3){\vector(-3,1){.25}}
				\put(-1.1,0){\makebox(0,0){$\color{black} \vec{x}^0_i$}}
			\end{picture}
			\setlength{\unitlength}{4.2cm}
			\begin{picture}(1.1,0)(-1.35,-.1)
				\put(-1,0){\line(2,1){.97}}
				\put(-1,0){\line(3,1){.94}}
				\qbezier(-.056, .312)(-0.088, 0.38)(-.03, .483)
				
				\put(-.056, .312){\line(1,7){.005}}
				\put(-.051, .36){\line(1,7){.005}}
				\put(-.045, .41){\line(1,7){.005}}
				\put(-.039, .45){\line(1,7){.005}}
				
				\put(-.079, .305){\line(1,7){.024}}
				
				\put(-.05,.0){\makebox(0,0){$\color{black} \underline{K}_i^j$}}
				\put(-.1,.05){\vector(-2,1){.35}}
				\put(.25,.28){\makebox(0,0){$\color{black} F_i^j$}}
				\put(.2,.3){\vector(-3,1){.26}}
				\put(.25,.48){\makebox(0,0){$\color{black} \tilde{F}_i^j$}}
				\put(.2,.47){\vector(-4,-1){.235}}
				\put(-.5,.4){\makebox(0,0){$\color{black} \underline{F}_i^j$}}
				\put(-.45,.4){\vector(7,-1){.37}}
				\put(-1.05,0){\makebox(0,0){$\color{black} \vec{x}^0_i$}}
			\end{picture}
			
			\caption{(Reproduction of \cite[Figure 4]{CanDongGeo2021}). Partitioned curved sub-element $K_{F_i}\subset K\in\Ta$; \vx denotes a vertex of $K$ (left); detail with $\underline{K}_i^j$ and faces $\tilde{F}_i^j$ and $\underline{F}_i^j$ (right).}
			\label{fig:curved_inv_est}
		\end{figure}
    Let also, $\kappa_i^j = K_i^j \setminus \underline{K}_i^j$ and assume for the moment, that $\kappa_i^j\neq\varnothing$. Since each $F_i$ is image of a finite number of Lipschitz functions locally, we set $L$ be the Lipschitz constant of a parametrization of $F_i^j$ with respect to $\underline{F}_i^j$, giving $|F_i^j|\leq L |\underline{F}_i^j|$. Also, $|\kappa_i^j|\leq L\operatorname{diam}(\underline{F}_i^j)|\underline{F}_i^j|$, as the maximum Euclidean distance between $\underline{F}_i^j$ and $\Tilde{F}_i^j$ is bounded from above by $L\operatorname{diam}(\underline{F}_i^j)$. Therefore, $|\kappa_i^j|\to 0$ as $N\to\infty$, faster than $|F_i^j|$, by an order of $L\operatorname{diam}(\underline{F}_i^j)$.

    Now, from the Divergence Theorem we have 
    \begin{align}\label{divtheo}
        \int_{\kappa_i^j}\operatorname{div}\big[ (v^2 + & |\kappa_i^j|)^{q/2}\vec{m}_i\big]\d\vec{x} 
        = \int_{F_i^j}(v^2 + |\kappa_i^j|)^{q/2}\vec{m}_i\cdot\vec{n}\d s + \int_{\underline{F}_i^j}(v^2 + |\kappa_i^j|)^{q/2}\vec{m}_i\cdot\vec{n}_{\underline{F}_i^j}\d s,
    \end{align}
    with $\vec{n}_\omega$ denoting the outward unit normal vector of a domain $\omega$, upon observing that $\vec{m}_i\cdot\vec{n}_{\partial\kappa_i^j}=0$ on $\partial\kappa_i^j\setminus(F_i^j\cup\underline{F}_i^j)$. The product rule and standard estimation yield
    \begin{align*}
        \int_{\kappa_i^j}\operatorname{div}\big[ (v^2 + |\kappa_i^j|)^{q/2}\vec{m}_i\big]\d\vec{x} 
        \leq &\ q\|\vec{m}_i\|_{L^\infty(\kappa_i^j)}\|\nabla v\|_{L^\infty(\kappa_i^j)}\|(v^2 + |\kappa_i^j|)^{(q-1)}\|_{L^\infty(\kappa_i^j)}|\kappa_i^j| \\
        &+ d(\|v\|_{L^\infty(\kappa_i^j)}^2 + |\kappa_i^j|)^{q/2}|\kappa_i^j|,
    \end{align*}
  as $\operatorname{div}\vec{m}_i=d$. Since $q\in(0, 1]$, we have 
   $ 
        \|(v^2 + |\kappa_i^j|)^{(q-1)/2}\|_{L^\infty(\kappa_i^j)}|\kappa_i^j| 
        \leq |\kappa_i^j|^{(1+q)/2}
    $. 
    Thus, from the above relations we infer
    \begin{align}\label{divchain}
        \min_{F_i^j}(\vec{m}_i\cdot\vec{n})\int_{F_i^j}|v|^q\d s &\leq \min_{F_i^j}(\vec{m}_i\cdot\vec{n})\int_{F_i^j}(v^2 + |\kappa_i^j|)^{q/2}\d s \notag \\
        &\leq q\|\vec{m}_i\|_{L^\infty(\kappa_i^j)}\|\nabla v\|_{L^\infty(\kappa_i^j)}|\kappa_i^j|^{1/2+q/2} + d(\|v\|_{L^\infty(\kappa_i^j)}^2 + |\kappa_i^j|)^{q/2}|\kappa_i^j| \notag\\
        &\quad + \Big| \int_{\underline{F}_i^j}(v^2 + |\kappa_i^j|)^{q/2}\vec{m}_i\cdot\vec{n}_{\underline{F}_i^j}\d s \Big|.
    \end{align}
    Now, since $(v^2 + |\kappa_i^j|)^{q/2}\le |v|^q + |\kappa_i^j|^{q/2}$ the last term can be further bounded as follows
    \begin{align}\label{breaklast_combo}
        \Big| \int_{\underline{F}_i^j}(v^2 + |\kappa_i^j|)^{q/2}\vec{m}_i\cdot\vec{n}_{\underline{F}_i^j}\d s \Big|
        \leq \max_{\underline{F}_i^j}(\vec{m}_i\cdot\vec{n}_{\underline{F}_i^j})\int_{\underline{F}_i^j}|v|^q\d s + \max_{\underline{F}_i^j}(\vec{m}_i\cdot\vec{n}_{\underline{F}_i^j})|\underline{F}_i^j||\kappa_i^j|^{q/2}.
    \end{align}
   For the face integral at the last step of \eqref{breaklast_combo}, we apply the trace-inverse estimate \eqref{LpFinv}, to obtain
    \begin{equation} \label{straight_face_inv}
        \int_{\underline{F}_i^j}|v|^q\d s \leq 2^{q+1}C_{\inv,d}\frac{(2r^2 + 1)|\underline{F}_i^j|}{d|\underline{K}_i^j|}\int_{\underline{K}_i^j}|v|^q\d\vec{x}.
    \end{equation}
    We combine \eqref{divchain}, \eqref{breaklast_combo} and \eqref{straight_face_inv}, to get
    \begin{equation}\label{before_limit}
        \int_{F_i^j}|v|^q\d s \leq 2^{q+1}C_{\inv,d} \frac{\max_{\underline{F}_i^j}(\vec{m}_i\cdot\vec{n}_{\underline{F}_i^j})(2r^2 + 1)|\underline{F}_i^j|}{d\min_{F_i^j}(\vec{m}_i\cdot\vec{n})|\underline{K}_i^j|}\int_{\underline{K}_i^j}|v|^q\d\vec{x} + \lambda(\kappa_i^j),
    \end{equation}
    where
    \begin{equation*}
        \lambda(\kappa_i^j) = \frac{\max_{\underline{F}_i^j}(\vec{m}_i\cdot\vec{n}_{\underline{F}_i^j})|\underline{F}_i^j|}{\min_{F_i^j}(\vec{m}_i\cdot\vec{n})}|\kappa_i^j|^{q/2} + q\|\vec{m}_i\|_{L^\infty(\kappa_i^j)}\|\nabla v\|_{L^\infty(\kappa_i^j)}|\kappa_i^j|^{(1+q)/2} + d(\|v\|_{L^\infty(\kappa_i^j)}^2 + |\kappa_i^j|)^{q/2}|\kappa_i^j|.
    \end{equation*}

    On the other hand, if $\kappa_i^j=\varnothing$, then we have $F_i^j = \underline{F}_i^j$, since $F_i^j$ is not curved, and so with $\lambda(\kappa_i^j) = 0$. Thus, \eqref{before_limit} follows immediately from \eqref{straight_face_inv}.
    
   As $N\to\infty$, we have $|\kappa_i^j|\to 0$ and, hence,
      $  \lambda(\kappa_i^j)\to 0$ and  $\max_{\underline{F}_i^j}(\vec{m}_i\cdot\vec{n}_{\underline{F}_i^j})/\min_{F_i^j}(\vec{m}_i\cdot\vec{n}) \to 1$.
    Thus, for all $\epsilon>0$, we can choose $N\in\mathbb{N}$ large enough to satisfy
    \begin{align*}
        \int_{F_i^j}|v|^q\d s &\leq 2^{q+1}C_{\inv,d} (1 + \epsilon) \frac{(2r^2 + 1)|\underline{F}_i^j|}{d|\underline{K}_i^j|}\int_{\underline{K}_i^j}|v|^q\d\vec{x}
        \leq 2^{q+1}C_{\inv,d} (1 + \epsilon) \frac{(2r^2 + 1)|F_i^j|}{d|\underline{K}_i^j|}\int_{K_i^j}|v|^q\d\vec{x},
    \end{align*}
    since $\underline{F}_i^j\subset F_i^j$ and $\underline{K}_i^j\subset K_i^j$. Moreover, the divergence theorem shows
    \begin{equation*}
        d|K_i^j| = \int_{K_i^j}\operatorname{div}\vec{m}_i\d\vec{x} = \int_{F_i^j}\vec{m}_i\cdot\vec{n}\d s \geq \min_{F_i^j}(\vec{m}_i\cdot\vec{n})|F_i^j|,
    \end{equation*}
    or 
    \begin{equation*}
        \frac{|F_i^j|}{d|\underline{K}_i^j|}  \leq \frac{1}{\min_{F_i^j}(\vec{m}_i\cdot\vec{n})} \Big( 1 + \frac{|\kappa_i^j|}{|\underline{K}_i^j|} \Big) \leq \frac{1 + \epsilon}{\min_{F_i^j}(\vec{m}_i\cdot\vec{n})},
    \end{equation*}
    for all $\epsilon>0$, when $N$ is large enough. Combining the above, we deduce
    \begin{align*}
        \int_{F_i}|v|^q\d s &= \sum_{j=1}^N \int_{F_i^j}|v|^q\d s \leq 2^{q+1}C_{\inv,d} (1 + \epsilon)^2 \sum_{j=1}^N \frac{2r^2 + 1}{\min_{F_i^j}(\vec{m}_i\cdot\vec{n})}\int_{K_i^j}|v|^q\d\vec{x} \\ 
        &\leq 2^{q+1}C_{\inv,d} (1 + \epsilon)^2 \frac{2r^2 + 1}{\min_{F_i}(\vec{m}_i\cdot\vec{n})}\int_{K_{F_i}}|v|^q\d\vec{x}.
    \end{align*}
       
    The above inequality holds for all $\epsilon>0$; thus, the proof of the case $q\in(0,1]$ is complete.
\end{proof}

\subsection{An auxiliary $L^\infty\to L^q$ inverse estimate.}\label{DGpoly_looinv}
We now prove an inverse estimate bounding $\|v\|_{L^\infty(\hat\Kk)}$ of a polynomial $v$ by $\|v\|_{L^q(\hat\Kk)}$ for $q\in(0,\infty)$, over, so-called, generalized reference prisms $\hat{\Kk}$, defined below.

\begin{defin}\label{grp}
    Let $\hat{F}^0:=[0,1]^{d-1}\subset\mathbb{R}^d$ and $\phi:\hat{F}^0\to\mathbb{R}$ be a Lipschitz continuous function. A generalized reference prism is a domain $\hat\Kk=\hat\Kk_\phi\subset\mathbb{R}^d$, defined by
    \begin{equation*}
        \hat\Kk := \{(\vec{x},x_d)\in\hat{F}^0\times\mathbb{R}: \text{ } 0 \leq x_d \leq \phi(\vec{x})\},
    \end{equation*}
    such that: 1) $[0,1]^d\subset\hat\Kk$, and 2) the flat base $\hat{F}^0$ is star-shaped with respect to the curved base $\hat{F} = \{(\vec{x},\phi(\vec{x}))\in\mathbb{R}^d: \text{ } \vec{x}\in\hat{F}^0\}$.
\end{defin}
Since $[0,1]^d\subset\hat\Kk$, we immediately infer that $\phi(\vec{x})\geq1$, for all $\vec{x}\in\hat{F}^0$.

We now show that the simplicial trace-inverse estimate \eqref{LpFinv} implies a polynomial inverse estimate on reference generalized prisms.

\begin{lemma}\label{th_LinfLq}
    Let $\hat\Kk$ a generalized reference prism, $q\in(0,\infty)$ and $r\in\mathbb{N}_0$. Then, for $v\in\Pa_r(\hat\Kk)$, we have 
    \begin{equation}\label{LinfLq}
        \|v\|_{L^\infty(\hat\Kk)}^q \leq 2^{d(q+1)} C_{\infty,d}(2r^2 + 1)^d\|v\|_{L^q(\hat\Kk)}^q,
    \end{equation}
    with $C_{\infty,d}=4,32$, for $d=2,3$, respectively.
\end{lemma}
\begin{proof}
    Let $\vec{x}^{\max}\in\hat\Kk$, such that $\|v\|_{L^\infty(\hat\Kk)} = |v(\vec{x}^{\max})|$. We distinguish to two cases: $\vec{x}^{\max}\in[0,1]^d\subset\hat\Kk$, and $\vec{x}^{\max}\in\hat\Kk\setminus[0,1]^d$. For $\vec{x}^{\max}\in[0,1]^d$, we apply \eqref{LpFinv} in one dimension for the $d$-component of $\vec{x^{\max}}$, so that
    \begin{equation*}
        |v(\vec{x}^{\max})|^q \leq 2^{q+1} (2r^2 + 1)\int_0^1|v(x_1^{\max},\dotso,x_{d-1}^{\max},x_d)|^q\d x_d.
    \end{equation*}
    Repeating the same process iteratively, for each component $x_j^{\max}$, we deduce
    \begin{equation*}
        \|v\|_{L^\infty(\hat\Kk)}^q = |v(\vec{x}^{\max})|^q \leq 2^{d(q+1)} (2r^2 + 1)^d\|v\|_{L^q([0,1]^d)}^q \leq 2^{d(q+1)} (2r^2 + 1)^d\|v\|_{L^q(\hat\Kk)}^q,
    \end{equation*}
    which completes the proof for the first case.

    If $\vec{x}^{\max}\in\hat\Kk\setminus[0,1]^d$, let $\hat\Kk^{\max}$ to be the pyramid with base $\hat{F}^0 = [0,1]^{d-1}$, and the remaining vertex being $\vec{x}^{\max}$. For $d=2$, we define $\hat{F}^{\max}$ to be the line segment with endpoints $(0,0)$ and $\vec{x}^{\max}$. 
    Thus, \eqref{LpFinv} for $d=1$ implies
    \begin{equation*}
        |v(\vec{x}^{\max})|^q \leq 2^{q+1}\frac{2r^2+1}{|\hat{F}^{\max}|}\|v\|_{L^q(\hat{F}^{\max})}^q.
    \end{equation*}
    At the same time, $\hat{F}^{\max}$ is a face of the two-dimensional simplex $\hat\Kk^{\max}$, and hence, we apply \eqref{LpFinv} in two dimensions, to obtain
    \begin{equation*}
        \|v\|_{L^q(\hat{F}^{\max})}^q \leq 2^{q+3}\frac{(2r^2+1)|\hat{F}^{\max}|}{2|\hat{\Kk}^{\max}|}\|v\|_{L^q(\hat\Kk^{\max})}^q.
    \end{equation*}
    The above two inequalities show 
    \begin{equation*}
        |v(\vec{x}^{\max})|^q \leq 4\cdot2^{2(q+1)}\frac{(2r^2+1)^2}{2|\hat{\Kk}^{\max}|}\|v\|_{L^q(\hat\Kk^{\max})}^q.
    \end{equation*}
    Now, simple calculations show that $|\hat{\Kk}^{\max}| = x_2^{\max}/2\geq1/2$, and hence, we combine the above inequality and the fact that $\hat\Kk^{\max}\subset\hat\Kk$, to show
    \begin{equation*}
        |v(\vec{x}^{\max})|^q \leq 4\cdot2^{2(q+1)}(2r^2+1)^2\|v\|_{L^q(\hat\Kk)}^q,
    \end{equation*}
    which is the desired estimate for $d=2$.

    For $d=3$, we set $\hat{S}_2^{\max}$ to be the triangle with vertices $(0,0,0)$, $(0,1,0)$ and $\vec{x}^{\max}$. Working exactly as in the two-dimensional case, we infer
    \begin{equation*}
        |v(\vec{x}^{\max})|^q \leq 4\cdot2^{2(q+1)}C_\inv^2\frac{(2r^2+1)^2}{2|\hat{S}_2^{\max}|}\|v\|_{L^q(\hat{S}_2^{\max})}^q.
    \end{equation*}
    We now consider the tetrahedron $\hat{S}_3^{\max}\subset\hat\Kk^{\max}$, with vertices $(0,0,0)$, $(0,1,0)$, $(1,0,0)$ and $\vec{x}^{\max}$. We observe that $\hat{S}_3^{\max}$ is a three-dimensional simplex and $\hat{S}_2^{\max}$ is one of its faces. Thus, we apply \eqref{LpFinv} in three dimensions, to obtain
    \begin{equation*}
        \|v\|_{L^q(\hat{S}_2^{\max})}^q \leq 8\cdot2^{q+1}\frac{(2r^2+1)|\hat{S}_2^{\max}|}{3|\hat{S}_3^{\max}|}\|v\|_{L^q(\hat{S}_3^{\max})}^q \leq 8\cdot2^{q+1}\frac{(2r^2+1)|\hat{S}_2^{\max}|}{3|\hat{S}_3^{\max}|}\|v\|_{L^q(\hat\Kk^{\max})}^q.
    \end{equation*}
    The above two inequalities show
    \begin{equation*}
        |v(\vec{x}^{\max})|^q \leq 32\cdot2^{3(q+1)}\frac{(2r^2+1)^3}{6|\hat{S}_3^{\max}|}\|v\|_{L^q(\hat\Kk^{\max})}^q.
    \end{equation*}
    Again, some straightforward calculations show that $|\hat{S}_3^{\max}|=x_3^{\max}/6\geq1/6$, and hence, under the observation that $\hat\Kk^{\max}\subset\hat\Kk$, we have
    \begin{equation*}
        \|v\|_{L^\infty(\hat\Kk)}^q = |v(\vec{x}^{\max})|^q \leq 32\cdot2^{3(q+1)}(2r^2+1)^3\|v\|_{L^q(\hat\Kk)}^q.
    \end{equation*}
    The last inequality completes the proof.
\end{proof}

\subsection{Stability under domain perturbations.}\label{DGpoly_perturb}
We will also need an $L^q$-stability estimate under perturbations of a generalized prism $\hat\Kk$; this is an extension of \cite[Lemma 4.14]{CanDongGeo2021} for $q\neq2$.

\begin{lemma}\label{th_pert_stab}
    Let $\hat\Kk$ be a reference generalized prism, $\hat\Kk_\epsilon=\hat\Kk\cap(\hat\Kk - \epsilon\vec{e}_d)$, where $V - a = \{v - a: \text{ } v\in V\}$, for any set $V$, $a\in V$, and $\vec{e}_d = (0,\dotso,0,1)$. Then, for all $r\in\mathbb{N}_0$, $v\in\Pa_r(\hat\Kk)$ and $0<\epsilon \leq (2^{q+2}(2r^2+1))^{-1}$, we have
    \begin{equation}\label{pert_stab}
        \frac{1}{2}\|v\|_{L^q(\hat\Kk)}^q \leq \|v\|_{L^q(\hat\Kk_\epsilon)}^q.
    \end{equation}
\end{lemma}
\begin{proof}
Set $\ell_{{\bf x},\epsilon}:=\hat{\Kk}\cap (\ell_{{\bf x}}-\epsilon {\rm e}_d) $. Then, Fubini's Theorem and \eqref{1Dinv} for $d=1$, give, respectively,
			\begin{equation}\label{strip_small_gen}
				\begin{aligned}
					\|v\|_{L^q(\hat{\Kk}\backslash \hat{\Kk}_\epsilon)}^q 
					\le 
					\epsilon\int_{\hat{F}^0}\|v\|_{L^\infty(\ell_{{\bf x}}\backslash \ell_{{\bf x},\epsilon})} ^q\d {\bf x}
					\le 
					\epsilon\int_{\hat{F}^0}\|v\|_{L^\infty(\ell_{{\bf x}})}^q\d {\bf x} \le 2^{(q+1)} (2r^2 + 1)\|v\|_{L^q(\hat\Kk)}^q,
				\end{aligned}
			\end{equation}
			 since the length of $\ell_{{\bf x}}$ is bounded from below by one. Then, for the range of $\epsilon$ in the statement, the result follows upon observing that $\|v\|_{L^q(\hat{\Kk})}^q -	\|v\|_{L^q(\hat{\Kk}_\epsilon)}^2=
			\|v\|_{L^q(\hat{\Kk}\backslash \hat{\Kk}_\epsilon)}^q $.
\end{proof}

\subsection{Trace-inverse estimates on general elements.}\label{DGpoly_geninv}
It is possible that the constant in the inverse estimate in Lemma \ref{th_LpFinv_poly} becomes arbitrarily large if $\vec{m}_i\cdot \vec{n}\to 0$ locally. Very small $\vec{m}_i\cdot \vec{n}>0$ can occur in cases of locally highly variable element boundaries. As we shall see below, this translates in large discontinuity-penalization parameter locally, which in turn may result into to conditioning issues, or even loss of approximation.

To address this, we introduce the notion of $(q,r)$-coverability that generalizes the concept of $p$-coverability introduced in \cite{CanDongGeoHoustSB}: $p$-coverability in \cite{CanDongGeoHoustSB} corresponds to $(2,r)$-coverability herein.

\begin{defin}
    An element $K\in\Ta$ is said to be $(q,r)$-coverable, if there exists a family $\{\hat\Kk_j\}_{j=1}^{N_K}$ of reference generalized prisms and corresponding affine maps $\vec{\Phi}_j$, such that the mapped generalized prisms $\overline{K}_j = \vec{\Phi}_j(\hat\Kk_j)$, $j=1,\dotso,N_K$, form a, possibly overlapping, covering of K, such that
    \begin{equation}\label{pcov}
        \operatorname{dist}(\partial\overline{K}_j, K) \leq \left(2^{q+2}(2r^2 + 1)\right)^{-1}\hh_{\overline{K}_j} \text{ and} \quad  |\overline{K}_j| \geq c_{as} |K|,
    \end{equation}
    for all $j=1,\dotso,N_K$, where $\hh_{\overline{K}_j} = \sup_{\vec{x}\in\hat{F}^0}|\vec{\Phi}_j(\ell_{\vec{x},j})|$, $c_{as}$ is a positive constant, independent of $K$ and $\Ta$, $\operatorname{dist}(\partial\overline{K}_j,K)=\sup_{\vec{x}\in\partial\overline{K}_j}\inf_{\vec{y}\in K}|\vec{x} - \vec{y}|$, and $\ell_{\vec{x},j} = \hat\Kk_j\cap\{\vec{x}+\alpha\vec{e}_d: \text{ } \alpha\in\mathbb{R}\}$.
\end{defin}

Thus, if $K$ is $(q,r)$-coverable, then \eqref{pcov} constructs a covering of mapped generalized prisms $\overline{K}_j$ and respective sub-prisms $\underline{K}_j=\overline{K}_{j,\epsilon}$, with $\overline{K}_{j,\epsilon}$ defined as in Lemma \ref{th_pert_stab} for $0<\epsilon\leq (2^{q+2}(2r^2 + 1))^{-1}\hh_{\overline{K}_j}$, such that $\underline{K}_j\subset K$. Then, Lemma \ref{th_pert_stab} shows for any $v\in\Pa_r(K)$, that 
\begin{equation}\label{rgp_to_K}
    \frac{1}{2}\|v\|_{L^q(\overline{K}_j)}^q \leq \|v\|_{L^q(\underline{K}_j)}^q \leq \|v\|_{L^q(K)}^q.
\end{equation}
This observation is crucial for controlling the trace-inverse estimate constant and, subsequently, preventing the possibility of over-penalization.

We are now in position to derive an inverse estimate for elements that satisfy Assumption \ref{starshaped}.

\begin{lemma}\label{th_LpFinv_poly_gen}
    Let $K\in\Ta$ satisfy Assumption \ref{starshaped} and let $\{F_i\}_{i=1}^{n_K}$ be the respective subdivision of $\partial K$. Then, for all $q\in(0,\infty)$, $r\in\mathbb{N}_0$, $v\in\Pa_r(K)$ and $i=1,\dotso,n_K$, we have
    \begin{equation}\label{LpFinv_poly_gen}
        \int_{F_i}|v|^q \d s \leq \Ca_\INV(q,r,K,F_i)\frac{(2r^2+1)|F_i|}{|K|}\int_K|v|^q \d \vec{x},
    \end{equation}
    where 
    \begin{equation}\label{C_inv}
        \Ca_\INV(q,r,K,F_i) := \begin{cases}
            \min\{\Ca_\mathrm{reg}(q,K,F_i), \Ca_\mathrm{cov}(q,r)\}, & K \text{ is } (q,r)-\text{coverable}; \\
            \Ca_\mathrm{reg}(q,K,F_i), & \text{otherwise},
        \end{cases}
    \end{equation}
    $\Ca_\mathrm{reg}(q,K,F_i) = 2^{q+1}C_{\inv,d}|K|/(|F_i|\min_{F_i}(\vec{m}_i\cdot\vec{n}))$ and $\Ca_\mathrm{cov}(q,r) = 2^{d(q+1)+1}C_{\infty,d}c_{as}^{-1}(2r^2 + 1)^{d-1}$.
\end{lemma}
\begin{proof}
    Whether or not $K$ is $(q,r)$-coverable, Lemma \ref{th_LpFinv_poly} implies 
    \begin{equation}\label{not_pcov}
        \|v\|_{L^q(F_i)}^q \leq 2^{q+1}C_{\inv,d}\frac{2r^2+1}{\min_{F_i}(\vec{m}_i\cdot\vec{n})}\|v\|_{L^q(K_{F_i})}^q \leq \Ca_\mathrm{reg}(q,K,F_i)\frac{(2r^2+1)|F_i|}{|K|}\|v\|_{L^q(K)}^q.
    \end{equation}
    On the other hand, if $K$ is $(q,r)$-coverable, then $K_{F_i}\subset K \subset \cup_{j=1}^{N_K}\overline{K}_j$, and hence,
    \begin{equation}\label{yes_pcov1}
        \|v\|_{L^q(F_i)}^q \leq |F_i|\|v\|_{L^\infty(F_i)}^q \leq |F_i|\|v\|_{L^\infty(K_{F_i})}^q \leq |F_i|\max_{1\leq j \leq N_K}\|v\|_{L^\infty(\overline{K}_j)}^q.
    \end{equation}
    Now, for $j=1,\dotso,N_K$, we apply the $L^\infty\to L^q$ inverse estimate \eqref{LinfLq} together with a scaling argument, and then \eqref{rgp_to_K} together with \eqref{pcov}, to obtain
    \begin{equation}\label{yes_pcov2}
        \|v\|_{L^\infty(\overline{K}_j)}^q \leq 2^{d(q+1)}C_{\infty,d}\frac{(2r^2 + 1)^d}{|\overline{K}_j|}\|v\|_{L^q(\overline{K}_j)}^q \leq 2^{d(q+1)+1}C_{\infty,d}\frac{(2r^2 + 1)^d}{c_{as}|K|}\|v\|_{L^q(K)}^q
    \end{equation}
    The combination of \eqref{yes_pcov1} and \eqref{yes_pcov2} implies
    \begin{equation}\label{yes_pcov}
        \|v\|_{L^q(F_i)}^q \leq \Ca_\mathrm{cov}(q,r)\frac{(2r^2+1)|F_i|}{|K|}\|v\|_{L^q(K)}^q,
    \end{equation}
    and the proof is complete after taking the minimum of the right-hand sides of \eqref{not_pcov} and \eqref{yes_pcov}.
\end{proof}

Note that the constant in \eqref{LpFinv_poly_gen} depends on the points $\vec{x}_i^0$ corresponding to each $F_i\subset\partial K$, through $\vec{m}_i$ and it can be further reduced by optimizing the location of $\vec{x}_i^0$ for a given partition $\{F_i\}_{i=1}^{n_K}$, in accordance to Assumption \ref{starshaped}. For certain shapes, e.g., a ball, we can compute the constant explicitly. In practice, we typically consider convenient $\vec{x}_i^0$'s without seeking to perform optimization. 

For the analysis of the robust IPDG method below, 
we require an extension of Lemma \ref{th_QN_invest} for arbitrarily-shaped elements $K$. An inspection of the proof of Lemma \ref{th_QN_invest} reveals that replacing the trace-inverse estimate used in \eqref{invweight} by \eqref{LpFinv_poly_gen} results into the following quasi-norm trace inverse estimate for general element shapes.
\begin{lemma}\label{th_QN_invest_poly}
    Let $p\in(2,\infty)\cap\mathbb{Q}$, and let $r\in\mathbb{N}$. Then, for $w,v\in\Pa_r(K)$, we have
    \begin{equation}\label{QN_invest_poly}
        \int_{F_i}\left( |\nabla w| + |\nabla v| \right)^{p-2}|\nabla v|^2\d s \leq \Ca_\mathrm{Q}^p(r,K,F_i)\frac{(2(r-1)^2+1)|F_i|}{|K|}\int_K\left( |\nabla w| + |\nabla v| \right)^{p-2}|\nabla v|^2\d \vec{x},
    \end{equation}
    where
     $
        \Ca_\mathrm{Q}^p(r,K,F_i) := 2^{p/2-1}\Ca_\INV(1/k_p,pk_p(r-1),K,F_i)$.
\qed
\end{lemma}

\subsection{Polynomial approximation on curved polytopic elements.}\label{DG_poly_approx}
We consider $hp$-version best approximation error estimates that will be used to extract convergence rates for \eqref{IPDG} for essentially arbitrarily-shaped elements. The setting echoes and extends that of \cite{{CanDongGeoHoustSB},CanDongGeo2021}.

\begin{defin}\label{def_cover}
    We consider a, possibly overlapping, covering $\Ta^{cov}$ of $\Ta$, consisting of open shape-regular simplices, such that for all $K\in\Ta$, there exists a simplex $\Ka\in\Ta^{cov}$, such that $K\subset\Ka$. We also define the covering domain $\Omega^{cov}\supset\Omega$ via $\overline\Omega^{cov}=\cup_{\Ka\in\Ta^{cov}}\overline{\Ka}$.
\end{defin}

From now on, we impose the following assumption on the mesh $\Ta$, for the extraction of $hp$-version polynomial approximation results.

\begin{assumption}\label{assumpt_overlap}
    We assume that there exists a covering $\Ta^{cov}$ of $\Ta$ as per Definition \ref{def_cover}, and a constant $M_\Omega^{cov}\in\mathbb{N}$, such that 
    \begin{equation*}
        \max\operatorname{card}\{K'\in\Ta: K'\cap\Ka\neq\varnothing \text{ } \forall \Ka\in\Ta^{cov} \text{ such that } K \subset \Ka \} \leq M_\Omega^{cov}.
    \end{equation*}
    We further assume that $h_\Ka=\operatorname{diam}(\Ka) \leq C_\mathrm{diam}h_K$ for all $K\in\Ta$, $\Ka\in\Ta^{cov}$ with $K\subset\Ka$. Here, $C_\mathrm{diam}$ is assumed to be independent of the mesh parameters.
\end{assumption}

The purpose of Assumption \ref{assumpt_overlap} is to facilitate the use of  $hp$-version best approximation errors into general domains, viz., employ the simplices $\Ka\supset K$, $K\in\Ta$, using bounds like \eqref{proj_elem}.
However, if $K\in\Ta$ is close enough to the boundary $\partial\Omega$, part of the corresponding element $\Ka\supset K$ will lie outside of the domain $\Omega$. For this reason, an extension of $v$ in $\Omega^{cov}$ will be necessary. 

\begin{theorem}[\cite{Stein1970}]\label{th_ext}
    Let $\Omega\subset\mathbb{R}^d$ be a domain with Lipschitz boundary. Then, there exists a bounded linear extension operator $\Ea:W^{s,p}(\Omega)\to W^{s,p}(\mathbb{R}^d)$, $s\in\mathbb{N}_0$, $p\in[2,\infty)$, such that 
       $ \Ea v|_\Omega = v$, and   
        $\|\Ea v\|_{W^{s,p}(\mathbb{R}^d)} \leq C_\Ea\|v\|_{W^{s,p}(\Omega)}$,
    where $C_\Ea$ depends only on $s$, $p$ and $\Omega$.
\end{theorem}

We are now in position to state the relevant polynomial approximation results.

\begin{lemma}\label{th_polyap_polyg}
    Let $K\in\Ta$ satisfy Assumption \ref{starshaped} and let the covering simplex $\Ka\in\Ta^{cov}$ with $K\subset\Ka$ satisfy Assumption \ref{assumpt_best_app}. Let $v\in W^{\vec{s},p}(\Omega;\Ta)$, for some $\vec{s}=(s_K: K\in\Ta)$, $2 \leq s_K \leq r_K+1$, and suppose that $\Ea v|_\Ka\in W^{s_K,p}(\Ka)$. Then, there exists approximation $\Pi_{r_K}v\in\Pa_{r_K}(K)$, such that 
    \begin{equation}\label{hp_polyap_polyg}
        \|v - \Pi_{r_K}v\|_{W^{m,p}(K)} \leq C\frac{h_K^{s_K-m}}{r_K^{s_K-m-\beta_m}}\|\Ea v\|_{W^{s_K,p}(\Ka)}, \text{ } 0 \leq m \leq 2 \leq s \leq r_K + 1,
    \end{equation}
    \begin{equation}\label{hp_LpF}
        \|v - \Pi_{r_K}v\|_{L^p(F_i)} \leq \Ca_{ap}^{L^p}(r_K,K,F_i)\frac{h_K^{s_K-1/p}}{r_K^{s_K-1/p}}\|\Ea v\|_{W^{s_K,p}(\Ka)},
    \end{equation}
    and
    \begin{equation}\label{hp_W1pF}
        \|\nabla(v - \Pi_{r_K}v)\|_{L^p(F_i)} \leq \Ca_{ap}^{W^{1,p}}(r_K,K,F_i)\frac{h_K^{s_K-1-1/p}}{r_K^{s_K-1-1/p}}\|\Ea v\|_{W^{s_K,p}(\Ka)},
    \end{equation}
    with 
    $
        \Ca_{ap}^{L^p}(r_K,K,F_i) = C\min\big\{ \big(\min_{F_i}(\vec{m}_i\cdot\vec{n})\big)^{-1/p}r_K^{\beta_0+(\beta_1-\beta_0)/p}h_K^{1/p}, |F_i|^{1/p}r_K^{d/2-1/p}h_K^{(1-d)/p} \big\},
    $
    and  \\  
    $
        \Ca_{ap}^{W^{1,p}}(r_K,K,F_i) = C\min\big\{ \big(\min_{F_i}(\vec{m}_i\cdot\vec{n})\big)^{-1/p}r_K^{\beta_1+(\beta_2-\beta_1)/p}h_K^{1/p}, |F_i|^{1/p}r_K^{d-1/p}h_K^{(1-d)/p} \big\};
    $
    here, $F_i\subset\partial K$ are as in Assumption \ref{starshaped}, $\beta_0$, $\beta_1$ and $\beta_2$ are the exponents from Assumption \ref{assumpt_best_app} and $C$ is a constant, depending only on $s_K$, $p$, $\Omega$, $C_\mathrm{diam}$ and the shape-regularity constant of $\Ka$.
\end{lemma}
\begin{proof}
    The proof follows analogous steps to the proof of the respective approximation estimates in Hilbertian norms, derived in \cite[Lemma 4.31]{CanDongGeo2021}. The difference is the use Lemma \ref{th_multitr} and Assumption \ref{assumpt_best_app} (or Lemma \ref{th_hp_polyap}) instead of the respective Hilbertian counterparts.
\end{proof}

\subsection{Discontinuity-penalization parameter.}\label{DGpoly_staberr}
The monotonicity, boundedness and coercivity conditions for the robust IPDG method \eqref{IPDG} on meshes comprising essentially arbitrarily-shaped elements are shown in a completely analogous fashion to their counterparts in Section \ref{DG_monot_cont} for simplicial meshes and are, thus, omitted for brevity. The choice of the discontinuity-penalization parameter for which these properties hold is of practical importance and is discussed in detail.

\begin{defin}\label{IFK}
        We define the index set
        $I_F^K := \{1\leq i \leq n_K: \text{ } |\mathring{F} \cap F_i^K| \neq 0\}\subset\{1,\dotso,n_K\}$, for each interface $F\in\Fa$ of an element $K\in\Ta_F$, for each of its subdivisions $\{F_i^K\}_{i=1}^{n_K}$ of $\partial K$, described in Assumption \ref{starshaped}.
\end{defin}

 In other words,   for $F\in\Fa$, $K\in\Ta_F$, $I_F^K$ is the index set of the relevant $F_i^K\subset\partial K$ to interface $F$ only, i.e.,  $F \subset \cup_{i \in I_F^K} F_i^K$. In addition, $I_F^K$ defines the indices of a mutually exclusive partition $\{F \cap F_i^K\}_{i \in I_F^K}$ of $F$, that is $F = \cup_{i \in I_F^K} (F \cap F_i^K)$.

\begin{lemma}\label{th_DGmonot_poly}
    Let $\varepsilon>0$, $p\in(2,\infty)\cap\mathbb{Q}$. For $F\in\Fa$, we define
    \begin{equation*}
        \Gg(K,F_i^K) := \Ca_\mathrm{Q}^p(r_K,K,F_i^K)\frac{(2 (r_K-1)^2 + 1)|F_i^K|}{|K|}, \text{ } i\in I_F^K, \text{ } K\in\{K^+,K^-\},
    \end{equation*}
    where $\Ca_\mathrm{Q}^p(r_K,K,F_i^K)$ is the constant in \eqref{QN_invest_poly}. We also consider the piecewise constant function $\zeta_F^\pm:F\to\mathbb{R}$, defined by
    \begin{equation}\label{defzeta_poly}
      \zeta_F^\pm|_{F_i^{K^\pm}\cap F} = \zeta_{F_i^{K^\pm}} := \varepsilon/(m_{K^\pm}\mu_\varepsilon \operatorname{card}(I_F^{K^\pm}) \Gg(K^\pm,F_i^{K^\pm})),\text{ } i\in I_F^{K^\pm},
    \end{equation}
    where $\varepsilon$ is defined by \eqref{epsilon}. We select
    \begin{equation}\label{sigma_poly}
        w_F^\pm = \zeta_F^\pm/(\zeta_F^+ + \zeta_F^-) \text { and } \sigma|_F = 1/(\zeta_F^+ + \zeta_F^-).
    \end{equation} 
    Then, for $u_h,v_h\in S_\Ta^\vec{r}$ and $u\in W^{1,p}_0(\Omega)\cap W^{1+1/p,p}(\Omega;\Ta)$, the monotonicity condition \eqref{DGmonot} holds.
    Moreover, the boundedness condition \eqref{DGcont} holds. Furthermore for all $p\in [2,\infty)$ the coercivity condition \eqref{coer} also holds.
\end{lemma}

\begin{remark}\label{rem_const_pen_weight}
   The penalty \eqref{sigma_poly} and of the weights $\vec{w}$ are constant on each $F\cap F_i^K\cap F_j^{K'}$, for $i\in I_F^K$, $j\in I_F^{K'}$, for all $K,K'\in\Ta_F$, $F\in\Fa$, whenever $|F\cap F_i^K\cap F_j^{K'}|\neq0$. That is both the penalty and the weights are, in general, piecewise constant on each interface.
\end{remark}
Also, Remark \ref{rem_good_penalty} is still valid in the case of general meshes in that the harmonic average in the weights allows to select the smallest of the two contributions as penalty.

\subsection{Stability and \emph{a priori} error bounds.}\label{DGpoly_stab_besterr} 
To prove stability, a broken Poincar\'{e} inequality of the form \eqref{poinc} is required. For straight-faced polytopic elements that are star-shaped with respect to a ball \eqref{poinc} has been recently proven in \cite{BoMasc2025}; this is a stronger requirement than Assumption \ref{starshaped}. Current work, \cite{CanDonGeoPoincare}, is concerned with broken Poincar\'{e} inequalities for meshes comprising curved elements satisfying Assumption \ref{starshaped}, for functions in $W^{1,p}_0(\Omega)\cap W^{2,p/(p-1)}(\Omega)$. 
As this is an active area of research, we prefer to state the stability and \emph{a priori} error bounds \emph{assuming} the validity of a broken Poincar\'{e} inequality of the form \eqref{poinc} for the meshes considered. The proofs of the following results remains essentially identical to the case of simplicial element meshes proven above.

\begin{theorem}\label{th_stab_poly}
     Theorem \ref{th_stab} holds for all (polytopic) meshes for which \eqref{poinc} is valid.\qed
\end{theorem}

\begin{theorem}\label{th_ba_poly} 
    Theorem \ref{th_best_app} holds of all meshes satisfying Assumption \ref{starshaped}.\qed
\end{theorem}

Theorems \ref{th_stab_poly} and \ref{th_ba_poly} imply the following \emph{a priori} error bound for the quasi-norm and the broken Sobolev norm.

\begin{theorem}\label{th_rates_poly}
     Suppose that Assumptions \ref{assumpt_best_app}, \ref{starshaped} and \ref{assumpt_overlap} are satisfied. Assume further that the extension $\Ea u$ of the exact solution $u$ satisfies $\Ea u|_\Ka\in W^{s_K,p}(\Ka)$, 
     $\Ka\in\Ta^{cov}$
     . Then, there exists a constant $C>0$, depending only on $\vec{s},p,\Omega,C_\mathrm{diam}$ and the shape-regularity of $\Ta^{cov}$, such that 
    \begin{align}\label{rates_poly}
        \trn u_h - u \trn_{(\Ta;p,u_h)}^2 + \trn u_h - u \trn_p^p \leq CZ\Bigg(  \sum_{K\in\Ta}\Da(K)\frac{h_K^{p(s_K-1)}}{r_K^{p(s_K-1)}}\|\Ea u\|_{W^{s_K,p}(\Ka)}^p \Bigg)
        ,
    \end{align}
    where
    \begin{align*}
        \Da(K) = r_K^{p\beta_1} + \sum_{F\in\Fa: \text{ } F\subset\partial K}\sum_{i\in I_F^K} \Bigg( \left( \frac{\Da_i(K)r_Kh_K}{|K|} \right)^{p-1}(\Ca_{ap}^{L^p}(r_K,K,F_i^K))^p& \\
        \text{  } + \frac{|K|}{\Da_i(K)r_Kh_K} (\Ca_{ap}^{W^{1,p}}(r_K,K,F_i^K))^p\Bigg),
    \end{align*}
   with $\Da_i(K) := m_K\operatorname{card}(I_F^K)\Ca_\INV(r_K,K,F_i^K)|F_i^K|$, and $\Ca_{ap}^{L^p}(r_K,K,F_i^K)$, $\Ca_{ap}^{W^{1,p}}(r_K,K,F_i^K)$ as in Lemma \ref{th_polyap_polyg}.
\end{theorem}
\begin{proof} Let $v_h\in S_\Ta^\vec{r}$ and $\eta\equiv \eta(v_h) := v_h-u$.
  We first show that there exists a constant $C_p$, depending on $p$, such that 
    \begin{align}\label{best_app_high_poly}
        &\trn u_h - u \trn_{(\Ta;p,u_h)}^2 + \trn u_h - u \trn_p^p \notag \\
        &\text{ } \leq C_p\inf_{v_h\in S_\Ta^\vec{r}}\Bigg[ Z\Big( \|\nabla_h \eta \|_{L^p(\Omega)}^p\Big)+ Z\bigg(\sum_{K\in\Ta}\sum_{F\in\Fa: \text{ } F\subset\partial K}\sum_{i\in F_i^K}\left(\frac{\Da_i(K)r_K^2}{|K|}\right)^{p-1}\|\eta\|_{L^p(F\cap F_i^K)}^p\bigg) 
        \notag\\
        &\qquad\qquad\qquad  + Z\bigg(\sum_{K\in\Ta}\sum_{F\in\Fa: \text{ } F\subset\partial K}\sum_{i\in F_i^K}\frac{|K|}{\Da_i(K)r_K^2}\|\nabla\eta\|_{L^p(F\cap F_i^K)}^p\bigg)\Bigg].
    \end{align}

    It suffices to show that the right-hand side of \eqref {best_app} (corresponding to Theorem \ref{th_ba_poly}) is bounded by the right-hand side of \eqref{best_app_high_poly}.  First, we note that \eqref{quasi_to_norm_Omega} holds in the current setting also. 
    
    As before, we set
       $ \trn \eta \trn_{(\Ta;p,u_h)}^2 
        = \Ia_1 + \Ia_2$, with $\Ia_1$ and $\Ia_2$ as in the proof of Theorem \ref{th_rates}.
        
    For $\Ia_1$, we work exactly as before to show
       $ \Ia_1 \leq C_pZ\big( \|\nabla_h \eta\|_{L^p(\Omega)}^p\big)$.
    For $\Ia_2$, we work exactly as in \eqref{I2first}. 
    The resulting term $\|(w_F^\pm \sigma^{-1})^\frac{1}{p}\nabla u_h^\pm\|_{L^p(F)}$, for $F\in\Fa$, $\pm\in\{+,-\}$, is further bounded by applying the inverse estimate \eqref{LpFinv_poly}, to deduce \eqref{bound_uh_face} for the present setting also.
    Moreover, due to the choice of weighted averages, 
 \[
        \sigma|_{F\cap F_i^K} \leq 1/\zeta_{F_i^K} \leq C_p\frac{\Da_i(K)r_K^2}{|K|}, \qquad i\in I_F^K,
        \]
for $F\in\Fa$, $K\in\Ta_F$. Thus, for the jump-term of $\Ia_2$, working as above, we deduce
    \begin{align}\label{jump_er_1_poly}
        &\int_\Gamma \sigma^{p-1}|\llbracket \eta \rrbracket|^p \d s  \leq 
        C_p \sum_{K\in\Ta}\sum_{F\in\Fa: \text{  } F\subset\partial K}\sum_{i\in I_F^K}\left( \frac{\Da_i(K)r_K^2}{|K|} \right)^{p-1}\| \eta \|_{L^p(F \cap F_i^K)}^p,
    \end{align}
    and, thus,
    \begin{align}\label{QNbound2_poly}
        \Ia_2 \leq C_pZ\bigg( \sum_{K\in\Ta}\sum_{F\in\Fa: \text{  } F\subset\partial K}\sum_{i\in I_F^K}\left( \frac{\Da_i(K)r_K^2}{|K|} \right)^{p-1}\| \eta \|_{L^p(F \cap F_i^K)}^p \bigg).
    \end{align}
    
    It remains to bound the last term in \eqref{best_app} in the current `polytopic' setting, the only difference with respect to \eqref{trace1} being the piecewise constant nature of the weights and of $\sigma$ along each interface $F$, arriving eventually at 
   \[     \int_\Gamma\sigma^{-1}\lbrac \left(|\nabla u_h| + |\nabla\eta| \right)^{p-2}|\nabla\eta|^2\rbrac_\vec{w}\d s \leq C_p Z\bigg(\sum_{K\in\Ta}\sum_{F\in\Fa: \text{ } F\subset\partial K}\sum_{i\in I_F^K}\frac{|K|}{\Da_i(K)r_K^2}\|\nabla\eta\|_{L^p(F \cap F_i^K)}^p\bigg).
   \]
  Combining the above bounds, we arrive at \eqref{best_app_high_poly}.
  The estimate \eqref{rates_poly} follows immediately from the substitution of the polynomial approximation estimates \eqref{hp_polyap_polyg}, \eqref{hp_LpF} and \eqref{hp_W1pF} into \eqref{best_app_high_poly}.
\end{proof}

The above theorem provides with error estimates for the robust IPDG method on a very general class of polytopic meshes. For an illustration of the rates of convergence, we provide a corollary under additional assumptions on the element shapes of a mesh $\Ta$.

\begin{corollary}\label{cor_apriori}
    With the hypotheses of Theorem \ref{th_rates_poly}, and additionally assuming that: 
    \begin{itemize}
        \item[(i)] there exists a constant $c_{sr}$, such that
    $
\min_{F_i^K}(\vec{m}_i\cdot\vec{n})\ge c_{sr}h_K
    $
    for all $i\in I_F^K$, $F\subset \partial K$ interfaces of $K\in\Ta$;
    \item[(ii)] there exists a positive constant $C_\mathrm{ssh}$, independent of the mesh $\Ta$ in the mesh family $\{\Ta\}$, such that $\max_{K\in\Ta}m_K \leq C_\mathrm{ssh}$, for all $K\in\Ta$;
    \item[(iii)] the mesh $\Ta$ is shape-regular,  and each interface $F\subset\partial K$, as well as the $F_i^K\subset \partial K$ from Assumption \ref{starshaped}, have comparable diameters with $K$, viz.:
\begin{equation*}
   |K|\sim h_K^d \text{ for all } K\in\Ta,\ \text{ and }\  \operatorname{diam}(F) \sim \operatorname{diam}(F_i^K) \sim h_K, \ \text{ and }\  |F| \sim |F_i^K| \sim h_K^{d-1}.
\end{equation*}
    \end{itemize}
    Then, we have the following error bound
\begin{equation}\label{rates_poly_shapereg}
    \trn u_h - u \trn_{(\Ta;p,u_h)}^2 + \trn u_h - u \trn_p^p \leq CZ\bigg( \sum_{K\in\Ta}\Ca_2^p(m_K,r_K)\frac{h_K^{p(s_K-1)}}{r_K^{p(s_K-1)}}\|\Ea u\|_{W^{s_K,p}(\Ka)}^p \bigg),
\end{equation}
with $C>0$ depending only on $\vec{s}$, $p$,$\Omega$, $C_\mathrm{diam}$, the shape-regularity constant of $\Ta^{cov}$, $C_{\mathrm{ssh}}$, and $c_{sr}$, where
\begin{equation*}
    \Ca_2(m_K,r_K) = \max\left\{r_K^{\beta_1}, m_K\min\{r_K^{\beta_0 + (\beta_1 - \beta_0)/p + 1/p'}, r_K^{d/2-d/p + 1/p'}\}, \min\{r_K^{\beta_1 + (\beta_2 - \beta_1 + 1)/p}, r_K^{d/p'+1/p}\}\right\}.
\end{equation*}
\end{corollary}

\begin{proof}
    We employ the additional assumptions to bound the right-hand side of \eqref{rates_poly}, starting with {\rm (ii)}, which immediately implies $ \max_{K\in\Ta}\max_{F\in\Fa: F\subset\partial K}\operatorname{card}(I_F^K) \leq C_\mathrm{ssh}$, for all $K\in\Ta$. 

    Also, {\rm (i)} and {\rm (iii)} imply 
    \begin{align*}
    \Ca_{ap}^{L^p}(r_K,K,F_i^K) 
&\leq C\min\{r_K^{\beta_0 + (\beta_1 - \beta_0)/p}, r_K^{d/2-d/p}\},\\
     \Ca_{ap}^{W^{1,p}}(r_K,K,F_i^K)
    &\leq C\min\{r_K^{\beta_1 + (\beta_2 - \beta_1)/p}, r_K^{d-1/p}\}.
\end{align*}
and, also, $ \Ca_\INV(r_K,K,F_i^K) 
\sim 1$, with `$\sim$' denoting proportionality with respect to unimportant constants for the argument. Thus, $\Da_i(K)r_Kh_K\sim m_Kr_K h_K^d\sim m_Kr_K |K|$. Therefore,
\begin{align*}
    \Da(K) &\le C \big(  r_K^{p\beta_1} + m_K^pr_K^{p-1}\min\{r_K^{p\beta_0 + \beta_0-\beta_1}, r_K^{dp/2 - d} \} + r_K^{-1} \min\{r_K^{p\beta_1 + \beta_2 - \beta_1}, r_K^{d(p-1)}\}\big),
\end{align*}
and the result follows by elementary manipulations.
\end{proof}

\begin{remark}
    Assumption {\rm (i)} in Corollary \ref{cor_apriori} is rather mild, asserting essentially that, for each element $K$, there exists a, possibly overlapping, partition of the $K_{F_i}$'s as per Assumption \ref{starshaped}, with each $K_{F_i}$ being star-shaped with respect to a ball of comparable size to the element diameter. 
\end{remark}

The error bound \eqref{rates_poly_shapereg} is optimal with respect to $h_K$ and slightly suboptimal with respect to $r_K$. Moreover, if optimal $hp$-version polynomial approximation results in $W^{k,p}$ are to be provided (whence $\beta_0=\beta_1=\beta_2=0$),  as in the case of simplicial meshes above, then the suboptimality of our bound in $r_K$ reduces to $1/p'$. This showcases the robustness of the analysis, even on very general mesh concepts. Furthermore, the rates in $r_K$ are the same as in the case of simplicial meshes above for most cases: there is a slight difference on the rates when $p>6$, $d = 3$ and $s_K = 2$ are simultaneously satisfied, due to the different framework that was considered for polynomial approximation estimates.

\section{Numerical experiments}\label{DG_numerics}
We have implemented the IPDG method \eqref{IPDG} for $d=2$, using the finite element software NetGen/NGSolve \cite{schoeberl}. Our aim is to show the super-linear order of convergence of the method for regular solutions when employing $r\ge 2$. To that end, we consider shape-regular and quasi-uniform simplicial meshes, and uniform local polynomial degrees $r$. As such, the weights of the averages in the interior skeleton can be selected as $w_F^+ = w_F^- = 1/2$. For the solution of the underlying nonlinear systems, we employ a Newton--Raphson iteration. The process begins with the Poisson problem ($q=2$), and the $q$-Laplacian problems are then solved iteratively by updating $q \gets q + 0.5$ at each step, until the final computation with $q = p$, so that the initial guesses for each step are chosen as the IPDG solutions of the corresponding $(p - 0.5)$-Laplacian problems.  We consider the following examples.

\begin{example}\label{ex_sym} 
    Let $\Omega = (0,1)^2$ and $u:\Omega\to\mathbb{R}$ given by
    \begin{equation}\label{xpml1}
        u(x_1,x_2) := x_1x_2(1 - x_1)(1 - x_2)\sin(2\pi x_1x_2), \text{ } (x_1,x_2)\in[0,1]^2;
    \end{equation}
    the right-hand side $f$ is chosen so that $u$ is the solution to \eqref{BVP}.
\end{example}

\begin{example}\label{ex_tanh} 
   Let $\Omega = (-1,1)^2$$u:[-1,1]^2\to\mathbb{R}$ by
    \begin{equation}\label{xpml2}
        u(x_1,x_2) := -(1 - x_1^2)(1 - x_2^2)\tanh\left(50\left((x_1 - 1/2)^2 + x_2^2 - 0.01\right)\right), \text{ } (x_1,x_2)\in[-1,1]^2;
    \end{equation}
     the right-hand side $f$ is chosen so that $u$ is the solution to \eqref{BVP}.
\end{example}

\subsection{$h$-Version.}\label{DG_hversion}
First, we consider the $h$-version of \eqref{IPDG} for both examples. We select fixed polynomial degrees $r=1,2$ and we employ a sequence $\{\Ta_j\}_{j=0}^4$ of meshes, with $h_j = \max_{K\in\Ta_j}h_K = 0.2/2^j$, $j=0,\dotso,4$. The errors with respect to the norms $\trn\cdot\trn_p$ and the quasi-norms $\trn\cdot\trn_{(\Ta_j;p,u_h)}$ are presented in logarithmic scale, together with their least-square approximations in Figures \ref{fig:hVersion_sym} and \ref{fig:hVersion_tanh}, for $p=2.5,4,4.5$.

\begin{figure}
    \centering
    \subfloat[\centering]{{\includegraphics[width=7.6cm]{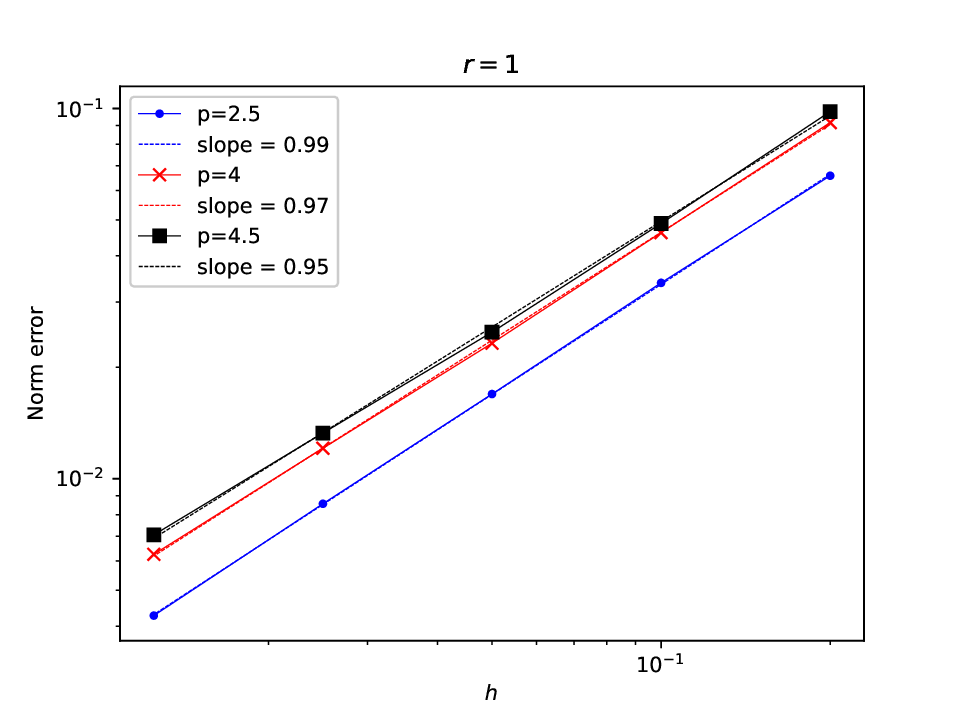} }}
    \hspace{0.00001cm}
    \subfloat[\centering]{{\includegraphics[width=7.6cm]{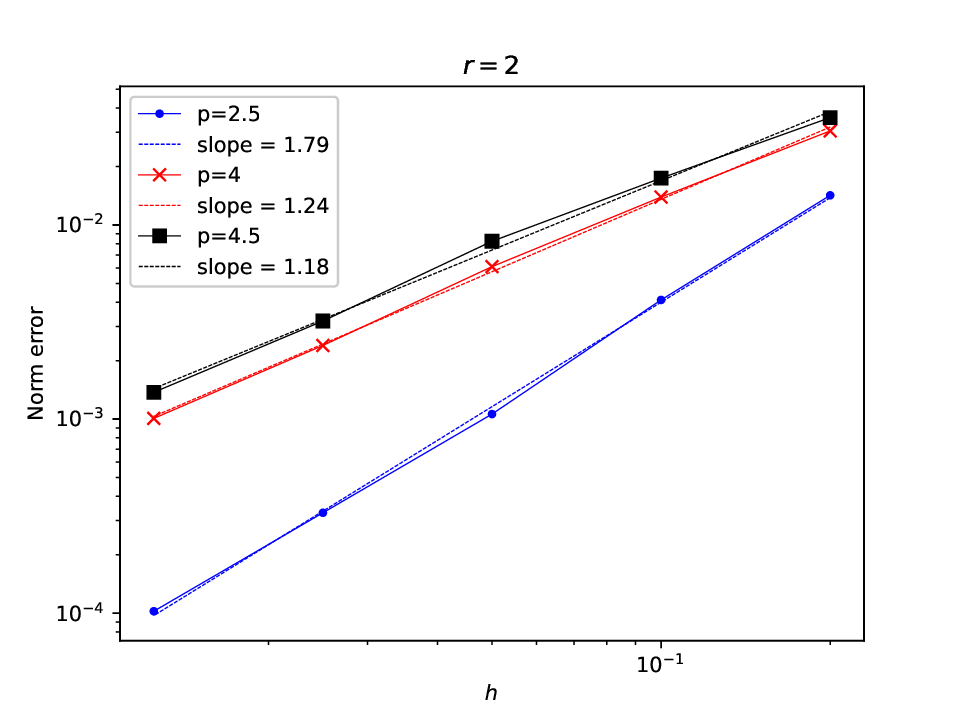} }}
     \hspace{0.00001cm}
    \subfloat[\centering]{{\includegraphics[width=7.6cm]{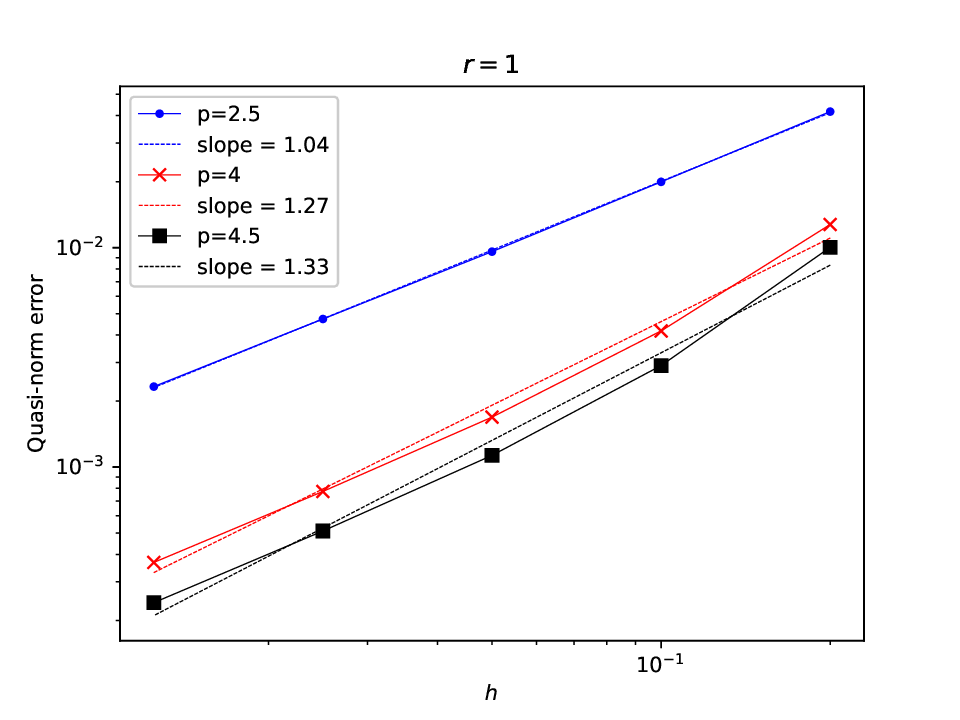} }}
    \hspace{0.00001cm}
    \subfloat[\centering]{{\includegraphics[width=7.6cm]{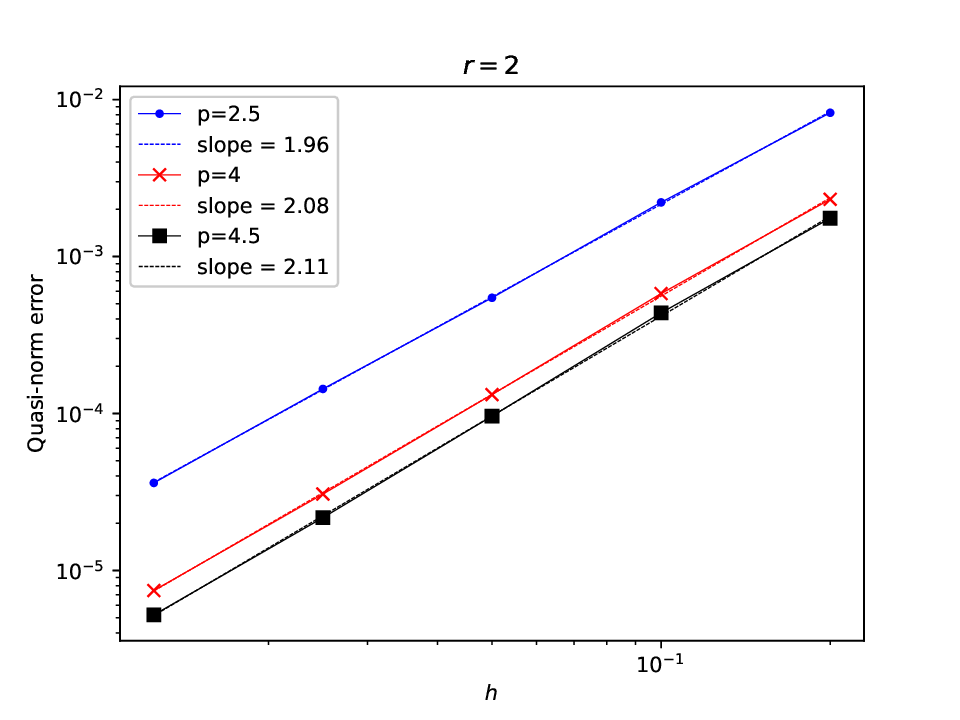} }}
    \caption{Example \ref{ex_sym}, $h$-version: (a) Norm errors for $r=1$. (b) Norm errors for $r=2$. (c) Quasi-norm errors for $r=1$. (d) Quasi-norm errors for $r=2$.}
    \label{fig:hVersion_sym}
\end{figure}

\begin{figure}
    \centering
    \subfloat[\centering]{{\includegraphics[width=7.6cm]{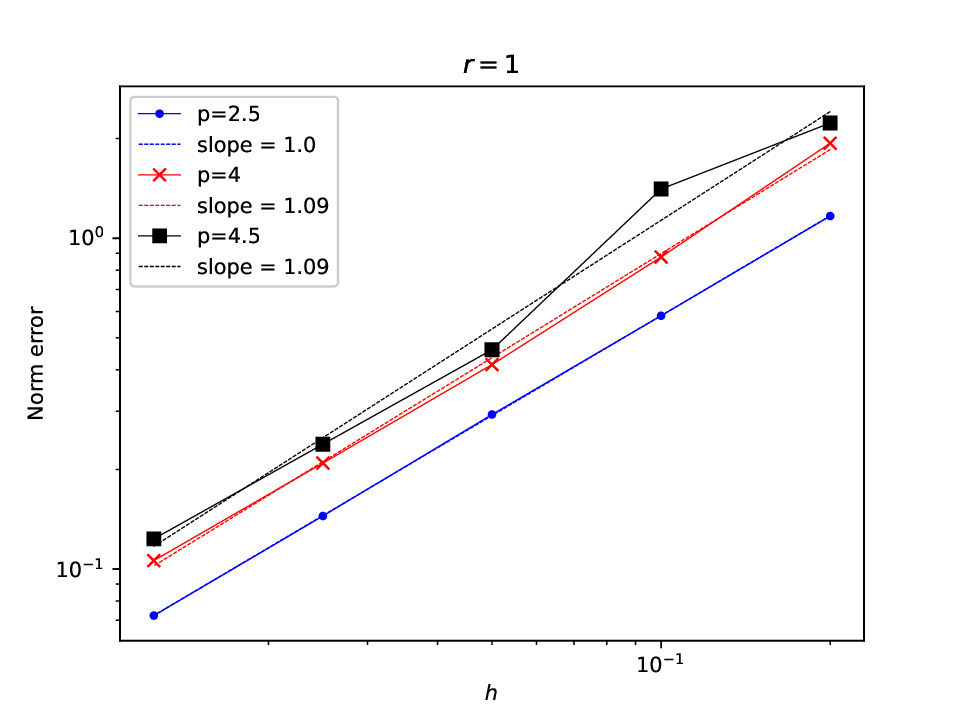} }}
    \hspace{0.00001cm}
    \subfloat[\centering]{{\includegraphics[width=7.6cm]{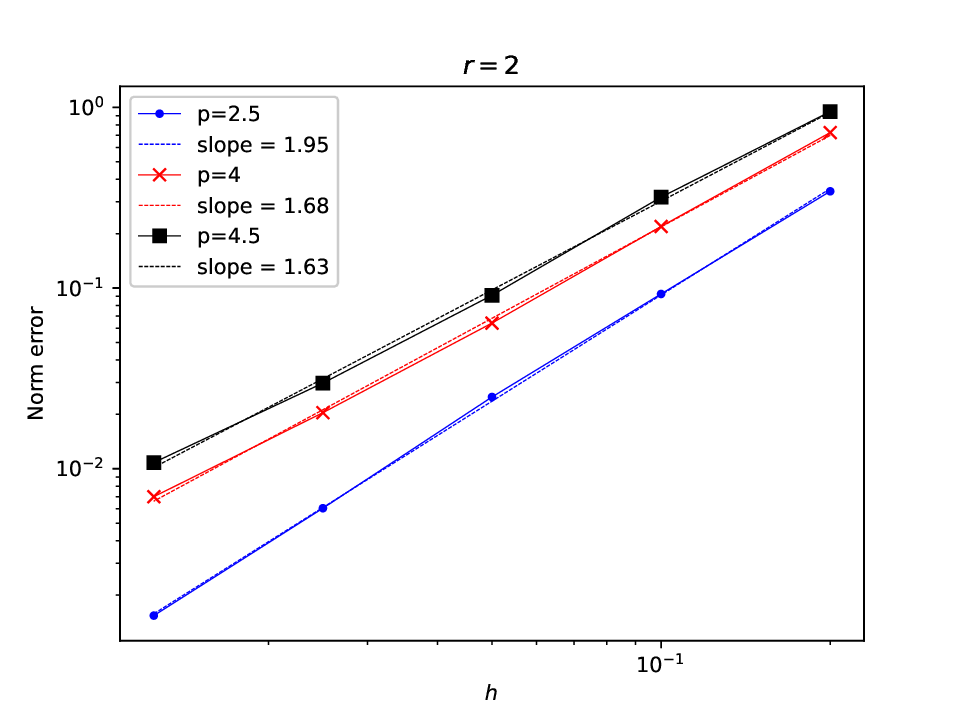} }}
     \hspace{0.00001cm}
    \subfloat[\centering]{{\includegraphics[width=7.6cm]{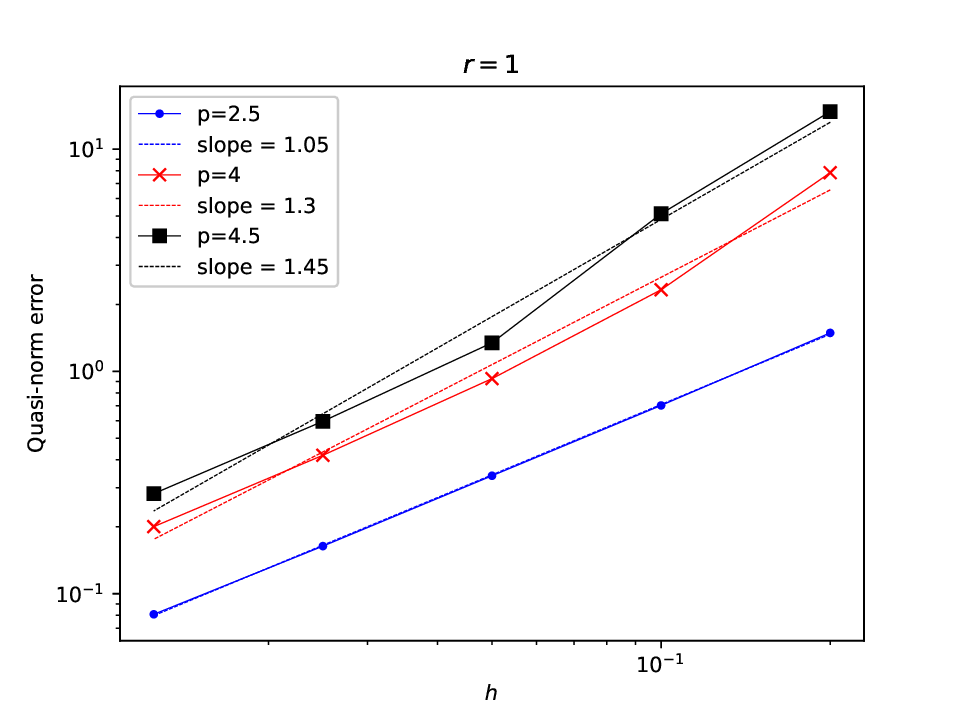} }}
    \hspace{0.00001cm}
    \subfloat[\centering]{{\includegraphics[width=7.6cm]{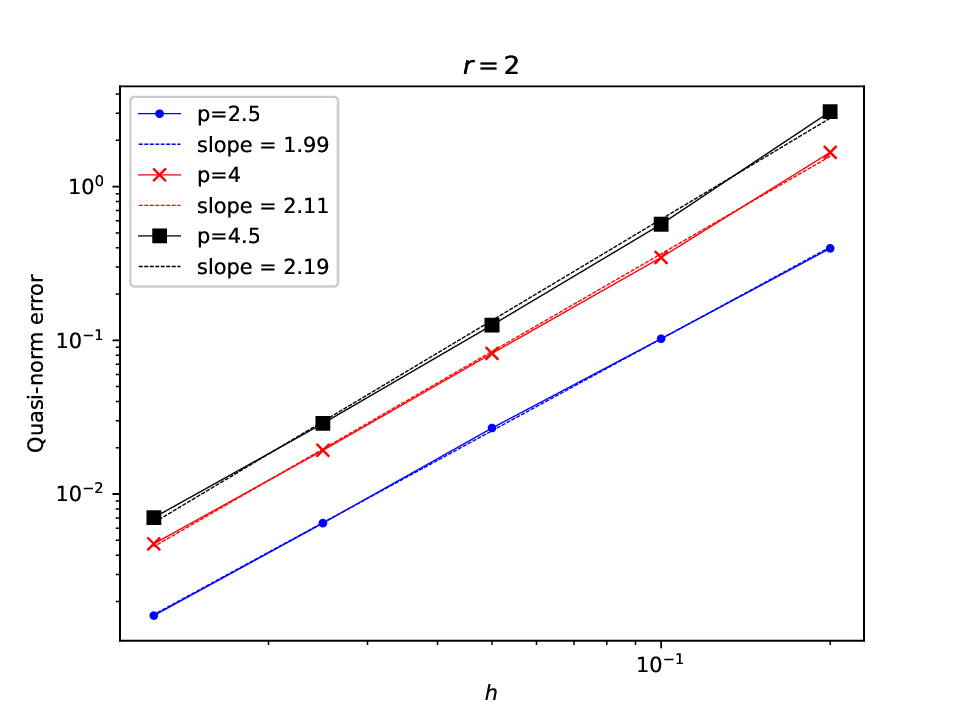} }}
    \caption{Example \ref{ex_tanh}, $h$-version: (a) Norm errors for $r=1$. (b) Norm errors for $r=2$. (c) Quasi-norm errors for $r=1$. (d) Quasi-norm errors for $r=2$.}
    \label{fig:hVersion_tanh}
\end{figure}

For both examples, Figures \ref{fig:hVersion_sym} and \ref{fig:hVersion_tanh} indicate that the quasi-norm errors exhibit the optimal convergence rates, as per the error analysis. Regarding the norm errors, we observe optimal convergence for $r=1$. For $r=2$, the suboptimality is visible in the actual computations. In fact, we even observe that the rates are decreasing as $p$ increases, which is in accordance to the error estimates, even though suboptimality appears to be slightly lower in practice compared to the respective \emph{a priori} error bounds.

\subsection{$p$-Version.}\label{DG_pversion}
We now consider the $p$-version of \eqref{IPDG} for both examples. We select a fixed mesh $\Ta$, with $h = \max_{K\in\Ta}h_K = 0.2$ and compute the corresponding discrete solutions with polynomial degrees $r=1,\dotso,5$. The norm and quasi-norm errors are presented in semi-logarithmic scale, together with their least-square approximations in Figures \ref{fig:pVersion_sym} and \ref{fig:pVersion_tanh}, for $p=2.5,4,4.5$.

\begin{figure}
    \centering
    \subfloat[\centering]{{\includegraphics[width=7.6cm]{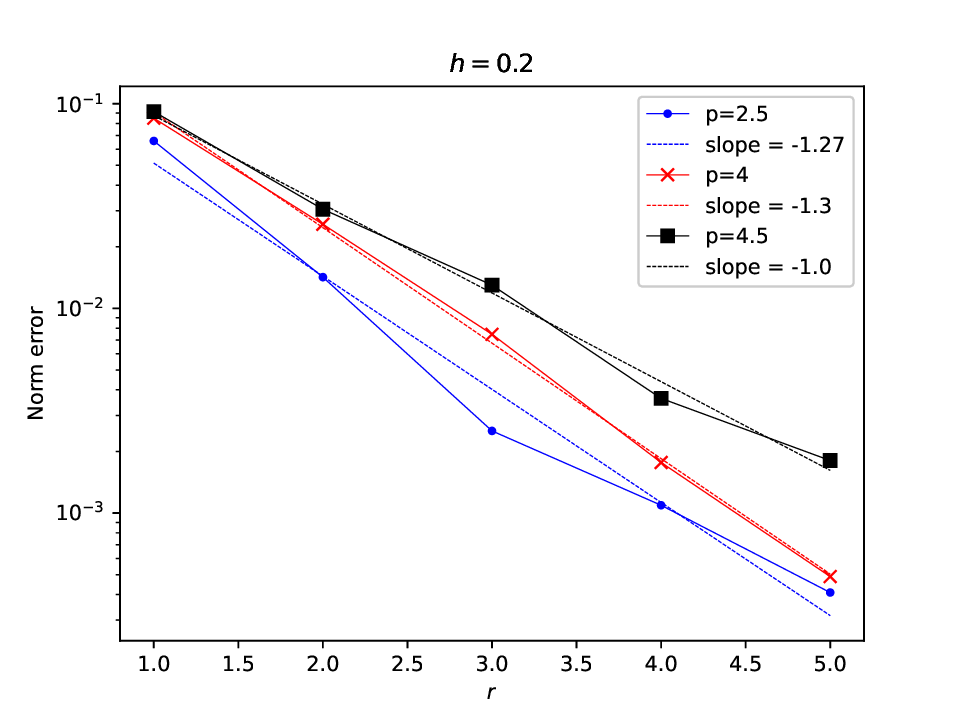} }}
    \hspace{0.00001cm}
    \subfloat[\centering]{{\includegraphics[width=7.6cm]{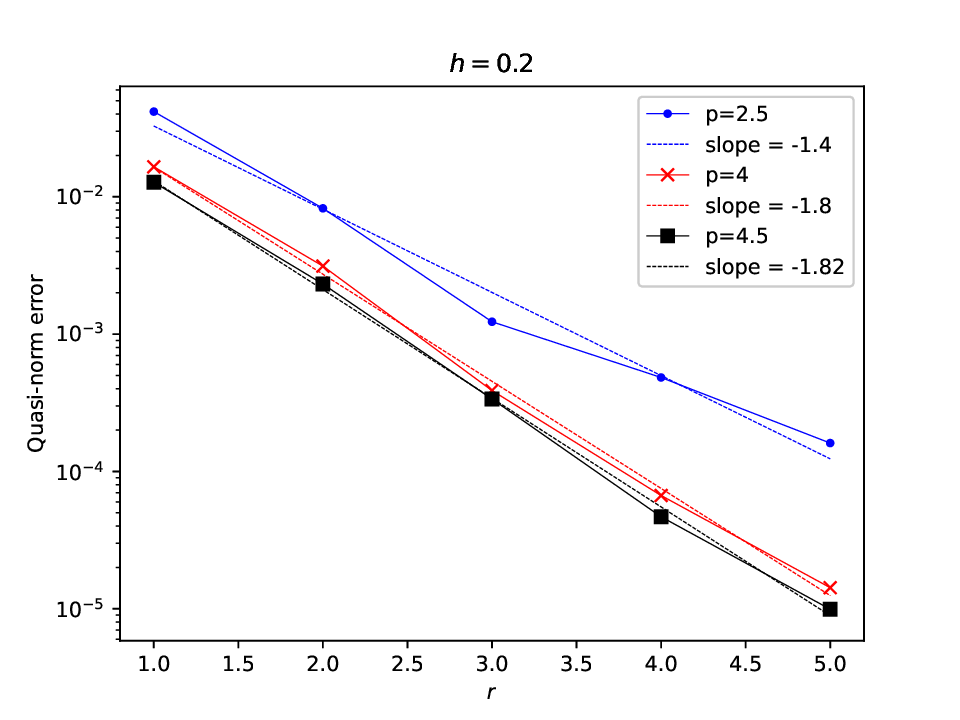} }}
    \caption{Example \ref{ex_sym}, $p$-version: (a) Norm errors. (b) Quasi-norm errors.}
    \label{fig:pVersion_sym}
\end{figure}

\begin{figure}
    \centering
    \subfloat[\centering]{{\includegraphics[width=7.6cm]{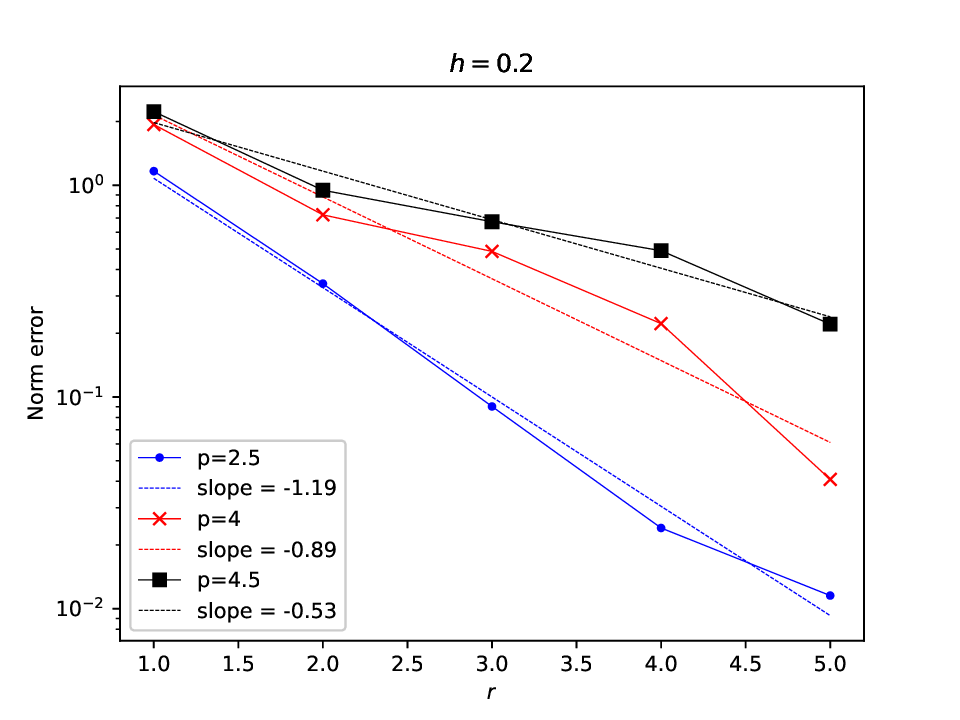} }}
    \hspace{0.00001cm}
    \subfloat[\centering]{{\includegraphics[width=7.6cm]{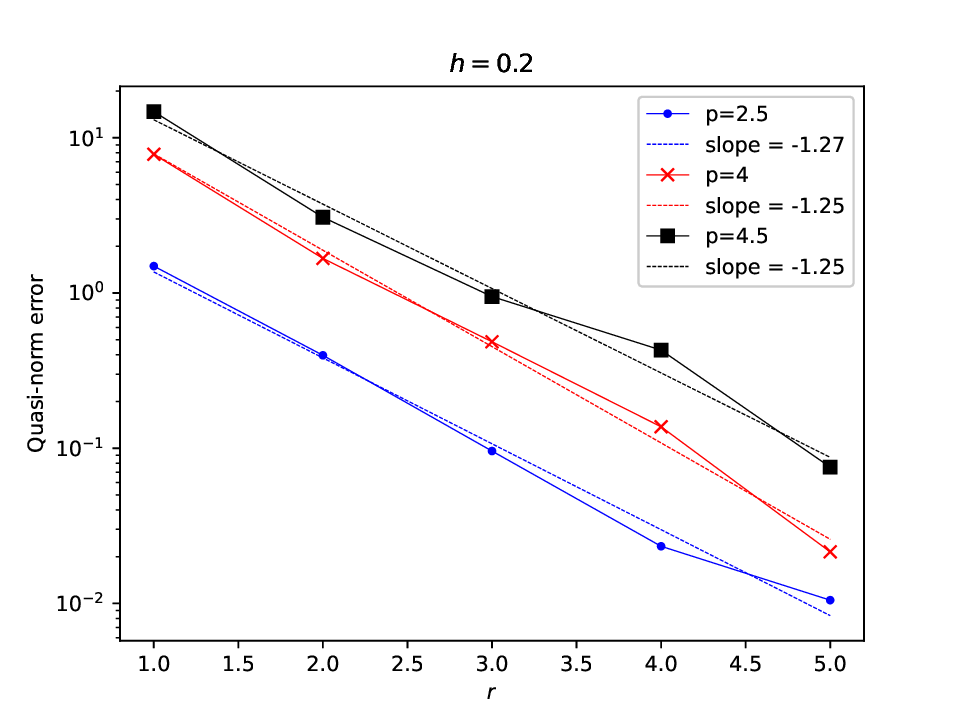} }}
    \caption{Example \ref{ex_tanh}, $p$-version: (a) Norm errors. (b) Quasi-norm errors.}
    \label{fig:pVersion_tanh}
\end{figure}

Figures \ref{fig:pVersion_sym} and \ref{fig:pVersion_tanh} show exponential decay of the errors with the local polynomial degrees. We observe that, in all cases considered, the norm rates are lower than the corresponding quasi-norm rates.

\section{Concluding remarks} 
A robust IPDG method discretizing the $p$-Laplacian equation, delivering provably superlinear convergence rate in appropriate scenarios has been developed and analyzed. Both simplicial meshes and meshes comprising general, possibly curved polytopic elements have been considered. The analysis hinges on new trace-inverse estimates, which may be of independent interest. The analysis on essentially arbitrarily-shaped element meshes highlights the need of extension of standard best approximation results to non-Hilbertian settings. Numerical experiments have also been presented, highlighting the relevance of the theoretical findings.

\section*{Acknowledgments} This work was supported by the Hellenic Foundation for Research and Innovation (H.F.R.I.) under the “First Call for H.F.R.I. Research Projects to support Faculty members and Researchers and the procurement of high-cost research equipment grant” (Proj. no. 3270). PP was also supported by H.F.R.I. under “4th Call for H.F.R.I. Scholarships for PhD Candidates” (Proj. no. 10688), and from the European Research Council (ERC) under the European Union’s Horizon 2020 research and innovation programme (Grant agreement No. 101125225). Also, EHG gratefully acknowledges the financial support of EPSRC (grant number EP/W005840/2).

\end{document}